%% file: KLHoweDual2.tex
\documentclass[12pt,oneside]{amsart}
\usepackage[top=1in, left=1in, right=1in, bottom=1in]{geometry}
\allowdisplaybreaks[1]
\usepackage{latexsym}
\usepackage{amsmath}
\usepackage{amssymb}
\usepackage{amscd}
\usepackage{array}
\usepackage{hhline}
\usepackage{color}
\usepackage{hyperref}
\hypersetup{
    colorlinks,
    citecolor=black,
    filecolor=black,
    linkcolor=black,
    urlcolor=black
}
\usepackage[vcentermath]{youngtab}
\usepackage{bbold}
\usepackage[all]{xy}
\usepackage{graphicx}
\newtheorem{definition}{Definition}[section]
\newtheorem{theorem}[definition]{Theorem}
\newtheorem{proposition}[definition]{Proposition}
\newtheorem{remark}[definition]{Remark}
\newtheorem{lemma}[definition]{Lemma}
\newtheorem{corollary}[definition]{Corollary}

\newtheorem{example}[definition]{Example}

\def\C{{\mathbb{C}}}

\def\N{{\mathbb{N}}}
\def\Z{{\mathbb{Z}}}

\def\deg{{\mathrm{deg}}}
\def\dim{{\mathrm{dim}}}

\def\ui{{\text{$\underline{i}$}}}

\def\oE{{\widehat{E}}}

\def\l{{\lambda}}

\def\idm{\text{\bfseries 1}}

\def\u{{\mathcal{U}}}
\def\e{{\mathcal{E}}}
\newcommand{\btime}[1]{\boxtimes\hspace{-1em}\raisebox{-0.8ex}{${}_{#1}$}\,\,}

\newcommand{\U}[1]{\dot{\mathbf U}_q(\mathfrak{sl}_{#1})}

\newcommand{\Uv}[1]{\dot{\mathbf U}_v(\mathfrak{sl}_{#1})}

\title{The $\mathfrak{sl}_N$-web algebras and dual canonical bases}
\author{Marco Mackaay}
\thanks{M.M. was supported by the FCT - Fundac\~{a}o para a 
Ci\^{e}ncia e a Tecnologia, through project number PTDC/MAT/101503/2008, 
New Geometry and Topology.}
\date{\today}
\begin{document}

\begin{abstract}
In this paper, which is a follow-up to~\cite{my}, 
I define and study $\mathfrak{sl}_N$-web algebras, for any $N\geq 2$. 
For $N=2$ these algebras are isomorphic to 
Khovanov's~\cite{kh} arc algebras and for $N=3$ 
they are Morita equivalent to the 
$\mathfrak{sl}_3$-web algebras which I defined and studied together with Pan 
and Tubbenhauer~\cite{mpt}. 

The main result of this paper is that 
the $\mathfrak{sl}_N$-web algebras are Morita equivalent to 
blocks of certain level-$N$ cyclotomic KLR algebras, for which I use 
the categorified quantum skew Howe duality defined in~\cite{my}. 

Using this Morita equivalence and Brundan and Kleshchev's~\cite{bk} work 
on cyclotomic KLR-algebras, I show that there exists an isomorphism between 
a certain space of $\mathfrak{sl}_N$-webs and the 
split Grothendieck group of the corresponding 
$\mathfrak{sl}_N$-web algebra, which maps 
the dual canonical basis elements to the Grothendieck classes of 
the indecomposable projective modules (with a certain normalization of their 
grading).      
\end{abstract}

\maketitle

\paragraph*{Acknowledgements}
I thank Bruce Fontaine, Joel Kamnitzer, Mikhail Khovanov, Bernard Leclerc, 
Daniel Tubbenhauer, Ben Webster and Yasuyoshi Yonezawa 
for helpful exchanges of emails and discussions on related 
(and unrelated) topics. 

I also thank Yonezawa for making the pictures in this paper, some of which 
already appeared in~\cite{my}.   
\tableofcontents
%%%%%%%%%%%%%%%%%%%%%%%%%%%%%%%%%%%%%%%%%%%%%%%%%%%%%%%%%%%%%%%%%%%%%%%%%%%%%%%%%%%%%%%%%%%%%%%%%%%%%%%%
%
%
% section 1 : Introduction
%
%
%%%%%%%%%%%%%%%%%%%%%%%%%%%%%%%%%%%%%%%%%%%%%%%%%%%%%%%%%%%%%%%%%%%%%%%%%%%%%%%%%%%%%%%%%%%%%%%%%%%%%%%%
\section{Introduction}\label{intro}

In~\cite{ckm} Cautis, Kamnitzer and Morrison defined $\mathfrak{sl}_N$-webs 
and the relations they satisfy, for arbitrary $N\in\N_{\geq 2}$. 
In~\cite{my} Yonezawa and I defined certain $\mathfrak{sl}_N$-web spaces 
$W_{\Lambda}$ for arbitrary $N\in\N_{\geq 2}$ and  
$\Lambda:=N\omega_{\ell}$, where $\omega_{\ell}$ is the $\ell$-th 
fundamental $\mathfrak{sl}_m$-weight with $m=N\ell$ for arbitrary 
$\ell\in\N$. By quantum skew Howe duality, also due to Cautis, Kamnitzer and 
Morrison~\cite{ckm}, we obtained a $\U{m}$-action on $W_{\Lambda}$ and 
showed that there exists an isomorphism of $\U{m}$-modules 
\begin{equation}
\label{eq:Howeintro}
V_{\Lambda}\to W_{\Lambda}.
\end{equation}
Here $V_{\Lambda}$ is the irreducible $\U{m}$-module of highest weight 
$\Lambda$, obtained as a quotient of the Verma module with the same highest 
weight. 

In the same paper, we also defined $\C$-linear additive 
$\mathfrak{sl}_N$-web categories 
$\mathcal{W}^{\circ}_{\Lambda}$, using colored $\mathfrak{sl}_N$-matrix 
factorizations. We showed that $\mathcal{W}_{\Lambda}^{\circ}$ is a 
strong $\mathfrak{sl}_m$ $2$-representation and 
that there exists an equivalence of 
strong $\mathfrak{sl}_m$ $2$-representations
\begin{equation}
\label{eq:introiso}
\mathcal{V}_{\Lambda}^p\to \dot{\mathcal{W}}_{\Lambda}^{\circ}.
\end{equation}
Here $ \dot{\mathcal{W}}_{\Lambda}^{\circ}$ 
denotes the Karoubi envelope of $\mathcal{W}_{\Lambda}^{\circ}$ and 
$\mathcal{V}_{\Lambda}^p:=R_{\Lambda}-\mathrm{pmod}_{\mathrm{gr}}$ 
is the category of finite-dimensional 
graded projective modules of the cyclotomic Khovanov-Lauda-Rouquier (KLR) 
algebra $R_{\Lambda}$. As we argued in~\cite{my}, this 
result can be seen as a categorification 
of an instance of the quantum skew Howe duality defined in~\cite{ckm}.  

Brundan and Kleshchev~\cite{bk} showed that there exists a bijective 
$\U{m}$-intertwiner 
\begin{equation}
\label{eq:introiso2}
\delta\colon V_{\Lambda}\to K_0^q(\mathcal{V}_{\Lambda}^p)
\end{equation}
where $K_0^q$ denotes the split Grothendieck group tensored with $\C(q)$ over 
$\Z[q,q^{-1}]$. 

Using \eqref{eq:introiso} and~\eqref{eq:introiso2}, Yonezawa and I  
defined a bijective $\U{m}$-intertwiner 
\begin{equation}
\label{eq:introiso3}
\delta^{\circ}\colon W_{\Lambda}\to K_0^q(\dot{\mathcal{W}}_{\Lambda}^{\circ}),
\end{equation}
such that the following square commutes:
\begin{equation}
\label{eq:Moritasquare1intro}
\begin{CD}
V_ {\Lambda}&@>{\delta}>>&K_0^q(\mathcal{V}_{\Lambda}^p)\\
@V{}VV&&@VV{}V\\
W_{\Lambda}&@>{\delta^{\circ}}>>& K_0^q(\dot{\mathcal{W}}_{\Lambda}^{\circ}) 
\end{CD}.
\end{equation}
\vskip0.5cm
The category $\mathcal{W}_{\Lambda}^{\circ}$ has infinitely many objects, but 
the space $W_{\Lambda}$ is finite-dimensional. Therefore, in this paper 
I choose a finite basis of $W_{\Lambda}$ and look at the 
full subcategory of $\mathcal{W}_{\Lambda}^{\circ}$ generated by the objects 
corresponding to these basis webs, denoted $\mathcal{W}_{\Lambda}^p$. I use 
some ``general arguments'' to show that   
\begin{equation}
\label{eq:introiso4}
\mathcal{W}_{\Lambda}^p\cong \dot{\mathcal{W}}_{\Lambda}^{\circ}
\end{equation}
and that there exists a $\Z$-graded finite-dimensional 
algebra $H_{\Lambda}$ such that 
\begin{equation}
\label{eq:introiso5}
\mathcal{W}_{\Lambda}^p\cong H_{\Lambda}-\mathrm{pmod}_{\mathrm{gr}}.
\end{equation}  
Although conceptually the path I sketched above is probably clearer, 
in the paper I will follow the inverse path. I will first 
define $H_{\Lambda}$ (Definitions~\ref{def:webalgebra} 
and~\ref{def:catwebspace}), then 
use~\eqref{eq:introiso5} as the definition of $\mathcal{W}_{\Lambda}^p$ 
(Definitions~\ref{def:webalgebramod} and~\ref{def:catwebspace}) and 
finally prove~\eqref{eq:introiso4} (Lemma~\ref{lem:embeddingwebalgintowebcat}). 

By ``general arguments'' I mean that they would prove the analogous 
results for other bases of $W_{\Lambda}$. In this paper I have chosen 
the basis of $W_{\Lambda}$ which corresponds to the 
Leclerc-Toffin~\cite{lt} basis in 
$V_{\Lambda}$ by quantum skew Howe duality. However, 
I could have used {\em Fontaine's basis}~\cite{fon} of $W_{\Lambda}$, for 
example. This would give rise to 
a Morita equivalent web algebra and all results in this paper for $H_{\Lambda}$ 
would have analogues for this web algebra.   

Leclerc and Toffin~\cite{lt} showed that their basis 
can be used to compute the canonical basis of $V_{\Lambda}$. Therefore, 
quantum skew Howe duality implies that 
the corresponding basis can be used 
to compute the dual canonical basis in $W_{\Lambda}$, as I will show in 
Section~\ref{sec:webs}.\footnote{This idea first arose in a discussion with 
Daniel Tubbenhauer on web bases.} In particular, Leclerc and Toffin's results 
indicate how one should normalize 
the generating $\mathfrak{sl}_N$-intertwiners corresponding to the 
generating $\mathfrak{sl}_N$-webs, such that the isomorphism 
in~\eqref{eq:Howeintro} maps the canonical basis in $V_{\Lambda}$ 
precisely to the dual canonical basis of $W_{\Lambda}$. As I will explain in 
Remark~\ref{rem:unfortunately}, in order to do this 
properly one has to switch to a new quantum parameter 
$v=-q^{-1}$, which first appeared in~\cite{fkk} and~\cite{kk}. 
Note that this normalization of the intertwiners differs from 
the normalization in the Cautis-Kamnitzer-Morrison paper~\cite{ckm}.  

Just as $R_{\Lambda}$, the algebra $H_{\Lambda}$ is a direct sum of blocks 
$$H_{\Lambda}=\bigoplus_{\vec{k}\in\Lambda(m,m)_N}H(\vec{k},N),$$
where 
\begin{equation}
\label{eq:levelmweights}
\Lambda(m,m)_N:=\{\vec{k}=(k_1,\ldots,k_m)\in \{0,\ldots,N\}^m 
\mid k_1+\cdots+k_m=m\}
\end{equation}
is the set of $\mathfrak{gl}_m$-weights of $V_{\Lambda}$. 
I show that $H(\vec{k},N)$ is a 
finite-dimensional graded symmetric Frobenius algebra, 
for each $\vec{k}\in \Lambda(m,m)_N$. In fact, $H(\vec{k},N)$ is only 
isomorphic to its dual $H(\vec{k},N)^{\vee}$ after a certain degree shift 
depending on $\vec{k}$. Therefore $H_{\Lambda}$ is not a graded Frobenius 
algebra, strictly speaking. 
 
By~\eqref{eq:introiso} and~\eqref{eq:introiso4}, we see that 
$$\mathcal{V}_{\Lambda}^p\cong \mathcal{W}_{\Lambda}^p$$
holds. The main result in this paper (Theorem~\ref{thm:categorification2}) 
is the extension of this equivalence 
to the categories of {\em all} finite-dimensional graded modules of 
$R_{\Lambda}$ and $H_{\Lambda}$:
\begin{equation}
\label{eq:Moritaintro}
\mathcal{V}_{\Lambda}\cong \mathcal{W}_{\Lambda}.
\end{equation}
In other words, $R_{\Lambda}$ and $H_{\Lambda}$ are Morita equivalent as graded 
algebras.  

The equivalence in~\eqref{eq:Moritaintro} allows me 
to define an anti-involution 
$$*\colon H_{\Lambda}\to H_{\Lambda}$$
which corresponds to Khovanov and Lauda's~\cite{kl1} anti-involution on 
$R_{\Lambda}$, and use it to define a duality 
$$\circledast\colon \mathcal{W}_{\Lambda}\to \mathcal{W}_{\Lambda}$$
which corresponds to Brundan and Kleshchev's~\cite{bk} 
duality on $\mathcal{V}_{\Lambda}$. Just as in their case, the duality 
on $\mathcal{W}_{\Lambda}$ can be restricted to $\mathcal{W}_{\Lambda}^p$, 
which induces a bar-involution on its Grothendieck group.  

Brundan and Kleshchev~\cite{bk} showed that $\delta$ in~\eqref{eq:introiso2} 
intertwines the bar-involutions and that it sends the canonical basis 
elements in $V_{\Lambda}$ to the Grothendieck classes of 
the indecomposable projective modules in $\mathcal{V}_{\Lambda}^p$ with 
a suitable normalization of their grading. 
I will show that the arguments above imply that there 
exists an isomorphism of $\Uv{m}$-modules 
$$\delta'\colon W_{\Lambda}\to K_0^v(\mathcal{W}^p_{\Lambda})$$
such that the square  
\begin{equation}
\label{eq:Moritasquare2intro}
\begin{CD}
V_ {\Lambda}&@>{\delta}>>&K_0^v(\mathcal{V}_{\Lambda}^p)\\
@A{}AA&&@AA{}A\\
W_{\Lambda}&@>{\delta'}>>& K_0^v(\mathcal{W}_{\Lambda}^{p}) 
\end{CD}
\end{equation}
commutes. Note that the vertical maps are inverted, compared 
to~\eqref{eq:Moritasquare1intro}. This is just for convenience: the proof 
of Theorem~\ref{thm:categorification2} becomes slightly shorter this way. 

By the above, it follows (Corollary~\ref{cor:dualcanonical}) that 
$\delta'$ maps the dual 
canonical basis elements in $W_{\Lambda}$ to the Grothendieck 
classes of the indecomposable projective modules in $\mathcal{W}_{\Lambda}^p$ 
which correspond to Brundan and Kleshchev's indecomposables in 
$\mathcal{V}_{\Lambda}^p$. 

Finally, the equivalence in~\eqref{eq:Moritaintro} implies 
(Corollary~\ref{cor:center}) that the 
center of $H(\vec{k},N)$ is isomorphic to the cohomology of a certain 
$N$-block Spaltenstein variety, for each $\vec{k}\in\Lambda(m,m)_N$. This 
follows from the analogous result for $R_{\Lambda}$, 
due to Brundan, Kleshchev and Ostrik~\cite{bru2,bk2,bo}. 

%%%%%%%%%%%%%%%%%%%%%%%%%%%%%%%%%%%%%%%%%%%%%%%%
\vskip0.5cm
For $N=2$ the web algebras were introduced by Khovanov~\cite{kh}, 
who called them {\em arc algebras}. For $N=3$ Pan and Tubbenhauer and I 
defined web algebras in~\cite{mpt} using Kuperberg's web basis. This basis 
is not equal to Leclerc and Toffin's, so the $\mathfrak{sl}_3$-web algebras 
in that paper are not isomorphic to the ones in this paper, but they are 
Morita equivalent. 

For $N=2$ Huerfano and Khovanov~\cite{hkh} used the arc algebras to category 
$V_{\Lambda}$, using a (partial) categorification of quantum skew Howe duality. 
The representation theory of the arc algebras and its connection 
with geometry have been studied in depth by a variety of 
people~\cite{bs,bs2,bs3,bs4,bs5,ck,kh,kh2,str,sw}. For $N=3$ the analogues 
of some of the results for arc algebras have been proved 
in~\cite{mpt} but much less is known. The results in this 
paper generalize some of the known results for $N=2,3$. 
\vskip0.5cm

For $N=2$ Stroppel and Webster~\cite{sw} proved that the 
arc algebras can be obtained from the intersection cohomology of the 
corresponding $2$-block Springer varieties. One can 
ask (as Kamnitzer asked M.~M. for $N=3$) if 
$H(\vec{k},N)$ can be obtained from the intersection cohomology of 
the $\mathfrak{sl}_N$ web varieties in~\cite{fkk}.  

Another open question is how to generalize the web algebras to 
{\em clasped webs}~\cite{caut,ku}. For such ``clasped web algebras'' one can 
also ask about the relation with the intersection cohomology of 
the web varieties in~\cite{fkk}, since the framework in that paper 
is quite general. 

%%%%%%%%%%%%%%%%%%%%%%%%%%%%%%%%%%%%%%%%%%%%%%%%%%%%%%%

%  Notations and conventions

%%%%%%%%%%%%%%%%%%%%%%%%%%%%%%%%%%%%%%%%%%%%%%%%%%%%%%%
 
\section{Notation and conventions}\label{sec:conventions}
In this section I fix some notations and explain some conventions. 
\vskip0.5cm
Let $\mathcal{C}^*$ be a $\Z$-graded $\C$-linear additive or 
Abelian category which admits translation (for a precise definition and 
some introductory remarks on this sort of categories, 
see~\cite{l} and the references therein for 
example). Then $\{t\}$ denotes a positive translation/shift of $t$ units. 
For any 
Laurent polynomial $f(q)=\sum a_iq^i\in \N[q,q^{-1}]$, 
define 
\begin{equation*}
X^{\oplus f(q)}:=\bigoplus_{i}\left(X\{i\}\right)^{\oplus a_i}.
\end{equation*}
Let $\mathcal{C}$ be the subcategory of $\mathcal{C}^*$ with the same objects 
but only degree-zero morphisms. In this paper $\mathcal{C}$ will always 
have finite-dimensional hom-spaces. 

Typical examples of such categories in this paper will be of the 
following sort. Let $\mathcal{C}$ be given by 
$$
A-\mathrm{mod}_{\mathrm{gr}}\quad\text{or}\quad A-\mathrm{pmod}_{\mathrm{gr}},
$$
which are the categories of finite-dimensional graded modules 
and finite-dimensional graded projective modules of a 
finite-dimensional $\Z$-graded complex algebra $A$. Both these categories 
admit translation, where the grading shifts are defined by 
$$X\{t\}_i:=X_{i-t}$$
for any object $X\in\mathcal{C}$ and any $i,t\in\Z$.

For any pair of objects $X,Y\in \mathcal{C}$, let $\mathrm{Hom}(X,Y)$ 
be the hom-space in $\mathcal{C}$. Then the graded hom-space of 
$\mathcal{C}^*$ is given by  
$$\mathrm{HOM}(X,Y):=\bigoplus_{t\in\Z}\mathrm{Hom}(X\{t\},Y).$$
 
For simplicity, assume 
that $\mathcal{C}^*$ has finite-dimensional hom-spaces too. 
Define the {\em quantum dimension} of $\mathrm{HOM}(X,Y)$ by 
$$\dim_q\,\mathrm{HOM}(X,Y):=\sum_{t\in\Z} q^t\dim\, \mathrm{Hom}(X\{t\},Y)\in 
\N[q,q^{-1}].$$
\vskip0.5cm 
Assume additionally that $\mathcal{C}$ is Krull-Schmidt. 
The split Grothendieck group 
$K_0(\mathcal{C})$ 
is by definition the Abelian group generated 
by the isomorphism classes of the objects in $\mathcal{C}$ modulo the relation  
$$[X\oplus Y]=[X]+[Y],$$
for any objects $X,Y\in\mathcal{C}$. It becomes a $\Z[q,q^{-1}]$-module, 
by defining 
$$q[X]=[X\{1\}],$$
for any object $X\in\mathcal{C}$. For any 
Laurent polynomial $f(q)=\sum a_iq^i\in \N[q,q^{-1}]$, we get 
$$f(q)[X]=[X^{\oplus f(q)}].$$

Assume that 
$S=\{X_1,\ldots,X_s\}$ is a set of indecomposable objects in 
$\mathcal{C}$ such that 
\begin{itemize}
\item any indecomposable object in $\mathcal{C}$ is 
isomorphic to $X_i\{t\}$ for a certain $i\in\{1,\ldots,s\}$ and $t\in\Z$;
\item for all $i\ne j\in \{1,\ldots,s\}$ and all $t\in\Z$ we have  
$$X_i\not\cong X_j\{t\}.$$
\end{itemize}
Then it is well-known that $K_0(\mathcal{C})$ is freely generated by $S$.  

In this paper I will mostly tensor $K_0(\mathcal{C})$ with $\C(q)$, 
so let 
$$K_0^q(\mathcal{C}):=K_0(\mathcal{C})\otimes_{\Z[q,q^{-1}]}\C(q).$$
\vskip0.5cm
A $q$-sesquilinear form on a finite-dimensional complex vector space $V$ 
is by definition a form 
$$\langle \cdot,\cdot\rangle\colon V\times V\to \C(q)$$
satisfying 
\begin{eqnarray*}
\langle f(q)v,v'\rangle &=& f(q^{-1})\langle v,v'\rangle\\
\langle v,f(q)v'\rangle &=& f(q)\langle v,v'\rangle,
\end{eqnarray*}
for any $f(q)\in \C(q)$ and $v,v'\in V$. 

There exists a well-known 
$q$-sesquilinear form on $K_0^q(\mathcal{C})$,  
which is called the {\em Euler form}. It is defined by 
$$\langle [X],[Y]\rangle := \dim_q\, \mathrm{HOM}(X,Y),$$
for any objects $X,Y\in \mathcal{C}$. Note that the Euler form takes 
values in $\N[q,q^{-1}]$.

%%%%%%%%%%%%%%%%%%%%%%%%%%%%%%%%%%%%%%%%%%%%%%%%%%%%%%%%%%%%%%%%%%%%%%%%%%%
%%                                                                
%%          fundamental representation theory
%%
%%%%%%%%%%%%%%%%%%%%%%%%%%%%%%%%%%%%%%%%%%%%%%%%%%%%%%%%%%%%%%%%%%%%%%%%%%%

\section{The special linear quantum algebra and its fundamental representations}
\label{sec:fund}
I briefly recall the special linear quantum algebra and the pivotal 
category of its fundamental representations. In the beginning I will 
use the ``neutral'' parameter $n\in\N_{\geq 2}$ for the quantum algebras 
and their representations. Later on I will always carefully choose between 
$n=N$ or $n=m$ in the different parts, which will be convenient for 
distinguishing the two sides of quantum skew Howe duality.  
\subsection{The special linear quantum algebra}
Let $n\geq 2$ be an arbitrary integer and let  
$$\alpha_i:=(0,\ldots,1,-1,\ldots,0)\in\Z^{n}$$ with $1$ on the $i$-th 
position, for $i=1,\ldots,n-1$. Denote the Euclidean inner product 
on $\Z^{n}$ by $(\cdot,\cdot)$. 
   
\begin{definition} For $n\in\N_{\geq 2}$ the {\em quantum special linear algebra} 
${\mathbf U}_q(\mathfrak{sl}_n)$ is 
the associative unital $\C(q)$-algebra generated by $K_i^{\pm 1}, E_{\pm i}$, 
for $i=1,\ldots, n-1$, subject to 
the relations
\begin{gather*}
K_iK_j=K_jK_i,\quad K_iK_i^{-1}=K_i^{-1}K_i=1,
\\
E_{+i}E_{-j} - E_{-j}E_{+i} = \delta_{i,j}\dfrac{K_i-K_i^{-1}}{q-q^{-1}},
\\
K_iE_{\pm j}=q^{\pm (\alpha_i,\alpha_j)}E_{\pm j}K_i,
\\
E_{\pm i}^2E_{\pm j}-(q+q^{-1})E_{\pm i}E_{\pm j}E_{\pm i}+E_{\pm j}E_{\pm i}^2=0,
\qquad\text{if}\quad |i-j|=1,\\
E_{\pm i}E_{\pm j}-E_{\pm j}E_{\pm i}=0,\qquad\text{else}
\end{gather*} 
\end{definition}

Recall that ${\mathbf U}_q(\mathfrak{sl}_n)$ is a Hopf algebra with 
coproduct given by 
\begin{equation*}
\label{eq:coproduct}
\Delta(E_{+i})=E_{+i}\otimes K_i+1\otimes E_{+i},\quad 
\Delta(E_{-i})=E_{-i}\otimes 1+K_i^{-1}\otimes E_{-i},\quad
\Delta(K_i^{\pm 1})=K_i^{\pm 1}\otimes K_i^{\pm 1}
\end{equation*}
and antipode by 
\begin{equation*}
\label{eq:antipode}
S(E_{+i})=-E_{+i}K_i^{-1},\quad 
S(E_{-i})=-K_iE_{-i},\quad
S(K_i)=K_i^{-1}.  
\end{equation*}
The counit is given by 
\begin{equation*}
\label{eq:counit}
\epsilon(E_{\pm i})=0,\quad \epsilon(K_i)=1.
\end{equation*}
The Hopf algebra structure is used to define 
${\mathbf U}_q(\mathfrak{sl}_n)$ actions on tensor products and duals of 
${\mathbf U}_q(\mathfrak{sl}_n)$-modules. 
\vskip0.5cm
Recall that the ${\mathbf U}_q(\mathfrak{sl}_n)$-weight lattice is 
isomorphic to $\Z^{n-1}$. For any $i=1,\ldots, n-1$, the element $K_i$ acts as 
$q^{\lambda_i}$ on the $\lambda$-weight space of any weight representation. 

Although I 
have not recalled the definition of 
${\mathbf U}_q(\mathfrak{gl}_n)$, it is sometimes convenient to use 
${\mathbf U}_q(\mathfrak{gl}_n)$-weights. Recall that 
the ${\mathbf U}_q(\mathfrak{gl}_n)$-weight lattice is isomorphic to $\Z^n$ 
and that any ${\mathbf U}_q(\mathfrak{gl}_n)$-weight 
$\vec{k}=(k_1,\ldots,k_n)\in\Z^n$ determines a unique 
${\mathbf U}_q(\mathfrak{sl}_n)$-weight 
$$\lambda=(k_1-k_2,\ldots,k_{n-1}-k_n)\in
\Z^{n-1}.$$ 
In this way, we get an isomorphism 
\begin{equation}
\label{eq:isolattices}
\Z^n/\langle(1^n)\rangle\cong \Z^{n-1}.
\end{equation}

\begin{remark}
Since the ${\bf U}_q(\mathfrak{sl}_n)$ and ${\bf U}_q(\mathfrak{gl}_n)$-weights 
and weight lattices are equal to those of the 
corresponding classical algebras, I will often refer to them as the 
$\mathfrak{sl}_n$ and $\mathfrak{gl}_n$-weights and weight lattices.  
\end{remark}

For weight representations, one can also use the idempotented 
version of ${\mathbf U}_q(\mathfrak{sl}_n)$, denoted $\U{n}$, due to 
Beilinson, Lusztig and MacPherson~\cite{B-L-M}. 
For $n=2$, define $i'=(2)$. For $n>2$, define    
$$i':=
\begin{cases}
(2,-1,0\ldots,0),&\text{for}\;i=1;\\
(0,\ldots,-1,2,-1,\ldots,0),&\text{for}\; 2\leq i\leq n-2;\\
(0,\ldots,0,-1,2),&\text{for}\;i=n-1.
\end{cases}
$$ 

Adjoin an idempotent $1_{\lambda}$ for each $\lambda\in\Z^{n-1}$ and add 
the relations
\begin{align*}
1_{\lambda}1_{\mu} &= \delta_{\lambda,\mu}1_{\lambda},   
\\
E_{\pm i}1_{\lambda} &= 1_{\lambda \pm i'}E_{i},
\\
K_i1_{\lambda} &= q^{\lambda_i}1_{\lambda}.
\end{align*}
\begin{definition} The {\em idempotented quantum special linear algebra} 
is defined by 
\[
\U{n}=\bigoplus_{\lambda,\mu\in\Z^{n-1}}1_{\lambda}{\mathbf U}_q(\mathfrak{sl}_n)1_{\mu}.
\]
\end{definition} 

\begin{remark}
It is sometimes convenient to consider $\U{n}$ as a category, whose objects 
are the weights $\lambda\in\Z^{n-1}$. The hom-space between 
$\lambda,\mu\in\Z^{n-1}$ is equal to 
$$1_{\lambda}{\mathbf U}_q(\mathfrak{sl}_n)1_{\mu}$$
and composition is given by multiplication. 
\end{remark}

\subsection{Fundamental representations}\label{sec:fundreps}
In this section I recall the fundamental 
${\mathbf U}_q(\mathfrak{sl}_n)$-representation theory, 
following~\cite{ckm,mor}. 
The basic ${\mathbf U}_q(\mathfrak{sl}_n)$-representation is denoted $\C^n_q$. 
It has a standard basis $\{x_1,\ldots,x_n\}$ on which the action is given by 
$$
E_{+i}(x_j)=
\begin{cases}
x_i,&\text{if}\;j=i+1;\\
0,&\text{else}.
\end{cases}
\quad
E_{-i}(x_j)=
\begin{cases}
x_{i+1},&\text{if}\;j=i;\\
0,&\text{else}.
\end{cases}
$$
$$
K_i(x_j)=
\begin{cases}
qx_i,&\text{if}\;j=i;\\
q^{-1}x_{i+1},&\text{if}\; j=i+1;\\
x_j,&\text{else}.
\end{cases}
$$

Using the basic representation, one can define all fundamental 
${\mathbf U}_q(\mathfrak{sl}_n)$-representations. Define the 
{\em quantum exterior algebra}
$$\Lambda^{\bullet}_q(\C^n_q):=T\C^n_q/\langle \{x_i\otimes x_i,\;x_i\otimes x_j+q
x_j\otimes x_i\mid 1\leq i<j\leq n\}\rangle.$$ 
We denote the equivalence class of $x\otimes y$ by $x \wedge_q y$. 
Note that 
$$\Lambda^{\bullet}_q(\C^n_q)=\bigoplus_{k=0}^n \Lambda^k(\C_q^n).$$
For each $0\leq k\leq n$, the homogeneous direct 
summand $\Lambda_q^k(\C_q^n)$ is an irreducible 
${\mathbf U}_q(\mathfrak{sl}_n)$-representation. For $k=0,n$ it is the 
trivial representation and for $1\leq k\leq n$ it is called the 
{\em $k$-th fundamental $\U{n}$-representation}. Recall 
that the dual of the $k$-th fundamental representation is 
isomorphic to the $(n-k)$-th fundamental representation.  
\vskip0.5cm
For each $k$ element subset $S\subset \{1,\ldots n\}$, order the elements in 
decreasing order and define 
$$x_S:=x_{s_1}\wedge_q x_{s_2}\wedge_q\cdots\wedge_q x_{s_k}.$$
We call the $x_S$ {\em elementary tensors}. The {\em standard basis} of 
$\Lambda^k(\C_q^n)$ is by definition 
$$\{x_S\mid S\subset \{1,\ldots n\},\; \vert S\vert=k\}.$$ 
Note that $x_S$ has $\mathfrak{gl}_n$-weight $\nu_S=(\nu_1,\ldots,\nu_n)$ with 
$$
\nu_j=
\begin{cases}
1,&\text{if}\;j\in S;\\
0,&\text{else}.
\end{cases}  
$$   
We call such a weight {\em $1$-bounded} of type $k$, 
following~\cite{ckm}. Below I 
will also use the notation 
$$x_{\nu}:=x_S\quad\text{for}\quad \nu=\nu_S.$$ 

In general I want to consider tensor products of elementary tensors. 
Let $\vec{k}=(k_1,\ldots,k_m)$ be an {\em $n$-bounded} 
$\mathfrak{gl}_m$-weight, i.e. satisfying 
$0\leq k_i\leq n$ for all $i=1,\ldots,m$. Define 
$$\Lambda_q^{\vec{k}}(\C_q^n):=\Lambda_q^{k_m}(\C_q^n)\otimes\cdots\otimes 
\Lambda_q^{k_1}(\C_q^n).$$
\begin{remark}
We write the tensor factors backwards in $\Lambda_q^{\vec{k}}(\C_q)$, which 
is more convenient for skew Howe duality and its categorification 
in this paper.  
\end{remark}
Let $\mathcal{S}(\vec{k},n)$ be the set of 
$m$-tuples 
$$\vec{\nu}=(\nu^1,\ldots,\nu^m)$$ 
such that $\nu^i$ is a $1$-bounded $\mathfrak{gl}_n$-weight of type $k_i$ 
for each $i=1,\ldots,m$. For short I say that $\vec{\nu}$ is $1$-bounded 
of type $\vec{k}$. The elements 
$$
x_{\vec{\nu}}:=x_{\nu^m}\otimes\cdots\otimes x_{\nu^1}
$$ 
form a basis of $\Lambda^{\vec{k}}_q(\C_q^n)$, which I also 
call the {\em standard basis}.

\vskip0.5cm
As already announced, I now will start using the parameter $N\in\N_{\geq 2}$ 
because intertwiners will only occur on one side of quantum skew Howe duality. 
I will also use introduce a second quantum parameter $v:=-q^{-1}$, as explained 
in the introduction.  
 
Let $\mathcal{R}ep(\mathrm{SL}_N)$ be the pivotal category whose objects are 
tensor products of fundamental ${\mathbf U}_q(\mathfrak{sl}_N)$-representations and their duals and 
whose morphisms are intertwiners. Cautis, Kamnitzer and Morrison defined a set of generating intertwiners 
in $\mathcal{R}ep(\mathrm{SL}_N)$. If 
$S,T$ are disjoint subsets of $\{1,\ldots,N\}$, let 
$$\ell(S,T):=\vert \{(i,j)\mid i\in S,\;j\in T\;\text{and}\; i<j\}\vert.$$
following~\cite{kk}. For any $0\leq a,b\leq N$ such that $a+b\leq N$, 
define
\begin{itemize}
\item the usual linear copairing 
$c_a\colon \mathbb{1}\to \Lambda^a_q(\C^N_q)\otimes 
(\Lambda^a_q(\C^N_q))^*$, which is an intertwiner;
\item the usual linear pairing $p_a\colon (\Lambda^a_q(\C^N_q))^*\otimes 
\Lambda^a_q(\C^N_q)\to \mathbb{1}$, which is an intertwiner;
\item the intertwiner $M_{a,b}\colon  
\Lambda_q^a(\C^N_q)\otimes\Lambda_q^b(\C^N_q)\to 
\Lambda_q^{a+b}(\C^N_q)$ by  
\begin{equation}
\label{eq:M}
M_{a,b}(x_S\otimes x_T):=x_S\wedge_q x_T=
\begin{cases}
v^{\ell(T,S)}x_{S\cup T},&\text{if}\; S\cap T=\emptyset;\\
0,&\text{else};
\end{cases}
\end{equation}
\item the intertwiner $M_{a,b}'\colon \Lambda_q^{a+b}(\C^N_q)\to 
\Lambda_q^a(\C^N_q)\otimes\Lambda_q^b(\C^N_q)$ by  
\begin{equation}
\label{eq:M'}
M_{a,b}'(x_S):=\sum_{T\subset S} v^{-\ell(T,S\backslash T)}x_T\otimes 
x_{S\backslash T};
\end{equation}
\item the bijective intertwiner $D_a\colon \Lambda_q^a(\C^N_q)\to 
(\Lambda_q^{N-a}(\C^N_q))^*$ by 
\begin{equation}
\label{eq:D}
D_a(x_S)(x_T):=
\begin{cases}
v^{\ell(T,S)},&\text{if}\; S\cap T=\emptyset;\\
0,&\text{else}.
\end{cases}
\end{equation}
\end{itemize}
Note that the usual linear copairing and pairing above define 
intertwiners, whereas the usual linear copairing and pairing   
$$\mathbb{1}\to (\Lambda^a_q(\C^N_q))^*\otimes \Lambda^a_q(\C^N_q)\quad\text{and}
\quad\Lambda^a_q(\C^N_q)\otimes (\Lambda^a_q(\C^N_q))^*\to \mathbb{1}$$ 
are not intertwiners. To get the intertwiners between those representations 
one has use the copairing and pairing above and 
$D_a$, $D_{N-a}$ and their inverses. For 
a good discussion about this point, see Section 3 in~\cite{mor}. 
 
As remarked in the introduction, I have normalized 
the maps differently from the ones in~\cite{ckm}. 
This is on purpose and I will explain the reason in 
Section~\ref{sec:canonical}. Note that the maps above 
still define intertwiners, because 
\begin{itemize}
\item this $M_{a,b}$ is equal to theirs multiplied by $(-q)^{-ab}$;
\item this $M'_{ab}$ is equal to theirs multiplied by $q^{ab}$;
\item this $D_a$ is equal to theirs multiplied by $(-q)^{-a(N-a)}$.
\end{itemize}

The following result is due to Morrison (Theorem 3.5.8 in~\cite{mor}).
\begin{theorem}[Morrison]
Any morphism in $\mathcal{R}ep(\mathrm{SL}_N)$ can be obtained by tensoring, 
composing and taking linear combinations of morphisms of the form  
$c_a$, $p_a$, $M_{a,b}$, $M_{a,b}'$ and $D_a$, for $a,b=0,\ldots,N$.  
\end{theorem}
\vskip0.5cm
For later use, I explain the {\em state-sum model} for the intertwiners. 
The coefficients in the formula for the intertwiners $M_{a,b}$ and $M'_{a,b}$ 
are powers of $v$ associated to triples of $1$-bounded 
$\mathfrak{gl}_N$-weights of type 
$a$,$b$ and $a+b$.  Similarly, the coefficient in the formula for 
$D_a$ is a power of $v$ associated to a pair of $1$-bounded 
$\mathfrak{gl}_N$-weights of type 
$a$ and $N-a$. By composing and tensoring these elementary 
intertwiners, we see that any intertwiner maps 
basis tensors to linear combinations of basis tensors 
and that the coefficient of each summand is a sum of 
powers of $v$. Each power of $v$ is associated to a particular assignment 
of $1$-bounded $\mathfrak{gl}_N$-weights to the edges of the web which 
represents the intertwiner. 
These assignments have to satisfy compatibility conditions, because 
intertwiners preserve the total weight. 
\begin{definition}\label{def:states}
Let $w$ be any monomial web and let $E(w)$ be the set of its edges. 
A {\em state}  
of $w$ is a map $\sigma\colon E(w)\to \Z^N/\langle (1^N)\rangle$ 
such that 
\begin{enumerate}
\item $\sigma(e)\in W(1^{k_e})\bmod (1^N)$, for any $e\in E(w)$ of color $k_e$;
\item $\sigma(e_1)+\sigma(e_2)\equiv\sigma(e_3)\bmod (1^N)$, 
whenever $e_1,e_2,e_3$ meet 
at a triple vertex and 
$e_1$ and $e_2$ are both oriented in the same direction w.r.t. that vertex;
\item $\sigma(e_1)+\sigma(e_2)\equiv 0\bmod (1^N)$, whenever $e_1$ and $e_2$ meet at a tag.
\end{enumerate}
The set of states of $w$ is denoted $\mathcal{S}(w)$. 
\end{definition} 
\noindent In this way, we see that any intertwiner can be written as a 
{\em state-sum}. 

\subsection{Tensors and tableaux}
\label{sec:TT}
For later use, I recall here two relations between certain basis tensors 
and column-strict tableaux. 

For simplicity, I assume that $m=N\ell$, which is the case of 
interest in this paper. Let $\mathrm{Col}^{(N^{\ell})}$ be the set 
of all {\em column-strict tableaux} of shape $(N^{\ell})$ whose fillings 
are natural numbers between $1$ and $m$. Recall that a 
tableau is called column-strict if its entries are strictly 
increasing along every column from top to bottom.  

Let $\mathrm{Std}^{(N^{\ell})}\subset 
\mathrm{Col}^{(N^{\ell})}$ be the subset of {\em semi-standard} tableaux. These 
satisfy the additional condition that the fillings are weakly increasing 
along each row from left to right. The semi-standard tableaux parametrize the 
canonical and dual canonical bases and also Leclerc and Toffin's intermediate 
basis and its Howe dual, as I will show later.   

For any $N$-bounded $\mathfrak{gl}_m$-weight 
$\vec{k}\in \Lambda(m,m)_N$ (defined in~\eqref{eq:levelmweights}), let 
$\mathrm{Col}^{(N^{\ell})}_{\vec{k}}\subset\mathrm{Col}^{(N^{\ell})}$ 
be the subset of tableaux of type $\vec{k}$. Recall that a tableau 
is of type $\vec{k}$ if its filling contains $k_i$ times the entry 
$i$, for $i=1,\ldots,m$. 
\vskip0.5cm
For the first relation between tableaux and tensors, 
let $\mathcal{S}(\vec{k},N)_0\subset \mathcal{S}(\vec{k},N)$ be 
the subset of elements $\nu$ such that 
$$\sum_{i=1}^m \nu^i=(\ell^N).$$
There is a well-known bijection 
between $\mathcal{S}(\vec{k},N)_0$ and 
$\mathrm{Col}^{(N^{\ell})}_{\vec{k}}$. If $T$ is a tableau, let $T(j)$ denote the 
$j$-th column of $T$. 
The bijection $\vec{\nu}_T\leftrightarrow T$ is determined by the rule  
$$
\nu^i_j=1 \Leftrightarrow i\in T(j)
$$ 
for any $1\leq i\leq m$ and $1\leq j\leq N$. Write 
$$x^T:=x_{\vec{\nu}_T}\in \Lambda^{\vec{k}}_q(\C_q^N)$$
and note that $x^T$ has $\mathfrak{sl}_N$-weight zero.  

\begin{example}\label{ex:young} 
Let $N=3$, $\ell=4$, $m=12$, and $\vec{k}=(2,2,1,2,1,3,1,0,0,0,0,0)$. 
The tableau
$$
\young(112,234,456,667)
$$
corresponds to 
$$
\vec{\nu}=((110),(101),(010),(101),(010),(111),(001)).
$$
\end{example} 
\vskip0.5cm

In the second relation between tableaux and tensors, 
the basis tensors of $\Lambda_q^{\ell}(\C_q^{m})^{\otimes N}$ of 
$\mathrm{gl}_m$-weight $\vec{k}$ are parametrized 
by the tableaux in $\mathrm{Col}^{(N^{\ell})}_{\vec{k}}$. In this case, 
the bijection is 
given by 
$$\mu^i_j=1\leftrightarrow j\in T(i)$$
for any $1\leq i\leq N$ and $1\leq j\leq m$. For any $T\in 
\mathrm{Col}^{(N^{\ell})}_{\vec{k}}$, write  
$$x_T:=x_{\vec{\mu}_T}\in \Lambda_q^{\ell}(\C_q^{m})^{\otimes N}.$$  
\begin{example}
For the tableau in Example~\ref{ex:young} we get 
$$\vec{\mu}=((110101000000),(101011000000),(010101100000)).$$
\end{example}

\begin{remark} Thus to the same tableaux 
$T\in \mathrm{Col}^{(N^{\ell})}_{\vec{k}}$ one 
can associate two different basis tensors in two different tensor 
spaces and the notation in this paper distinguishes the two:
$$x^T\in \Lambda_q^{\vec{k}}(\C_q^N)\quad\text{and}\quad x_T\in 
\Lambda_q^{\ell}(\C_q^{m})^{\otimes N}.$$ 

Both types of basis tensor will be of interest in this paper, because quantum 
skew Howe duality relates them. 
\end{remark}

For later use, I recall that there is a 
total ordering on $\mathrm{Col}^{(N^{\ell})}$. 
Consider columns as increasing sequences. Given two columns 
$c$ and $d$, define 
$$c=(c_1 < \cdots < c_{\ell})\succ d=(d_1 < \cdots < d_{\ell})$$
if the first $c_i$ different from $d_i$ is less than $d_i$. Order the 
columns from left to right. The lexicographical ordering w.r.t. to these 
two orderings gives a total ordering on $\mathrm{Col}^{(N^{\ell})}$. 

\begin{example}
\label{ex:highesttableau}
The greatest tableau in $\mathrm{Col}^{(N^{\ell})}$ w.r.t. this ordering 
is the one whose columns are all equal to $(1,2,\ldots,\ell)$. Let us denote 
it by $T_{\Lambda}$, because it corresponds to the highest weight 
vector of both $V_{\Lambda}$ and $W_{\Lambda}$, as we will see.  
\end{example}    

\section{$\mathrm{SL}_N$ webs}\label{sec:webs}
The morphisms in $\mathcal{R}ep(\mathrm{SL}_N)$ can be represented graphically 
by {\em $\mathfrak{sl}_N$-webs}. These are certain oriented 
trivalent graphs, whose edges 
are colored by integers belonging to $\{0,\ldots,N\}$. 
Webs can be seen as morphisms in a pivotal category, which in the literature is 
called a {\em spider} or {\em spider category}, denoted 
$\mathcal{S}p(\mathrm{SL}_{N})$.  

\subsection{The $\mathrm{SL}_N$ spider}
\label{sec:spider}

Recently, Cautis, Kamnitzer and Morrison~\cite{ckm} gave a presentation 
of $\mathcal{S}p(\mathrm{SL}_{N})$ in terms of generating webs and relations. 

\begin{definition}[Cautis-Kamnitzer-Morrison]
\label{def:webrelations}
The objects of $\mathcal{S}p(\mathrm{SL}_N)$ are finite sequences 
$\vec{k}$ of elements in $\{0^{\pm},\ldots,(N)^{\pm}\}$. 

The hom-space $\mathrm{Hom}(\vec{k},\vec{l})$ is the 
$\mathbb{C}(q)$-vector space freely generated by all diagrams, with 
lower and top boundary labeled from right to left by the entries of 
$\vec{k}$ and $\vec{l}$ respectively, which 
can be obtained by glueing and juxtaposing 
labeled cups and caps and the following elementary webs, together with 
the ones obtained by mirror reflections and arrow reversals: 

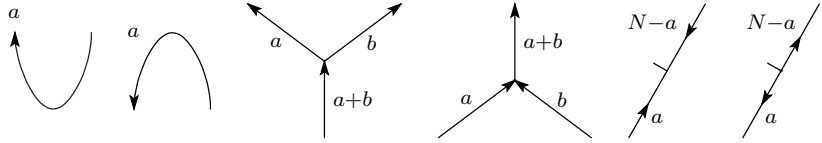
\begin{eqnarray*}
\txt{\input{figure/diag26}}
\hspace{0.5cm}
\txt{\input{figure/diag27}}
\hspace{0.5cm}
\txt{\input{figure/diag1}}
\hspace{0.5cm}
\txt{\input{figure/diag2}}
\hspace{0.5cm}
\txt{\input{figure/diag4}}
\hspace{0.5cm}
\txt{\input{figure/diag3}}
\end{eqnarray*}
with all labels between $0$ and $N$,  
modded out by planar isotopies (e.g. the zig-zag relations for cups and caps) 
and the following relations:
\begin{eqnarray}
\label{eq:tagswitch}\txt{\input{figure/diag4}}&\hspace{0.7cm}=&(-1)^{a(N-a)}\txt{\input{figure/diag5}}
\\[1em]
\label{eq:paralleldigon}\txt{\input{figure/diag12}}&\hspace{0.7cm}=&\begin{bmatrix}a+b\\a\end{bmatrix}_v\txt{\input{figure/diag10}}
\\[1.5em]
\label{eq:oppositedigon}\txt{\input{figure/diag13}}&\hspace{0.7cm}=&\begin{bmatrix}N-a\\b\end{bmatrix}_v\txt{\input{figure/diag10a}}
\\[1em]
\label{eq:associativity}\txt{\input{figure/diag14}}&\hspace{0.7cm}=&\txt{\input{figure/diag15}}
\\[1em]
\label{eq:parallelsquare}\hspace{1cm}\txt{\input{figure/diag17}}\hspace{0.5cm}&\hspace{0.7cm}=&\begin{bmatrix}s+t\\t\end{bmatrix}_v\hspace{1cm}\txt{\input{figure/diag16}}
\\[1em]
\label{eq:oppositesquare}\hspace{1cm}\txt{\input{figure/diag18}}\hspace{0.5cm}&\hspace{0.7cm}=&\sum_r \begin{bmatrix}a-b+t-s\\r\end{bmatrix}_v\hspace{1cm}\txt{\input{figure/diag19}}
\end{eqnarray}
together with the analogous relations obtained by mirror reflections and 
arrow reversals. 
\end{definition}
\vskip0.5cm
Let $\Gamma_N\colon \mathcal{S}p(\mathrm{SL}_N)\to \mathcal{R}ep(\mathrm{SL}_N)$ be the pivotal functor defined on objects by 
\begin{eqnarray*}
\vec{k}=(k_1^{\epsilon_1},\ldots,k_m^{\epsilon_m})&\mapsto& 
\Lambda^{\vec{k}}_q(\C_q^N)=(\Lambda^{k_m}_q(\C^N_q))^{\epsilon_m}\otimes 
\cdots \otimes (\Lambda^{k_1}_q(\C^N_q))^{\epsilon_m},\\
\end{eqnarray*}  
where $V^1:=V$ and $V^{-1}:=V^*$ by definition. 
On morphisms, define $\Gamma_N$ by 
\begin{eqnarray*}
\txt{\input{figure/diag26}}&\mapsto&c_a\\[1em]
\txt{\input{figure/diag27}}&\mapsto&p_a\\[1em]
\txt{\input{figure/diag1}}&\mapsto&M_{a,b}'\\[1em]
\txt{\input{figure/diag2}}&\mapsto&M_{a,b}\\[1em]
\txt{\input{figure/diag4}}&\mapsto&D_a\\[1em]
\txt{\input{figure/diag5}}&\mapsto&(-1)^{a(N-a)}D_a\\[1em]
\end{eqnarray*}
The following result can be found in~\cite{ckm} (Theorems 3.2.1 and 3.3.1):
\begin{theorem}[Cautis-Kamnitzer-Morrison]\label{thm:ckm}
The functor $\Gamma_N$ is well defined and gives an equivalence of 
pivotal categories. 
\end{theorem}
\noindent Since I have changed the normalization of the intertwiners, 
the web relations in Definition~\ref{def:webrelations} gain certain 
minus signs if the quantum parameter is taken to be $q$, 
as the reader can easily check. When we pass from 
$q$ to $v$, these minus signs get absorbed by the $v$-binomial coefficients 
because  
$$\begin{bmatrix} a+b\\ a\end{bmatrix}_v=(-1)^{ab}\begin{bmatrix} a+b\\ 
a\end{bmatrix}_q.$$

%%%%%%%%%%%%%%%%%%%%%%%%%%%%%%%%%%%%%%%%%%%%%%%%%%%%%%%%%%%%%%%%%%%%%%%%%%

\subsection{Quantum skew Howe duality}
\label{sec:Howe}
Let us briefly recall the instance of quantum skew Howe duality 
which was categorified in~\cite{my}. For more details see~\cite{ckm} and the 
references therein. 
\vskip0.5cm
As explained in~\cite{ckm}, there are $\C(q)$-linear isomorphisms
$$
\Lambda_q^{\bullet}(\C_q^m)^{\otimes N}\cong
\Lambda_q^{\bullet}(\C_q^{m}\otimes \C_q^{N})\cong 
\Lambda_q^{\bullet}(\C_q^N)^{\otimes m}.
$$
The commuting $\U{m}$ and $\U{N}$-actions on 
$\Lambda_q^{\bullet}(\C_q^{m}\otimes \C_q^{N})$ induce 
a $\U{N}$-action on $\Lambda_q^{\bullet}(\C_q^m)^{\otimes N}$, which 
commutes with the standard $\U{m}$-action, and 
a $\U{m}$-action on $\Lambda_q^{\bullet}(\C_q^N)^{\otimes m}$, which 
commutes with the standard $\U{N}$-action. The explicit 
definition of the ``non-standard'' actions is not 
immediately obvious and was worked out 
in~\cite{ckm}. 

\begin{remark}
\label{rem:unfortunately}
Unfortunately, the normalization of the 
intertwiners associated to the generating webs 
in~\cite{ckm} is not convenient for the purposes in this paper. 
For categorification 
we want all the coefficients in the web relations 
(for the webs without tags) and in 
the expressions for the generating intertwiners to be 
quantum natural numbers. Additionally, the top coefficient 
should always be equal to one for certain intertwiners if want we want a 
precise match between the dual canonical basis elements in $W_{\Lambda}$ 
and the indecomposable objects in $\mathcal{W}_{\Lambda}^p$. 
This is impossible if one insists on 
having the same quantum parameter $q$ for both $\U{m}$ and $\U{N}$. 
Therefore I define commuting 
$\Uv{m}$ and $\U{N}$-actions on 
$\Lambda_q^{\bullet}(\C_q^N)^{\otimes m}$ in this paper. As already remarked in 
the introduction, this is consistent with the use of 
$v=-q^{-1}$ in~\cite{fkk,kk}. 

The confusing part is that $\Lambda_q^{\bullet}(\C_q^N)^{\otimes m}$ is 
seen as a $\C(v)$-linear space in this approach. 
Perhaps the right way to think about it, is that both 
$\Lambda_v^{\bullet}(\C_v^m)^{\otimes N}$ and $\Lambda_q^{\bullet}(\C_q^N)^{\otimes m}$ 
as $\C(v)$-linear vector spaces are isomorphic to a direct summand of the 
$\C(v)$-linear Fock space generated by 
column-strict tableaux in~\cite{ug}. 
For each such tableau $T$, both the tensors $x_T$ and $x^T$ then become 
identified with the same ``abstract'' basis vector of the Fock space and 
they are no longer representatives of equivalence classes of 
tensors in exterior powers with different quantum parameters. 
On this Fock space there are two commuting actions of $\Uv{m}$ and 
$\dot{\mathbf{U}}_{-v^{-1}}(\mathfrak{sl}_N)$.\footnote{I thank Ben Webster for 
explaining this to me.} In fact, this 
Fock space is also used in~\cite{bk} for the $\Uv{m}$-side of the story.  
\end{remark}
    
\vskip0.5cm
Let $m,d,N$ be arbitrary non-negative integers. Define 
$$\Lambda(m,d):=\{\vec{k}\in \N^m\mid k_1+\cdots+k_m=d\}$$
and 
$$\Lambda(m,d)_N:=\{\vec{k}\in \Lambda(m,d)\mid 0\leq k_i\leq N\quad 
\text{for all}\; i=1,\ldots,m\}.$$

Given a ${\mathbf U}_v(\mathfrak{sl}_m)$-weight 
$\lambda=(\lambda_1,\ldots,\lambda_{m-1})$, 
the isomorphism in~\eqref{eq:isolattices} shows 
that in general there is not a unique way to lift 
$\lambda$ to a ${\mathbf U}_q(\mathfrak{gl}_m)$-weight. 
However, for a fixed value of $d\in\N$ there exists a partially 
defined map  
$$\phi_{m,d,N}\colon \Z^{m-1}\to \Lambda(m,d)_N$$ 
determined by 
$$\phi_{m,d,N}(\lambda)=\vec{k}$$
such that 
\begin{eqnarray}
k_i&\in&\{0,\ldots,N\}\quad\text{for all}\; i=1,\ldots,m\\
\label{eq:sl-gl-wts1}
k_i-k_{i+1}&=&\lambda_i\quad\text{for all}\; i=1,\ldots,m,\\
\label{eq:sl-gl-wts2}
\qquad \sum_{i=1}^{m}k_i&=&d.
\end{eqnarray}
Note that $\vec{k}$ might not exist. But if it does, it is necessarily unique.  
If it does not exist, put $\phi_{m,d,N}(\lambda)=*$ for convenience. For 
more information on this map in relation with quantum Schur algebras 
see~\cite{msv}.
\vskip0.5cm
The following proposition follows from the results in Sections 4 and 5 
in~\cite{ckm}.
\begin{proposition}[Cautis-Kamnitzer-Morrison]
\label{prop:CKM}
The $\C(v)$-linear functor 
$$\gamma_{m,d,N}\colon \Uv{m}\to \mathcal{S}p(\mathrm{SL}_N)$$ determined by 
\begin{eqnarray*}
1_{\lambda}&\mapsto& 
\begin{cases}
\hspace{1cm}\txt{\input{figure/id-lambda}},&\hspace{1cm}\text{if}\;
\phi_{m,d,N}(\lambda)=\vec{k};\\
0,&\hspace{1cm}\text{if}\;\phi_{m,d,N}(\lambda)=*.
\end{cases}
\\[0.5em]
E_{+i}1_{\lambda}&\mapsto& \txt{\input{figure/e-mu-i-1}}\\
E_{-i}1_{\lambda}&\mapsto& \txt{\input{figure/f-mu-i-1}}\\
\end{eqnarray*} 
is well-defined and full. 
\end{proposition}
By~\eqref{eq:paralleldigon} and~\eqref{eq:associativity}, 
it is easy to determine the images of the divided powers 
$$E_{+i}^{(a)}:=E_{+i}^a/[a]!\quad\text{and}\quad E_{-i}^{(a)}=E_{-i}^a/[a]!$$
\begin{eqnarray*}
E_{+i}^{(a)}1_{\lambda}&\mapsto& \txt{\input{figure/e-mu-i-a}}\\
E_{-i}^{(a)}1_{\lambda}&\mapsto& \txt{\input{figure/f-mu-i-a}}\\
\end{eqnarray*} 
\vskip0.5cm
The images of products of divided powers in $\Uv{m}$ are special web diagrams, 
called {\em ladders} (see Section 5 in~\cite{ckm}). 
\begin{definition}[Cautis-Kamnitzer-Morrison]
\label{def:ladders}
An {\em $N$-ladder with $m$ uprights} is a rectangular $\mathfrak{sl}_N$-web 
diagram without tags, such that  
\begin{itemize}
\item its vertical edges are all oriented upwards and lie on $m$ 
parallel vertical lines running from bottom to top; 
\item it contains a certain number of horizontal oriented {\em rungs} 
connecting adjacent uprights.
\end{itemize} 
\end{definition}
\noindent Since ladders do not have tags, at each trivalent vertex 
the sum of the labels of the incoming edges has to be equal to the sum of 
the labels of the outgoing edges. 
\vskip0.5cm
By composing the functor $\Uv{m}\to \mathcal{S}p(\mathrm{SL}_N)$ 
in Proposition~\ref{prop:CKM} with  
$\Gamma_N\colon \mathcal{S}p(\mathrm{SL}_N)\to 
\mathcal{R}ep(\mathrm{SL}_N)$, we get a well-defined action of $\Uv{m}$ 
on $\Lambda_q^{\bullet}(\C_q^N)^{\otimes m}$. For any $X\in \Uv{m}$ and 
$x_{\vec{\nu}}\in 
\Lambda_q^{\bullet}(\C_q^N)^{\otimes m}$, first map $X$ to a 
ladder in $\mathcal{S}p(\mathrm{SL}_N)$ and then apply the corresponding 
intertwiner in $\mathcal{R}ep(\mathrm{SL}_N)$ to $x_{\vec{\nu}}$ if possible. If 
the domain of the intertwiner does not match the tensor type of $x_{\vec{\nu}}$, 
we declare the action to be zero. Since the 
$\Uv{m}$-action is defined by intertwiners, it commutes with the standard 
$\U{N}$-action on $\Lambda_q^{\bullet}(\C_q^N)^{\otimes m}$.

\vskip0.5cm
{\em For the rest of this paper, let $N\geq 2 $ and $m,\ell\geq 0$ be 
arbitrary but fixed integers such that $m=N\ell$. The extra parameter $d$ 
in Proposition~\ref{prop:CKM} is always taken to be equal to $m$}. 
\vskip0.5cm
The standard action of $\Uv{m}$ on 
$\Lambda^{\ell}_v(\C_v^m)^{\otimes N}$ is easy to describe in terms of tableaux. 
Choose $T\in \mathrm{Col}^{(N^{\ell})}$ and $i\in \{1,\ldots,m-1\}$. Then 
$$E_{-i}x_T=\sum_{T'} v^{-a_{T,T'}}x_{T'}$$
where the sum is over all tableaux $T'\in \mathrm{Col}^{(N^{\ell})}$ obtained from 
$T$ by changing one $i$ into $i+1$. Suppose that $T'$ is such 
a tableau and that it was obtained by changing the entry $(k,l)$ of $T$, then 
$$a_{T,T'}=\vert\{(k',l')\mid l'>l,\; T_{(k',l')}=i\}\vert-\vert
\{(k',l')\mid l'>l,\; T_{(k',l')}=i+1\}\vert.$$

The action of $E_i$ on $x_T$ can be described similarly:
$$E_ix_T=\sum_{T''} v^{b_{T,T''}} x_{T''}$$
where the sum is taken over all tableaux in $T''\in \mathrm{Col}^{(N^{\ell})}$ 
obtained from $T$ by changing one $i+1$ into $i$. Suppose that $T''$ is such 
a tableau and that it was obtained by changing the entry $(k,l)$ of $T$, then 
$$b_{T,T''}=\vert\{(k',l')\mid l'<l,\; T_{(k',l')}=i\}\vert-\vert
\{(k',l')\mid l'<l,\; T_{(k',l')}=i+1\}\vert.$$  
  
\begin{proposition}
\label{prop:vqHowe}
The action of $\Uv{m}$ on 
$\bigoplus_{\vec{k}\in \Lambda(m,m)_N}\Lambda_q^{\vec{k}}(\C_q^N)^{\otimes m}$ 
is given by 
$$E_{-i}x^T=\sum_{T'} v^{-a_{T,T'}}x^{T'}\quad E_ix^T=\sum_{T''} v^{b_{T,T''}} x^{T''},$$
for any $T\in \mathrm{Col}^{(N^{\ell})}$. Here $T',T'',a_{T,T'},b_{T,T''}$ are as 
above. 
\end{proposition}
\begin{proof}

We compute the action of $\Uv{m}$ on 
$\bigoplus_{\vec{k}\in \Lambda(m,m)_N}\Lambda_q^{\vec{k}}(\C_q^N)^{\otimes m}$ using 
Proposition~\ref{prop:CKM}. Recall that 
$$x^T=x_{S_m}\otimes \cdots \otimes x_{S_{i+1}}\otimes x_{S_i}\otimes\cdots\otimes x_{S_1},$$ 
where $S_k$ is the set of numbers of the columns in $T$ containing 
a filling equal to $k$. A small calculation shows that 
$$E_{-i} x^T= \sum_{j\in S_i\backslash S_{i+1}} v^{-(\ell(j,S_i\backslash \{j\})-\ell(j,S_{i+1}))}
x_{S_m}\otimes\cdots \otimes x_{S_{i+1}\cup \{j\}}\otimes 
x_{S_i\backslash \{j\}}\otimes\cdots x_{S_1}.$$
It follows immediately from the definitions that 
$$a_{T,T'}=\ell(j,S_i\backslash \{j\})-\ell(j,S_{i+1})$$
if $T'\in \mathrm{Col}^{(N^{\ell})}$ corresponds to 
$x_{S_m}\otimes\cdots \otimes x_{S_{i+1}\cup \{j\}}\otimes 
x_{S_i\backslash \{j\}}\otimes\cdots x_{S_1}$.

Similarly, we have 
$$E_{i} x^T= \sum_{j\in S_{i+1}\backslash S_{i}} v^{\ell(S_i,j)-\ell(S_{i+1}\backslash \{j\},j)}
x_{S_m}\otimes\cdots \otimes x_{S_{i+1}\backslash \{j\}}\otimes 
x_{S_i\cup \{j\}}\otimes\cdots x_{S_1}$$
and 
$$
b_{T,T''}=\ell(S_i,j)-\ell(S_{i+1}\backslash \{j\},j)
$$
if $T''\in \mathrm{Col}^{(N^{\ell})}$ corresponds to 
$x_{S_m}\otimes\cdots \otimes x_{S_{i+1}\backslash \{j\}}\otimes 
x_{S_i\cup \{j\}}\otimes\cdots x_{S_1}$. 
\end{proof}

\vskip0.5cm
As before, let $\Lambda:=N\omega_{\ell}$, where $\omega_{\ell}$ is the 
$\ell$-th fundamental $\mathfrak{sl}_m$-weight.  
Let $V_{\Lambda}$ be the irreducible $\Uv{m}$-module of highest weight 
$\Lambda$. It is well-known that $V_{\Lambda}$ is isomorphic to a 
direct summand of 
$\Lambda_v^{\ell}(\C_v^m)^{\otimes N}$. Therefore, 
$V_{\Lambda}$ is also isomorphic to a direct summand of 
the $\Uv{m}$-module $\Lambda_q^{\bullet}(\C_q^N)^{\otimes m}$. In order 
to see this explicitly, Let $P_{\Lambda}$ be the set of 
$\mathfrak{sl}_m$-weights of $V_{\Lambda}$.  
Note that the restriction of $\phi_{m,m,N}$ to $P_{\Lambda}\subset \Z^{m-1}$
gives a bijection
$$\phi_{m,m,N}\colon P_{\Lambda}\to \Lambda(m,m)_N.$$ 
In particular, we have $\phi_{m,m,N}(\Lambda)=(N^{\ell})$.  

\begin{definition}
\label{def:webmodule}
For any $\vec{k}=(k_1,\ldots,k_m)\in \Lambda(m,m)_N$, define the 
{\em web space} $W(\vec{k},N)$ by 
$$W(\vec{k},N):=\mathrm{Hom}((N^{\ell}),\vec{k})$$
in $\mathcal{S}p(\mathrm{SL}_N)$. Define also the $\Uv{m}$-{\em web module} 
with highest weight $\Lambda$ by 
$$W_{\Lambda}:=\bigoplus_{\vec{k}\in \Lambda(m,m)_N}W(\vec{k},N).$$
\end{definition}

\begin{remark}
Note that $W(\vec{k},N)$ is isomorphic to the space of invariant 
$\U{N}$-tensors in $\Lambda_q^{\vec{k}}(\C_q^N)$. 
\end{remark}

The action of $\Uv{m}$ on $W_{\Lambda}$  
is defined by applying the functor in Proposition~\ref{prop:CKM} 
and glueing the ladders on top of the webs in $W_{\Lambda}$ 
when the $\mathfrak{gl}_m$-weights match. If the $\mathfrak{gl}_m$-weights 
do not match, simply define the action to be zero. 
\vskip0.5cm
For any $u\in W(\vec{k},N)$, 
let $$u^*\in \mathrm{Hom}(\vec{k},(N^{\ell}))$$ be the web obtained via 
reflexion in the $x$-axis and reorientation. Note that $u$ and $u^*$ can 
be glued together such that $u^*u\in \mathrm{End}((N^{\ell}))$. 

Let $w_{\Lambda}$ be the web with $\ell$ vertical $N$-strands. Note that 
$$\mathrm{End}((N^{\ell}))\cong\C(v)w_{\Lambda},$$
so we can define a map 
$$\mathrm{ev}\colon \mathrm{End}((N^{\ell}))\to\C(v).$$ 
In~\cite{my}, the {\em $v$-sesquilinear web form} was 
defined by
\begin{equation}
\label{eq:webbracket}
\langle u, v\rangle:= v^{d(\vec{k})}\mathrm{ev}(u^*v)\in \C(v).
\end{equation}
for any monomial webs $u,v\in W(\vec{k},N)$. The normalization 
factor is defined by  
$$d(\vec{k})=1/2\left(N(N-1)\ell-\sum_{i=1}^{m}k_i(k_i-1)\right).$$
We assume that the web form is $v$-antilinear in the first entry and 
$v$-linear in the second. Note that all web relations 
have coefficients which are symmetric in $v$ and $v^{-1}$, so 
this definition makes sense.

Recall also the $v$-antilinear algebra anti-automorphism  
$\tau$ on $\Uv{m}$ defined by 
\[
\tau(1_{\lambda})=1_{\lambda},\;\;
\tau(1_{\lambda+i'}E_{+i}1_{\lambda})= 
v^{-1-\lambda_i}1_{\lambda}E_{-i}1_{\lambda+i'},\]
\[
\tau(1_{\lambda}E_{-i}1_{\lambda+i'})= v^{1+\lambda_i}
1_{\lambda+i'}E_{+i}1_{\lambda}.
\] 
The $v$-{\em Shapovalov form} $\langle\cdot,\cdot\rangle$ on 
$V_{\Lambda}$ is the unique $v$-sesquilinear form such that 
\begin{enumerate}
\item $\langle v_{\Lambda},v_{\Lambda}\rangle =1$, for a fixed 
highest weight vector $v_{\Lambda}$;
\item $\langle X v, v' \rangle=\langle v,\tau(X) v'\rangle$, for any 
$X\in \Uv{m}$ and any $v,v'\in V_{\Lambda}$.
\end{enumerate}  

In~\cite{my}, the following corollary was proved. 

\begin{corollary}
\label{cor:webirrep}
With the action above, $W_{\Lambda}$ is an irreducible $\Uv{m}$-representation 
with highest weight $\Lambda$. The $v$-sesquilinear 
web form is equal to the $v$-Schapovalov form. 
\end{corollary}

\begin{remark}
\label{rem:VversusW}
To circumvent the $v$ versus $q$ problem, in the rest of the paper 
I will consider $W_{\Lambda}$ as a $\C(v)$-vector space and 
write $V_{\Lambda}$ for the quotient of 
$\Uv{m}$ by the kernel of its action on $W_{\Lambda}$. Let 
$v_{\Lambda}$ for the image of $1_{\lambda}$ in $V_{\Lambda}$. We then have 
a well-defined $\C(v)$-linear $\Uv{m}$-isomorphism 
$V_{\Lambda}\to W_{\Lambda}$ such that  
$$Xv_{\Lambda}\mapsto Xw_{\Lambda}$$
for all $X\in \Uv{m}$. 
\end{remark}

%%%%%%%%%%%%%%%%%%%%%%%%%%%%%%%%%%%%%%%%%%%%%%%%%%%%%%%%%%%%%%%%%%%%%%%%%%
\subsection{The dual canonical basis}
\label{sec:canonical}

There is a well-known basis of $W(\vec{k},N)$, 
the {\em dual canonical basis}. In this section I recall 
its definition, following the exposition in~\cite{fkhk}. 
\vskip0.5cm
In Section~\ref{sec:TT} I recalled that, for any 
$\vec{k}=(k_1,\ldots,k_m)\in\Lambda(m,m)_N$, 
the zero weight basis tensors of $\Lambda_q^{\vec{k}}(\C_q^N)$
are in one-to-one correspondence with the elements of 
$\mathrm{Col}^{(N^{\ell})}_{\vec{k}}$. I also recalled the 
total order on $\mathrm{Col}^{(N^{\ell})}$, which of course 
restricts to $\mathrm{Col}^{(N^{\ell})}_{\vec{k}}$ and 
induces a total order on the subspace of zero-weight basis tensors in 
$\Lambda_q^{\vec{k}}(\C_q^N)$. 

Using the isomorphism $D_k\colon \Lambda^k(\C_q^N)\to 
(\Lambda^{N-k}(\C_q^N))^*$, we can extend this 
order to the zero weight basis tensors in 
$$(\Lambda_q^{k_m}(\C_q^N))^{\epsilon_m}\otimes\cdots\otimes 
(\Lambda_q^{k_1}(\C^N_q))^{\epsilon_1},$$
for any object $(k_1^{\epsilon_1},\ldots,k_m^{\epsilon_m})$ in 
$\mathcal{S}p(\mathrm{SL}_N)$. 
\vskip0.5cm
The coproduct $\Delta$ on ${\mathbf U}_q(\mathfrak{sl}_N)$ in this paper is the 
same as that in~\cite{ckm} and~\cite{cp}. The corresponding {\em 
quasi $R$-matrix} 
\begin{equation}
\label{eq:quasiR}
\Theta\in 1\widehat{\otimes} 1+{\mathbf U}_q(\mathfrak{sl}_N)^+\widehat{\otimes}{\mathbf U}_q(\mathfrak{sl}_N)^-
\end{equation} 
is given in Section 10.1 D in~\cite{cp} (where it is denoted $\tilde{R}_h$). 
Here $\widehat{\otimes}$ is a suitably completed tensor product. The 
precise formula for $\Theta$ is complicated and not needed in this paper, 
but let us give two simple examples of the action of $\Theta$ on 
tensor products of fundamental representations.  
\begin{example}
\label{ex:Rmatrix}
The action of $\Theta$ on $\C_q^2\otimes \C_q^2$ is given by 
$$1\otimes 1+(q-q^{-1})E_1\otimes E_{-1}$$
and its action on $\C_q^3\otimes \Lambda_q^2\C_q^3$ is given by 
$$
1\otimes 1+q\left(1-q^{-2}\right)\left(E_1\otimes E_{-1}+ \left(-E_1E_2+q^{-1}E_2E_1\right)
\otimes \left(-E_{-2}E_{-1}+qE_{-1}E_{-2}\right)+E_2\otimes E_{-2}\right)$$
$$+q^2\left(1-q^{-2}\right)^2 E_1E_2\otimes E_{-1}E_{-2}.
$$
\end{example}
Lusztig showed in Theorem 4.1.2 in~\cite{lu} that $\Theta$ is uniquely 
determined by its form in~\eqref{eq:quasiR} and the property
\begin{equation}
\label{eq:uniqueR}
\Theta\Delta(u)=\overline{\Delta}(u)\Theta,
\end{equation} 
for all $u\in {\mathbf U}_q(\mathfrak{sl}_N)$. Recall that the 
{\em bar-involution} on 
${\mathbf U}_q(\mathfrak{sl}_N)$ is the $q$-antilinear algebra 
involution defined by 
$$\overline{E_{\pm i}}=E_{\pm i},\;\overline{K_i}=K_i^{-1}$$
and $\overline{\Delta}$ is the coproduct defined by  
$$\overline{\Delta}(u):=\overline{\Delta(\overline{u})},$$
where the bar-involution acts on 
${\mathbf U}_q(\mathfrak{sl}_N)\widehat{\otimes}{\mathbf U}_q(\mathfrak{sl}_N)$ 
factorwise. By $q$-antilinear is meant that the map is $\C$-linear but 
sends $q$ to $q^{-1}$. 

The relation with Lusztig's coproduct $\Delta_L$ and quasi $R$-matrix 
$\Theta_L$ is easy to give:
$$\Delta_L=(\rho'\otimes\rho')\Delta\rho'\quad\Theta_L=
(\rho'\widehat{\otimes}\rho')
\Theta,$$
where $\rho'$ is the $q$-antilinear algebra anti-involution defined by 
$$\rho'(E_{+i})=E_{-i},\;\rho'(E_{-i})=E_{+i},\;\rho'(K_i)=K_i^{-1}.$$ 
Note that 
$$\overline{\rho'(x)}=\rho'(\overline{x}).$$

Corollary 4.1.3 in~\cite{lu} implies that 
\begin{equation}
\label{eq:Rinvertible}
\Theta\overline{\Theta}=\overline{\Theta}\Theta=1\widehat{\otimes}1.
\end{equation}

For each $k\in \N$, the bar-involution on $\Lambda_q^k(\C^N_q)$ 
is given by 
$$\psi(\sum_{\nu} a_{\nu}(q) x_{\nu})=\sum_{\nu} a_{\nu}(q^{-1}) 
x_{\nu}.$$
This bar-involution is compatible with the one on 
${\mathbf U}_q(\mathfrak{sl}_N)$ in the sense that 
$$\psi(Xz)=\overline{X}\;\psi(z)$$ 
for any $X\in {\mathbf U}_q(\mathfrak{sl}_N)$ and $z\in \Lambda_q^k(\C^N_q)$. 

In order to extend the bar-involution to tensor products of fundamental 
representations, write any tensor $z$ as 
$x\otimes y$, with $x$ and $y$ each having fewer tensor factors than $z$. 
Define inductively
$$\psi(x\otimes y):=\overline{\Theta}(\psi(x)\otimes\psi(y))
.$$
One can show that this definition does not depend on how one chooses 
the factorization $z=x\otimes y$. Note that this really 
is an involution by~\eqref{eq:Rinvertible}, and that 
$$\psi(X(\bullet\otimes \bullet))=\overline{X}\psi(\bullet\otimes \bullet)$$
for any $X\in {\mathbf U}_q(\mathfrak{sl}_N)$. 

Since $$\Theta\in 1\otimes 1+
{\mathbf U}_q(\mathfrak{sl}_N)^+\widehat{\otimes}
{\mathbf U}_q(\mathfrak{sl}_N)^-,$$
it follows that, for any $T\in\mathrm{Col}^{(N^{\ell})}_{\vec{k}}$, we have 
\begin{equation}
\label{eq:baractionontensors}
\psi(x^T)=x^{T}+\sum_{T'\prec T} c^{T,T'}(q) x^{T'}
\end{equation}
with $T'\in \mathrm{Col}^{(N^{\ell})}_{\vec{k}}$ and 
$c^{T,T'}(q)\in \C(q)$.

As already remarked, the dual canonical basis elements of $W(\vec{k},N)$ 
are parametrized by the subset 
$$\mathrm{Std}^{(N^{\ell})}_{\vec{k}}\subset \mathrm{Col}^{(N^{\ell})}_{\vec{k}}$$ 
of semi-standard tableaux. From now on, I adopt 
Leclerc and Toffin's~\cite{lt} convention 
to use Greek lower case letters for column-strict tableaux and 
Roman upper case letters for semi-standard tableaux. The 
following theorem is proved in 
Chapter 27 in~\cite{lu}, for example. Note that we switch to $v=-q^{-1}$ again.  
\begin{theorem}[Kashiwara, Lusztig]
\label{thm:KLdualcan}
For any $T\in \mathrm{Std}^{(N^{\ell})}_{\vec{k}}$, there exists a 
unique element $b^T\in W(\vec{k},N)$ such that 
\begin{equation}
\psi(b^T)=b^T
\end{equation}
\begin{equation}
\label{eq:negexp}
b^T=x^T+\sum_{\tau\prec T} d^{\tau,T}(v)x^{\tau}
\end{equation}
with $\tau\in \mathrm{Col}^{(N^{\ell})}_{\vec{k}}$ and 
$d^{\tau,T}(v)\in v^{-1}\Z[v^{-1}]$. The property 
in~\eqref{eq:negexp} is called the {\em negative exponent property}.

The $b^T$ form a basis of $W(\vec{k},N)$ which is called the 
{\bf dual canonical basis}. 
\end{theorem}

In the following examples I use the notation $x_S$ where $S$ is a subset 
of $\{1,\ldots,N\}$ as explained in Section~\ref{sec:fundreps}. 
\begin{example}
\label{ex:smallexs}
Using the expressions in Example~\ref{ex:Rmatrix} one can easily check 
that 
$$
x_2\otimes x_1+v^{-1}x_1\otimes x_2\in W((1,1),2)
$$
and
$$
x_3\otimes x_{\{2,1\}}+v^{-1}x_2\otimes x_{\{3,1\}}+v^{-2}x_1\otimes x_{\{3,2\}}\in 
W((1,2),3)
$$
are invariant $\U{2}$ and $\U{3}$ tensors which are 
$\psi$-invariant. Since they have the negative exponent property too, 
they are dual canonical basis elements.  
\end{example}

Lusztig also defined a symmetric bilinear inner product on 
$\Lambda^{\vec{k}}_q(\C_q^N)$. On zero-weight vectors it is defined by 
$$(x^{\tau},x^{\tau'}):=
\delta_{\tau,\tau'}$$
for any $\tau,\tau'\in \mathrm{Col}^{(N^{\ell})}_{\vec{k}}$.
The negative exponent property for the dual canonical basis 
elements is equivalent to the {\em almost orthogonality} property:
$$(b^T, b^{T'})\in \delta_{T,T'}+v^{-1}\Z[v^{-1}],$$
for all $T,T'\in\mathrm{Std}^{(N^{\ell})}_{\vec{k}}$. 

Alternatively, one can use the $v$-sesquilinear form defined by 
$$
\langle x^{\tau},x^{\tau'}\rangle := \overline{(x^{\tau},\psi(x^{\tau'}))}.  
$$
Since the dual canonical basis elements are $\psi$-invariant, we have 
$$\langle b^{T}, b^{T'}\rangle=\overline{(b^{T}, b^{T'})}$$
for any $T,T'\in\mathrm{Std}^{(N^{\ell})}_{\vec{k}}$. In particular we have 
$$\langle b^T,b^{T'}\rangle \in \delta_{T,T'}+ v\Z[v],$$
for all $T,T'\in \mathrm{Std}^{(N^{\ell})}_{\vec{k}}$. 
Note also that $\langle \cdot,\cdot\rangle$ is not 
symmetric, but satisfies
$$
\langle x,x'\rangle=\langle \psi(x'),\psi(x)\rangle.
$$
\vskip0.5cm
Just as we defined the $v$-sesquilinear 
web form $\langle\cdot,\cdot,\rangle$ on $W(\vec{k},N)$, we can also 
define a symmetric {\em $v$-bilinear web form} $(\cdot,\cdot)$ by 
$$\langle \cdot,\cdot\rangle=\overline{(\cdot,\psi(\cdot))}.$$ 

Let $w\in W(\vec{k},N)$ be any monomial web and write 
$$w=\sum_{\tau\in \mathrm{Col}^{(N^{\ell})}_{\vec{k}}} c^{\tau}(v) x^{\tau}$$
for certain coefficients $c^{\tau}(v)\in \C(v)$.   
\begin{lemma}
\label{lem:bilinearforms}
We have 
$$(w,w)=\sum_{\tau\in\mathrm{Col}^{(N^{\ell})}_{\vec{k}}} c^{\tau}(v)^2.$$
This means that the symmetric $v$-bilinear web form is the restriction to 
$W(\vec{k},n)$ of Lusztig's bilinear form on $\Lambda_q^{\vec{k}}(\C_q^{N})$. 
\end{lemma}
\begin{proof}
One way to compute $\mathrm{ev}(w^*w)$ is to use the 
state-sum model for intertwiners. Any state in $\mathcal{S}(w^*w)$ is given by 
a pair of states in $\mathcal{S}(w)$ and $\mathcal{S}(w^*)$ 
which match on the boundary. Therefore, it suffices to compare 
the powers of $v$ associated to $u$ and $u^*$ for any given state 
when $u$ is a cup or $Y$ or $Y^*$-shaped. I am using here the fact that 
there exists a bijection between $\mathcal{S}(u)$ and $\mathcal{S}(u^*)$ by 
assigning exactly the same weight to two corresponding edges in $u$ and $u^*$. 
   
I only show the case for a $Y$-shaped web, the other cases being similar. 
Suppose that 
$$
u=\txt{\input{figure/diag1}}
$$
and that a certain state assigns to 
the bottom $(a+b)$-edge an $(a+b)$-element subset $S\subseteq 
\{1,\ldots,N\}$, to the upper $a$-edge an $a$-element subset $T\subset S$ 
and to the upper $b$-edge the $b$-element subset $S\backslash T$. 
For this state, the corresponding 
power of $v$ is equal to 
$$
v^{-\ell(T,S\backslash T)}.
$$ 
For $u^*$, with the corresponding state, we get 
$$
v^{\ell(S\backslash T,T)}.
$$
Recall that 
$$\ell(T,S\backslash T)+\ell(S\backslash T,T)=ab.$$

Now compute $d(\vec{k})$ at the top and the bottom of $u$. We have 
$$
a(a-1)+b(b-1)=(a+b)(a+b-1)-2ab,
$$
so 
$$
d(\vec{k}_{\mathrm{top}})=v^{ab} d(\vec{k}_{\mathrm{bottom}}).
$$ 

Doing the same analysis for cups and $Y^*$-shaped webs, we arrive at 
$$
\mathrm{ev}(w^*w)=v^{d(\vec{k})}\sum_{T\in \mathrm{Col}^{(N^{\ell})}_{\vec{k}}} c^T(v)^2,
$$
which implies the lemma.
\end{proof}

\vskip0.5cm

Here are some more examples of dual canonical basis elements. In these 
examples I also use the notation of subsets.  
\begin{lemma}
\label{lem:dualcanex}
The following elements are all dual canonical:
\begin{enumerate}
\item 
$$
b^{++}_{a(N-a)}:=M'_{a(N-a)}(x_{\{N,\ldots,1\}})=
\sum_{|T|=a} v^{-\ell(T,T^c)} x_T\otimes x_{T^c},
$$
for any $1\leq a\leq N$;
\item 
$$ 
b_{a}^{-+}:=(D_{N-a}\otimes 1)(b_{(N-a)a}^{++})\quad\text{and}
\quad b_a^{+-}:=(1\otimes D_{N-a})(b_{a(N-a)}^{++}),
$$
for any $1\leq a\leq N$;
\item 
$$
t^{+++}_{abc}:=(M'_{ab}\otimes 1)(b^{++}_{(N-c)c}),
$$
for any $1\leq a,b,c\leq N$ such that $a+b+c=N$; 
\item 
$$
t^{---}_{abc}:=(1\otimes D_{a+c}M_{ac}\otimes 1)(b^{-+}_{a}\otimes b^{+-}_{c})
$$
for any $1\leq a,b,c\leq N$ such that $a+b+c=N$. 
\end{enumerate}
\end{lemma}
\begin{proof}
\begin{enumerate}
\item Note that $b_{a(N-a)}^{++}$ has the negative exponent property 
with top term   
$$x_{\{N,\ldots,N-a+1\}}\otimes x_{\{N-a,\ldots,1\}}.
$$ 
Since  
$\mathrm{Hom}(\Lambda^N(\C_q^N),\Lambda^a(\C^N_q)\otimes \Lambda^{N-a}(\C_q^N))
$ has dimension one, this implies that $b_{a(N-a)}^{++}$ 
has to be dual canonical. 
\item Follows directly from (1).
\item Note that $t^{+++}_{abc}$ has the negative exponent property with 
top term 
$$
x_{\{N,\ldots,b+c+1\}}\otimes x_{\{b+c,\ldots,c+1\}}\otimes x_{\{c,\ldots,1\}}
$$
Again we have 
$$\dim (\mathrm{Hom}(\Lambda^N(\C_q^N),\Lambda^a(\C^N_q)\otimes \Lambda^b(\C_q^N)
\otimes \Lambda^c(\C^N_q)))=1,$$
so $t^{+++}_{abc}$ is indeed dual canonical.
\item Note that $t^{---}_{abc}$ has the negative exponent property with 
top term
$$
\hat{x}_{\{N,\ldots,a+1\}}\otimes 
\hat{x}_{\{N,\ldots,a+b+1,a\ldots,1\}}\otimes 
\hat{x}_{\{a+b,\ldots,1\}}
$$
Also in this case the relevant hom-space is one-dimensional, so 
the result follows.  
\end{enumerate}
\end{proof}
Not all dual canonical basis elements can be 
represented by monomial webs. However, one can prove the 
following analogue of Proposition 2 in~\cite{kk}. I identify webs 
with intertwiners, so I call a certain web 
$\psi$-equivariant if the corresponding intertwiner is $\psi$-equivariant. 
\begin{proposition} 
\label{prop:webpsiinvariant}
Any monomial web in $\mathcal{S}p(\mathrm{SL}_N)$ is $\psi$-equivariant. In particular, any monomial web in 
$W(\vec{k},N)$ is $\psi$-invariant, for any $m=N\ell$ and 
$\vec{k}\in \Lambda(m,m)_N$. 
\end{proposition}
\begin{proof}
By Lemma~\ref{lem:dualcanex}, the proposition holds true for the following 
webs (the $\pm$-signs appear because the pictures do not specify whether 
left or right tags are used at the bi-valent vertices):
\vskip0.5cm
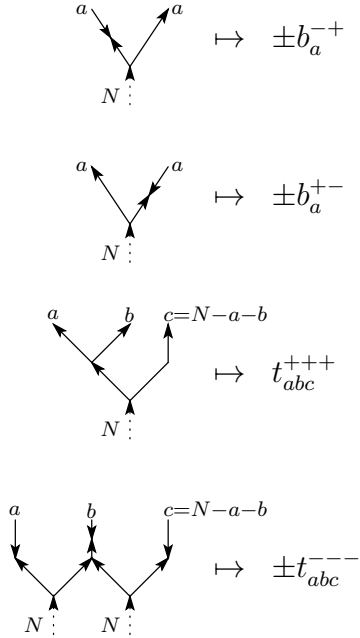
\begin{eqnarray*}
\txt{\input{figure/diag31}}&\mapsto&\pm b^{-+}_{a}\\[1em]
\txt{\input{figure/diag32}}&\mapsto&\pm b^{+-}_{a}\\[1em]
\txt{\input{figure/diag33}}&\mapsto&t^{+++}_{abc}\\[1em]
\txt{\input{figure/diag34}}&\mapsto&\pm t^{---}_{abc}\\[1em]
\end{eqnarray*}

By the zig-zag relations for cups and caps, the $\psi$-invariance of 
the cups implies the $\psi$-equivariance of the caps. 

By definition, the web 
$$
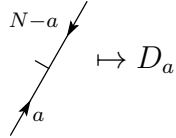
\txt{\input{figure/diag4}}\mapsto D_a
$$ 
is also $\psi$-equivariant. 

Up to a sign, any monomial web can be obtained by gluing instances 
of the above ones, so we see that any monomial web is $\psi$-equivariant. 
\end{proof}

\begin{remark}
\label{rem:normalization}
For Lemma~\ref{lem:dualcanex} and Proposition~\ref{prop:webpsiinvariant} to 
be true, the generating intertwiners in the image of 
$\Gamma_N$ have to be normalized as in Section~\ref{sec:fundreps}. 
\end{remark}

\subsection{The canonical and the LT-bases and their Howe duals}
\label{sec:webbasis}

In this section I recall the canonical basis and Leclerc and Toffin's 
intermediate basis of $V_{\Lambda}$. The latter basis 
gives rise to a nice basis of $W(\vec{k},N)$ by quantum skew Howe duality. 
Finally I will show that one can identify the canonical $\Uv{m}$-basis 
and the dual canonical $\U{N}$-basis of $W_{\Lambda}$. 
\vskip0.5cm
Let us first briefly recall the {\em canonical basis} of 
the irreducible $\Uv{m}$-representation $V_{\Lambda}$. For more details 
the reader can consult~\cite{bk2,bs3,lu} for example. 

As was already remarked above, $V_{\Lambda}$ is isomorphic to a direct summand of 
$\Lambda^{\ell}_v(\C_v^m)^{\otimes N}$. 
The highest weight vector corresponds to 
$$v_{\Lambda}:=x_{T_{\Lambda}}=
x_{\{12\ldots \ell\}}\otimes \cdots\otimes x_{\{12\ldots \ell\}}\in 
\Lambda^{\ell}_v(\C_v^m)^{\otimes N}.$$ 

There is a unique $v$-antilinear bar-involution on $V_{\Lambda}$ determined 
by the conditions
\begin{gather*}
\widetilde{\psi}(v_{\Lambda})=v_{\Lambda}\\
\widetilde{\psi}(Xv)=\overline{X} \tilde{\psi}(v),
\end{gather*}
for any $X\in\Uv{m}$ and $v\in V_{\Lambda}$. 
\vskip0.5cm

Just as the dual canonical basis of $W_{\Lambda}$, the canonical 
basis of $V_{\Lambda}$ is parametrized by $\mathrm{Std}^{(N^{\ell})}$. 
For a proof of the following theorem see~\cite{lt} 
and the references therein. Note that our coproduct differs from that in~\cite{lt}, so we get the negative exponent property rather than the positive one.  

\begin{theorem}[Kashiwara, Lusztig]
\label{thm:KLcan}
For each $T\in\mathrm{Std}^{(N^{\ell})}$, there exists a unique element $b_T$ 
such that 
\begin{eqnarray*}
\tilde{\psi}(b_T)&=&b_T\\
b_T&=&x_T+\sum_{T'\prec T} d_{\tau,T}(v) x_{\tau}
\end{eqnarray*}
with $\tau\in \mathrm{Col}^{(N^{\ell})}$ and $d_{\tau,T}(v)\in v^{-1}\Z[v^{-1}]$. The 
elements $b_T$ form the {\bf canonical basis} of $V_{\Lambda}$. 
\end{theorem} 

As before, we have a symmetric $v$-bilinear form $(\cdot,\cdot)$ on 
$\Lambda^{\ell}_v(\C_v^m)^{\otimes N}$ determined by 
$$(x_{\tau},x_{\tau'})=\delta_{\tau,\tau'}$$
for all $\tau,\tau'\in \mathrm{Col}^{(N^{\ell})}$. Lemmas 17.1.3 and 26.2.2 
in~\cite{lu} show that the restriction of this 
form to $V_{\Lambda}$ is the unique symmetric $v$-bilinear 
form on $V_{\Lambda}$ satisfying 
\begin{enumerate}
\item $(v_{\Lambda},v_{\Lambda})=1$;
\item $\langle X v, v' \rangle=\langle v,\rho(X) v'\rangle$, for any 
$X\in \Uv{m}$ and any $v,v'\in V_{\Lambda}$.
\end{enumerate}
Here $\rho$ is the $v$-linear anti-involution on $\Uv{m}$ defined in 
Section 19.1 in~\cite{lu}. 

One can easily check that 
$$\tau(X)=\overline{\rho(X)}$$
for any $X\in\Uv{m}$, so the relation between $(\cdot,\cdot)$ and 
the $v$-Shapovalov form is given by 
$$(\cdot,\cdot):=\overline{\langle \cdot,\widetilde{\psi}(\cdot)\rangle}.$$

Note that the negative exponent property in the above theorem is equivalent 
to the condition
$$
(b_{T},b_{T'})\in \delta_{T,T'}+v^{-1}\Z[v^{-1}].
$$
Since the canonical basis elements are $\widetilde{\psi}$-invariant, this 
is also equivalent to  
\begin{equation}
\label{eq:Shap}
\langle b_{T},b_{T'}\rangle=\overline{(b_{T},b_{T'})}\in 
\delta_{T,T'}+v\Z[v].
\end{equation}

\vskip0.5cm
Leclerc and Toffin~\cite{lt} (Section 4.1) defined a different basis of 
$V_{\Lambda}$, which I denote 
by $B_{\Lambda}$. The elements of $B_{\Lambda}$ 
are also parametrized by the elements in $\mathrm{Std}^{(N^{\ell})}$. 

Suppose $T\in \mathrm{Std}^{(N^{\ell})}$ is arbitrary. 
Let us recall how to construct 
the Leclerc-Toffin (LT) basis element $A_T\in B_{\Lambda}$. 
Let $1\leq i_1\leq \ell$ be the smallest integer such that the rows of 
$T=T_1$ with row number $\leq i_1$ contain entries equal to $i_1+1$. Denote 
the total number of such entries by $r_1>0$. Change all these entries to 
$i_1$ and denote the new semi-standard tableau 
by $T_2\in\mathrm{Std}^{(N^{\ell})}$. 
Let $1\leq i_2\leq \ell$ be the smallest integer such that the 
rows of $T_2$ with row number $\leq i_2$ contain 
$r_2>0$ entries equal to $i_2+1$. Then change these entries to $i_2$ and 
denote the new tableau by $T_3\in \mathrm{Std}^{(N^{\ell})}$. 
Continue this way until $T_s=T_{\Lambda}\in \mathrm{Std}^{(N^{\ell})}$. 
Leclerc and Toffin define the basis element $A_T$ as 
\begin{equation}
\label{eq:LTdef}
A_T:=E_{-i_1}^{(r_1)}\cdots E_{-i_s}^{(r_s)}v_{\Lambda}.
\end{equation}

By definition, we have
\begin{equation}
\label{eq:LTtildepsi-inv}
\tilde{\psi}(A_T)=A_T,
\end{equation}
for any $T\in \mathrm{Std}^{(N^{\ell})}$. 

Using the coproduct of ${\mathbf U}_v(\mathfrak{sl}_m)$, 
it is easy to work out the expansion of 
$A_T$ on the standard basis of $\Lambda_v^{\ell}(\C_v^m)^{\otimes N}$. 
In Lemma 9 in~\cite{lt}, 
Leclerc and Toffin showed that 
\begin{equation}
\label{eq:LT}
A_T=x_T+\sum_{\tau\prec T} \alpha_{\tau T}(v) x_{\tau},
\end{equation}
with $\tau\in \mathrm{Col}^{(N^{\ell})}$ and certain coefficients 
$\alpha_{\tau T}(v)\in \N[v,v^{-1}]$. For any 
$\tau\in\mathrm{Col}^{(N^{\ell})}$ and $T\in\mathrm{Std}^{(N^{\ell})}$, the 
coefficient $\alpha_{\tau T}(v)=0$ if $\tau$ and $T$ do not have the 
same $\mathfrak{sl}_m$-weight.
  
Leclerc and Toffin gave an algorithm to compute the canonical basis of 
$V_{\Lambda}$ which uses $B_{\Lambda}$ as an intermediate basis. In Section 
4.2 they showed that 
\begin{equation}
\label{eq:LT2}
b_T=A_T + \sum_{S \prec T } \beta_{ST}(v) A_S,
\end{equation}
with $S\in \mathrm{Std}^{(N^{\ell})}$ 
for certain bar-invariant 
coefficients $\beta_{ST}(v)\in \mathbb{Z}[v,v^{-1}]$. For 
any $S,T\in \mathrm{Std}^{(N^{\ell})}$, the coefficient $\beta_{ST}(v)=0$ 
if $S$ and $T$ do not have the same $\mathfrak{sl}_m$-weight. 
\vskip0.5cm

By Corollary~\ref{cor:webirrep} we see that the quantum 
skew Howe duality functor in Proposition~\ref{prop:CKM} 
maps the basis $B_{\Lambda}$ of $V_{\Lambda}$ to a web basis of $W_{\Lambda}$. 
The elements of this web basis are defined in exactly 
the same way as those of $B_{\Lambda}$ by letting 
the divided powers of $\Uv{m}$ act on the highest weight web $w_{\Lambda}$. 
We call this web basis the {\em LT-web basis} and denote it by ${BW}_{\Lambda}$. 
We denote the element in ${BW}_{\Lambda}$ associated to a tableau 
$T\in\mathrm{Std}^{(N^{\ell})}$ by $A^T$. 

For any $T\in \mathrm{Std}^{(N^{\ell})}$, we have 
\begin{equation}
\label{eq:LTpsi-inv}
\psi(A^T)=A^T
\end{equation}
by Proposition~\ref{prop:webpsiinvariant}.

Using the functor $\Gamma_N\colon \mathcal{S}p(\mathrm{SL}_N)\to 
\mathcal{R}ep(\mathrm{SL}_N)$ in Theorem~\ref{thm:ckm}, we can write every 
basis web $A^T$ as a linear combination of standard basis elements 
in $\Lambda_q^{\bullet}(\C_q^N)^{\otimes m}$. 
The following proposition is the analogue of Theorem 2 in~\cite{kk} and 
follows immediately from~\eqref{eq:LT} and 
Proposition~\ref{prop:vqHowe} in this paper.   
\begin{proposition}
\label{prop:triangwebtensor}
For any $T\in\mathrm{Std}^{(N^{\ell})}$, we have 
$$A^T=x^T+\sum_{\tau \prec T}\alpha_{\tau ,T}(v) x^{\tau},$$ 
with $\tau\in \mathrm{Col}^{(N^{\ell})}$. 
\end{proposition}

Note that, for any $\tau\in \mathrm{Col}^{(N^{\ell})}$ and 
$T\in \mathrm{Std}^{(N^{\ell})}$, we have 
$$\sum_{\tau \prec S\prec T}\alpha_{\tau S}(v)\beta_{ST}(v)=d_{\tau T}(v),$$
with $S\in \mathrm{Std}^{(N^{\ell})}$. By Theorems~\ref{thm:KLdualcan} 
and~\ref{thm:KLcan}, equation~\eqref{eq:LTpsi-inv} and 
Proposition~\ref{prop:triangwebtensor} we obtain 
$$d^{\tau T}(v)=d_{\tau T}(v),$$
for any $\tau\in \mathrm{Col}^{(N^{\ell})}$ and 
$T\in \mathrm{Std}^{(N^{\ell})}$, and the following result. 
\begin{proposition}
\label{prop:dualcanLT}
For any $T\in\mathrm{Std}^{(N^{\ell})}$, we have 
$$b^T=A^T+\sum_{S\prec T}\beta_{ST}(v)A^S,$$
with $S\in \mathrm{Std}^{(N^{\ell})}$.  
\end{proposition}

\begin{corollary}
\label{cor:VversusW}
Under the isomorphism in Remark~\ref{rem:VversusW} we have 
$$A_T\mapsto A^T\quad\text{and}\quad b_T\mapsto b^T,$$
for all $T\in \mathrm{Std}^{(N^{\ell})}$.
\end{corollary}

%%%%%%%%%%%%%%%%%%%%%%%%%%%%%%%%%%%%%%%%%%%%%%%%%%%%%%%%%%%%%%%%%%%%%%%%%%%%%%%%%%%%%%%%%%%%%%%%%%%%%%%%
%
%
% section
%
%
%%%%%%%%%%%%%%%%%%%%%%%%%%%%%%%%%%%%%%%%%%%%%%%%%%%%%%%%%%%%%%%%%%%%%%%%%%%%%%%%%%%%%%%%%%%%%%%%%%%%%%%%

\section{Web categories and algebras}
\label{sec:webalgebras}
Wu~\cite{wu} and independently Yonezawa~\cite{yo1,yo2} defined $(\Z/2\Z\times 
\Z)$-graded matrix factorizations associated to colored 
$\mathfrak{sl}_N$-webs, generalizing the ground breaking work of 
Khovanov and Rozanksy~\cite{kr}. The $\Z/2\Z$-grading is generally 
referred to as the 
{\em homological grading} and the $\Z$-grading as the {\em quantum grading}. 
A shift of $1$ in the homological grading is indicated by 
$\langle 1\rangle$ and in the quantum grading by $\{1\}$. 

In~\cite{my} Yonezawa and I recalled these matrix factorizations in 
great detail, so I refer to that paper and its references for the relevant 
background on graded matrix factorizations in general and 
the definitions of and the results on 
$\mathfrak{sl}_N$-web matrix factorizations in particular.  

All the reader of this paper has to know, is that 
to any monomial $\mathfrak{sl}_N$-web $u\in W(\vec{k},N)$ 
without tags (but with oriented $N$-colored edges) one can associate a 
matrix factorization $\hat{u}$ such that   
$$
\widehat{u^*}=\hat{u}_{\bullet}\{-d(\vec{k})\}\langle 1\rangle,
$$
where $\hat{u}_{\bullet}$ is the dual matrix factorization. As explained 
in Section 5 in~\cite{my}, this implies that 
\begin{equation}
\label{eq:ExtvH}
\mathrm{EXT}(\hat{u},\hat{v})\cong H(\hat{u}_{\bullet}\btime{R^{\vec{k}}}\hat{v})
\cong H(\widehat{u^*v})\{d(\vec{k})\}\langle 1\rangle 
\end{equation}
for any monomial webs $u,v\in W(\vec{k},N)$.

For the proof of the following theorem, which was recalled as 
Theorem 5.14 in~\cite{my}, I refer to Sections 6 through 11 
in~\cite{wu} and Section 3 in~\cite{yo1}.  
\begin{theorem}[Wu, Yonezawa]
\label{thm:catwebrels}
The matrix factorizations associated to webs without tags satisfy all 
relations in Definition~\ref{def:webrelations}, except the first one, up 
to homotopy equivalence. These equivalences are quantum degree preserving, 
but might involve homological degree shifts. 
\end{theorem}

Because of the last remark in Theorem~\ref{thm:catwebrels}, people 
working on $\mathfrak{sl}_N$-web matrix factorizations 
usually forget about the homological degree in their notation. 
I will follow that tradition and will never write 
homological shifts explicitly in the equations below.  

\subsection{The graded web category}\label{sec:webcat}
Let $\vec{k}=(k_1,\ldots,k_m)\in\Lambda(m,m)_N$. 
In this section I recall the $\C$-linear category 
$\mathcal{W}^{\circ}(\vec{k},N)$ from~\cite{my}. 

\begin{definition}
\label{def:web-category}
The objects of $\mathcal{W}^{\circ}(\vec{k},N)$ are by definition 
all matrix factorizations which are homotopy equivalent to 
direct sums of matrix factorizations of the form $\hat{u}$, where 
$u$ is an $N$-ladder with $m$ uprights in $W(\vec{k},N)$. 

For any pair of objects $X,Y\in \mathcal{W}^{\circ}(\vec{k},N)$, 
we define 
$$\mathcal{W}^{\circ}(X,Y):=\mathrm{Ext}(X,Y).$$
Composition in $\mathcal{W}^{\circ}(\vec{k},N)$ 
is induced by the composition of 
homomorphisms between matrix factorizations. 

Note that $\mathcal{W}^{\circ}(\vec{k},N)^*$ is a $\mathbb{Z}$-graded $\C$-linear 
additive category which admits translation and has finite-dimensional 
hom-spaces. 
\end{definition}

Identifying ladders with the corresponding matrix factorizations, we see 
that $\mathcal{W}^{\circ}(\vec{k},N)$ is a full subcategory of the 
homotopy category of matrix factorization with fixed 
potential determined by $\vec{k}$ and $N$. The latter category is 
Krull-Schmidt by Propositions 24 and 25 in~\cite{kr}, so we can take 
the Karoubi envelope of $\mathcal{W}^{\circ}(\vec{k},N)$, 
denoted $\dot{\mathcal{W}^{\circ}}(\vec{k},N)$, 
which is also Krull-Schmidt.  

\subsection{Graded web algebras}
For any $T\in\mathrm{Std}^{(N^{\ell})}_{\vec{k}}$, 
let $\hat{A}^T$ denote the $\mathfrak{sl}_N$-matrix 
factorization associated to $A^T$ in~\cite{my}.
\begin{definition}
\label{def:webalgebra}
For any pair $S,T\in \mathrm{Std}^{(N^{\ell})}_{\vec{k}}$, define  
$${}_SH(\vec{k},N)_T:=\mathrm{EXT}(\hat{A}^S,\hat{A}^T)\;(\cong 
H(\widehat{(A^S)^*A^T})\{d(\vec{k})\}).$$

The {\em web algebra} $H(\vec{k},N)$ is defined by 
$$H(\vec{k},N):=\bigoplus_{S,T\in \mathrm{Std}^{(N^{\ell})}_{\vec{k}}} {}_SH(\vec{k},N)_T,$$
with multiplication induced by the composition of maps between matrix 
factorizations.

Note that $H(\vec{k},N)$ is a finite-dimensional graded unital associative 
algebra.  
\end{definition} 

The following proposition is due to Buchweitz, see Proposition 10.1.5 and 
Example 10.1.6 in~\cite{bu}. The ring 
$R$ in Buchweitz's Example 10.1.6 is equal to the center of $H(\vec{k},N)$ in our case.  
As I will show in Corollary~\ref{cor:center} below, this 
center is isomorphic to the 
complex cohomology ring of the Spaltenstein variety $X^{(N^{\ell})}_{\vec{k}}$ 
of partial flags in $\C^{m}$ of type $\vec{k}$ and 
nilpotent linear operator of Jordan type $(\ell^N)$. 
The Gorenstein parameter in the proposition below is 
therefore equal to twice the top dimension of 
$H^*(X^{(N^{\ell})}_{\vec{k}})$ (``twice'' because $\deg(x_i)=2$), 
which is equal to 
$$2d(\vec{k}):=N(N-1)\ell-\sum_{i=1}^{m}k_i(k_i-1)$$ 
by Theorem 1.2 in~\cite{bo} and the remarks following that theorem. 
\begin{proposition}\label{prop:Frob}
The web algebra $H(\vec{k},N)$ is a graded symmetric Frobenius 
algebra of Gorenstein parameter $2d(\vec{k})$. This means that 
$$H(\vec{k},N)^{\vee}\cong H(\vec{k},N)\{-2d(\vec{k})\}$$
as graded $H(\vec{k},N)-H(\vec{k},N)$ bimodules, where $H(\vec{k},N)^{\vee}$ is the graded 
dual. 
\end{proposition}

To fix notation, let us already introduce the following two definitions here. 
\begin{definition}
\label{def:webalgebramod}
\begin{eqnarray*}
\mathcal{W}(\vec{k},N)&:=&H(\vec{k},N)-\mathrm{mod}_{\mathrm{gr}};\\
\mathcal{W}^p(\vec{k},N)&:=&H(\vec{k},N)-\mathrm{pmod}_{\mathrm{gr}}.
\end{eqnarray*} 
\end{definition}

We also give some important examples of projective $H(\vec{k},N)$-modules. 
\begin{example}
\label{ex:leftandrightprojectives}
For $T\in \mathrm{Std}^{(N^{\ell})}_{\vec{k}}$, we define
$$P^{T}:=H(\vec{k},N)1_{\hat{A}^T}\{-d(\vec{k})\}\quad\text{and}\quad 
{}^TP:=1_{\hat{A}^T}H(\vec{k},N)\{-d(\vec{k})\}.$$
Since $\sum_{T\in \mathrm{Std}^{(N^{\ell})}_{\vec{k}}} 1_{\hat{A}^T}=1$ in $H(\vec{k},N)$, 
we see that $P^T$ and ${}^TP$ are graded left and right projective 
$H(\vec{k},N)$-modules, respectively. Both are finite-dimensional, of course. 

We have normalized the grading of $P^T$ and ${}^TP$ so that 
$$P^{T}\cong (P^T)^{\vee}\quad\text{and}\quad {}^TP\cong ({}^TP)^{\vee}$$
as graded vector spaces (see Proposition~\ref{prop:Frob}).  

We note that $P^T$ and ${}^TP$ can be decomposable as $H(\vec{k},N)$-modules 
in general.    
\end{example}

%%%%%%%%%%%%%%%%%%%%%%%%%%%%%%%%%%%%%%%%%%%%%%%%%%%%%%%%%%%%%%%%%%%%%

%%%         Diagrammatic categorical quantum $\mathfrak{sl}_m$

%%%%%%%%%%%%%%%%%%%%%%%%%%%%%%%%%%%%%%%%%%%%%%%%%%%%%%%%%%%%%%%%%%

\section{Categorified quantum $\mathfrak{sl}_m$ and $2$-representations}
\label{sec:catgroupandrep}
\subsection{Categorified $\Uv{m}$}\label{kl}
Khovanov and Lauda introduced diagrammatic 2-categories $\u(\mathfrak{g})$ 
which categorify the integral version of the corresponding 
idempotented quantum groups~\cite{kl3}. Independently, Rouquier 
also introduced similar 2-categories~\cite{rou}.
Subsequently, Cautis and Lauda~\cite{cl} defined diagrammatic 2-categories 
$\u_Q(\mathfrak{g})$ with implicit scalers $Q$ consisting of $t_{ij}$, $r_i$ and 
$s_{ij}^{pq}$ which determine certain signs in the definition of 
the categorified quantum groups.
\\
\indent
In this section, I recall $\u(\mathfrak{sl}_m)=\u_Q(\mathfrak{sl}_m)$ 
briefly. The implicit scalars $Q$ are given by 
$t_{ij}=-1$ if $j=i+1$, $t_{ij}=1$ otherwise, $r_i=1$ and $s_{ij}^{pq}=0$. 
This corresponds precisely to the signed 
version in~\cite{kl3,kl4}. The other conventions here 
are the same as those in Section~\ref{sec:fund}.  

\begin{definition}[Khovanov-Lauda]\label{def:KL}
The $2$-category $\u(\mathfrak{sl}_{m})$ is defined as follows:
\begin{itemize}
\item[$\bullet$] The objects in $\u(\mathfrak{sl}_m)$ are the weights $\l \in  \Z^{m-1} $.
\end{itemize}
For any pair of objects $\l$ and $\l'$ in $\u(\mathfrak{sl}_m)$, the hom category 
$\u(\mathfrak{sl}_m)(\l,\l')$ is the graded additive $\C$-linear category consisting of:
\begin{itemize}
\item[$\bullet$] objects ($1$-morphisms in $\u(\mathfrak{sl}_m)$), 
which are finite formal sums of the form 
$\e_{\ui}{\idm}_{\l}\{t\}$ where $t\in \Z$ is the grading shift 
and $\ui$ is a signed sequence such that 
$\l'=\l+\sum_{a=1}^{l}\epsilon_ai_{a}'$.
\item[$\bullet$] morphisms from $\e_{\ui}{\idm}_{\l}\{t\}$ to 
$\e_{\underline{l}}{\idm}_{\l}\{t'\}$ in $\u(\mathfrak{sl}_m)(\l,\l')$ 
($2$-morphisms in $\u(\mathfrak{sl}_m)$) are 
$\C$-linear combinations of diagrams with degree $t'-t$ spanned by 
composites of the following diagrams:
\end{itemize}

\indent
\begin{eqnarray*}
&&\txt{\input{figure/e-dot}}:\e_{+i}\idm_{\l}\to\e_{+i}\idm_{\l}\{a_{ii}\}\hspace{1cm}
\txt{\input{figure/f-dot}}:\e_{+i}\idm_{\l}\to\e_{-i}\idm_{\l}\{a_{ii}\}
\\[1em]
&&\hspace{-.5cm}\txt{\input{figure/e-f-cup}}:\idm_{\l}\to\e_{(-i,+i)}\idm_{\l}\{\l_i+1\}\hspace{.5cm}
\txt{\input{figure/f-e-cup}}:\idm_{\l}\to\e_{(+i,-i)}\idm_{\l}\{-\l_i+1\}
\\[1em]
&&\hspace{-.5cm}\txt{\input{figure/e-f-cap}}:\e_{(-i,+i)}\idm_{\l}\to\idm_{\l}\{\l_i+1\}\hspace{.5cm}
\txt{\input{figure/f-e-cap}}:\e_{(+i,-i)}\idm_{\l}\to\idm_{\l}\{-\l_i+1\}
\\[1em]
&&\hspace{2cm}\txt{\input{figure/e-j-i1}}:\e_{(+i,+l)}\idm_{\l}\to\e_{(+l,+i)}\idm_{\l}\{-a_{il}\}
\\[1em]
&&\hspace{2cm}\txt{\input{figure/f-j-i1}}:\e_{(-i,-l)}\idm_{\l}\to\e_{(-l,-i)}\idm_{\l}\{-a_{il}\}
\end{eqnarray*}
As already remarked, the relations on the $2$-morphisms are 
those of the signed version in~\cite{kl3,kl4}, which I do not recall here 
because I do not need them explicitly in this paper.  
\end{definition}

Khovanov and Lauda's main result for type $A$ in~\cite{kl3} 
was their Proposition 1.4.
\begin{theorem}[Khovanov-Lauda]
\label{thm:KL}
The linear map 
$$\gamma \colon \Uv{m}\to K_0^v(\u(\mathfrak{sl}_m))$$ 
defined by 
$$v^t E_{\ui}1_{\lambda}\to \e_{\ui}{\idm}_{\l}\{t\}$$
is an isomorphism of $\C(v)$-algebras. 
\end{theorem}

\subsection{Cyclotomic KLR algebras and $2$-representations}\label{sec:cyclKLR}
Let $\Lambda$ be a dominant $\mathfrak{sl}_m$-weight, 
$V_{\Lambda}$ the irreducible $\U{m}$-module of highest 
weight $\Lambda$ and $P_{\Lambda}$ the set of weights in $V_{\Lambda}$.
 
\begin{definition}[Khovanov-Lauda, Rouquier]\label{def:cyclKLR}
The {\em cyclotomic KLR algebra} $R_{\Lambda}$ is 
the subquotient of $\u(\mathfrak{sl}_m)$ 
defined by the subalgebra of all diagrams with only downward 
oriented strands and right-most region labeled $\Lambda$ 
modded out by the ideal generated by diagrams of the form  
$$
\txt{\input{figure/cyclo-rel}}
$$
\end{definition}
Note that 
$$R_{\Lambda}=\bigoplus_{\mu\in P_{\Lambda}} R_{\Lambda}(\mu),$$
where $R_{\Lambda}(\mu)$ is the subalgebra generated by all diagrams 
whose left-most region is labeled $\mu$. Brundan and Kleshchev proved that 
$R_{\Lambda}$ is finite-dimensional in Corollary 2.2 in~\cite{bk2}.  
We also define 
$$\mathcal{V}_{\Lambda}:=R_{\Lambda}-\mathrm{mod}_{\mathrm{gr}};$$
$$\mathcal{V}_{\Lambda}^p:=R_{\Lambda}-\mathrm{pmod}_{\mathrm{gr}}.$$

Below I will use Cautis and Lauda's language of strong $\mathfrak{sl}_m$ 
$2$-representations (see Definition 1.2 in~\cite{cl}). For a comparison 
with Rouquier's~\cite{rou} definition of a 
Kac-Moody $2$-representation, see Cautis and Lauda's remark (1) below their 
Definition 1.2.  

In Section 4.4 in~\cite{bk} Brundan and Kleshchev defined a strong 
$\mathfrak{sl}_m$ $2$-representation on $\mathcal{V}_{\Lambda}$, which can 
be restricted to $\mathcal{V}_{\Lambda}^p$.

They also defined a duality $\circledast\colon \mathcal{V}_{\Lambda}\to 
\mathcal{V}_{\Lambda}$. To define it, 
they used Khovanov and Lauda's graded anti-automorphism  
$$*\colon R_{\Lambda}\to R_{\Lambda},$$ 
which is defined by reflecting diagrams in the $x$-axis followed by 
inversion of the orientation. Let $M\in \mathcal{V}_{\Lambda}$. 
On the underlying vector space, we have  
$$M^{\circledast}:=M^{\vee}.$$
The action on $M^{\circledast}$ is defined by 
$$xf(p):=f(x^* p).$$
(Note that I use a left action throughout the paper, contrary to Brundan 
and Kleshchev.) This duality can be restricted to 
$\mathcal{V}_{\Lambda}^p$ and induces 
a $q$-antilinear involution on $K_0^q(\mathcal{V}_{\Lambda}^p)$. 

Brundan and Kleshchev~\cite{bk} proved the following result (Theorems 4.18 
and Theorem 5.14). 

\begin{theorem}[Brundan-Kleshchev]\label{thm:cyclKLR}
There exists a unique $\C(v)$-linear isomorphism  
$$\delta\colon V_{\Lambda}\to K_0^{v}(\mathcal{V}_{\Lambda}^p)$$
of $\Uv{m}$-modules, such that 
\begin{itemize} 
\item $\delta$ intertwines the bar-involutions 
$\tilde{\psi}$ and $\circledast$;
\item $\delta$ intertwines the $v$-Shapovalov form 
and the Euler form;   
\item for each $T\in\mathrm{Std}^{(N^{\ell})}$, there exists an indecomposable 
projective module $Q_T\in \mathcal{V}_{\Lambda}^p$ such that $Q_T=\delta(b_T)$ (note that 
this implies that $Q_T^{\circledast}\cong Q_T$).
\end{itemize} 
\end{theorem} 

\begin{remark}
Note that Brundan and Kleshchev~\cite{bk} 
use multipartitions instead of column-strict 
tableaux. There is a well-known bijection between these two combinatorial 
data, which I will not recall here. Using this bijection, their module 
$Y(T)$ in Theorem 5.14 satisfies   
$$Y(T)=Q_T\{d(\vec{k})\}$$
as follows from their Lemma 3.12. The Grothendieck classes of the $Y(T)$ 
correspond to what Brundan and Kleshchev call the {\em quasi-canonical} 
basis of $V_{\Lambda}$.  
\end{remark}

Let $T\in\mathrm{Std}^{(N^{\ell})}$ and 
suppose that 
$$A_T=E_{-i_1}^{(r_1)}\cdots E_{-i_l}^{(r_l)}v_{\Lambda},$$
as in~\eqref{eq:LTdef}. 
Then we 
can define a projective module $P_T\in \mathcal{V}_{\Lambda}^p$ by 
\begin{equation}
\label{eq:defPT}
P_T:=\mathcal{E}_{-i_1}^{(r_1)}\cdots \mathcal{E}_{-i_l}^{(r_l)}V({\Lambda}),
\end{equation}
where $V(\Lambda)=P_{T_{\Lambda}}$ 
is the highest weight object in $\mathcal{V}_{\Lambda}^p$. 
By~\eqref{eq:LTtildepsi-inv} and Theorem~\ref{thm:cyclKLR}, we have 
\begin{equation}
\label{eq:propPT}
P_T=\delta(A_T)\quad\text{and}\quad P_T^{\circledast}\cong P_T.
\end{equation}
By~\eqref{eq:LT2} and Theorem~\ref{thm:cyclKLR}, we have 
\begin{equation}
\label{eq:LT2cat}
[Q_T]=[P_T]+\sum_{S\prec T}\beta_{ST}(v) [P_S]
\end{equation}
in $K_0^v(\mathcal{V}_{\Lambda}^p)$.
%%%%%%%%%%%%%%%%%%%%%%%%%%%%%%%%%%%%%%%%%%%%%%%%%%%%%%%%%%%%%%%%%%%%%%%%%%%%%%%%%%%%%%%%%%%%%%%%%%%%%%%%
%
%
% section
%
%
%%%%%%%%%%%%%%%%%%%%%%%%%%%%%%%%%%%%%%%%%%%%%%%%%%%%%%%%%%%%%%%%%%%%%%%%%%%%%%%%%%%%%%%%%%%%%%%%%%%%%%%%
%%
\section{Categorified skew Howe duality}
\label{sec:cathowe}
\subsection{Equivalences}
\label{sec:Morita}
Recall the algebra $H(\vec{k},N)$ and the categories 
$\mathcal{W}(\vec{k},N)$, $\mathcal{W}^p(\vec{k},N)$ and 
$\mathcal{W}^{\circ}(\vec{k},N)$ which were defined in 
Definitions~\ref{def:web-category},~\ref{def:webalgebra} and~\ref{def:webalgebramod}.
\begin{definition}
\label{def:catwebspace}
Define 
\begin{eqnarray*}
H_{\Lambda}&:=&\bigoplus_{\vec{k}\in\Lambda(m,m)_N} H(\vec{k},N);\\
\mathcal{W}_{\Lambda}&:=&\bigoplus_{\vec{k}\in \Lambda(m,m)_N}\mathcal{W}(\vec{k},N);\\
\mathcal{W}^p_{\Lambda}&:=&\bigoplus_{\vec{k}\in \Lambda(m,m)_N}\mathcal{W}^p(\vec{k},N);\\
\mathcal{W}^{\circ}_{\Lambda}&:=&\bigoplus_{\vec{k}\in \Lambda(m,m)_N}\mathcal{W}^{\circ}(\vec{k},N).
\end{eqnarray*} 
\end{definition}
\noindent We denote the Karoubi envelope of 
$\mathcal{W}^{\circ}_{\Lambda}$ by $\dot{\mathcal{W}}^{\circ}_{\Lambda}$.

In Definition 9.1 in~\cite{my} 
Yonezawa and I defined a strong $\mathfrak{sl}_m$ 
$2$-representation on $\dot{\mathcal{W}}^{\circ}_{\Lambda}$. This was defined 
by sending $\idm_{\lambda'}\mathcal{E}_{\ui}\idm_{\lambda}$ 
to the functor defined by tensoring with the matrix factorization 
$$\oE_{\ui,[\vec{k}]}':=\oE_{\ui,[\vec{k}]}\{d(\vec{k})-d(\vec{k}')\}$$ 
which was associated to the ladder obtained by 
quantum skew Howe duality, for 
$\ui=(\epsilon_1i_1,\ldots,\epsilon_li_l)$, 
$\lambda,\lambda'\in P_{\Lambda}$ and $\phi_{m,m,N}(\lambda)=\vec{k}$ and 
$\phi_{m,m,N}(\lambda')=\vec{k}'$. 

The results in Section 9.2 of~\cite{my} can be summarized as follows:
\begin{theorem}
\label{thm:categorification1}
$\dot{\mathcal{W}}^{\circ}_\Lambda$ is a strong additive $\mathfrak{sl}_m$ 
$2$-representation equivalent to $\mathcal{V}_{\Lambda}^p$, such that 
the following square commutes
\begin{equation}
\label{eq:Moritasquare1}
\begin{CD}
V_ {\Lambda}&@>{\delta}>>&K_0^v(\mathcal{V}_{\Lambda}^p)\\
@V{}VV&&@VV{}V\\
W_{\Lambda}&@>{\delta^{\circ}}>>& K_0^v(\dot{\mathcal{W}}_{\Lambda}^{\circ}) 
\end{CD}.
\end{equation}
The map $\delta$ is the one from Theorem~\ref{thm:cyclKLR}. 
The map $\delta^{\circ}$ is determined by 
$$A^T\mapsto q^{-d(\vec{k})}[\hat{A}^T]=[\mathcal{E}_{-i_1}^{(r_1)}\cdots \mathcal{E}_{-i_l}^{(r_l)}\hat{w}_{\Lambda}],$$
for any $T\in \mathrm{Std}^{(N^{\ell})}$. 
The map $V_{\Lambda}\to W_{\Lambda}$ was defined in Remark~\ref{rem:VversusW}. 
All maps in the square are isometric isomorphisms of $\Uv{m}$-modules.     
\end{theorem}

In this section I am first going to show that 
$\dot{\mathcal{W}}^{\circ}_{\Lambda}$ is 
equivalent to $\mathcal{W}^p_{\Lambda}$. Recall 
the definition of $P^T\in \mathcal{W}_{\Lambda}^p$ in 
Example~\ref{ex:leftandrightprojectives}.

\begin{definition}
\label{def:gammamap}
We define a linear map 
$$\delta_{\vec{k},N}'\colon W(\vec{k},N)\to K_0^v(\mathcal{W}^p(\vec{k},N))$$
by 
$$A^T\mapsto [P^T]
\in K_0^v(\mathcal{W}^p(\vec{k},N))$$
for all $T\in \mathrm{Std}^{(N^{\ell})}_{\vec{k}}$.   
\end{definition}
\noindent In Corollary~\ref{cor:categorification13} I will show that 
$\delta_{\vec{k},N}'$ is an isomorphism. For now, all I can show is 
injectivity.  
\begin{lemma}
\label{lem:injectivity}
The map $\delta_{\vec{k},N}'$ is injective.
\end{lemma}
\begin{proof}
Recall that the Euler form 
$$\langle [P],[Q]\rangle=\dim_v \mathrm{HOM}(P,Q)$$ 
is a non-degenerate $v$-sesquilinear form on 
$K_0^v(H(\vec{k},N))$. Furthermore, the $v$-sesquilinear web form 
gives a non-degenerate $v$-sesquilinear form 
on $W(\vec{k},N)$. 

The map $\delta_{\vec{k},N}'$ is an isometry w.r.t. these two forms 
because we have 
$$\dim_v \mathrm{HOM}(P^S,P^T)=\dim_v 
{}_SH(\vec{k},N)_T=$$
$$\dim_v \mathrm{EXT}(\hat{A}^S,\hat{A}^T) 
=v^{d(\vec{k})}\dim_v H^*(\widehat{(A^S)^*A^T})=$$
$$v^{d(\vec{k})}\mathrm{ev}((A^S)^*A^T)=\langle A^S,A^T\rangle.$$
for any $S,T\in \mathrm{Std}^{(N^{\ell})}_{\vec{k}}$, 
by Theorem~\ref{thm:catwebrels}. 

Since isometries for non-degenerate forms are always injective, this proves 
the lemma.  
\end{proof}

\begin{lemma}
\label{lem:embeddingwebalgintowebcat}
There exists an equivalence of additive categories  
$$\mathcal{W}^p(\vec{k},N) \cong \dot{\mathcal{W}^{\circ}}(\vec{k},N).$$
\end{lemma}
\begin{proof}
First define the image of projective modules associated to 
basis webs. For any $T\in \mathrm{Std}^{(N^{\ell})}_{\vec{k}}$ and any 
$t\in \Z$, we define 
\begin{equation}
\label{eq:embeddingwebalgintowebcat}
\mathcal{W}^p(\vec{k},N) \ni P^T\{t\}\mapsto \hat{A}^T\{t-d(\vec{k})\}\in 
\dot{\mathcal{W}^{\circ}}(\vec{k},N).
\end{equation}
By definition, for any pair $S,T\in \mathrm{Std}^{(N^{\ell})}_{\vec{k}}$ we have 
$$\mathrm{HOM}(P^S,P^T)=\mathrm{EXT}(\hat{A}^S\{-d(\vec{k})\},
\hat{A}^T\{-d(\vec{k})\})=
\mathcal{W}^{\circ}(\hat{A}^S\{-d(\vec{k})\},\hat{A}^T\{-d(\vec{k})\}),$$
where the latter is the hom-space in $\dot{\mathcal{W}^{\circ}}(\vec{k},N)^*$. 
So any intertwiner in $\mathrm{HOM}(P^S,P^T)$ can be mapped to 
the corresponding morphism in 
$\mathcal{W}^{\circ}(\hat{A}^S\{-d(\vec{k})\},\hat{A}^T\{-d(\vec{k})\})$. 

Any indecomposable object in 
$\mathcal{W}^p(\vec{k},N)$ has to be isomorphic to a direct summand of 
$P^T\{t\}$, for some 
$T\in \mathrm{Std}^{(N^{\ell})}_{\vec{k}}$ and $t\in\Z$, 
because by definition we have 
$$H(\vec{k},N)=\bigoplus_{T\in \mathrm{Std}^{(N^{\ell})}_{\vec{k}}}P^T\{d(\vec{k})\}.$$    
Therefore~\eqref{eq:embeddingwebalgintowebcat} determines the desired embedding 
up to natural isomorphism. 

To see this, let $Q$ be an indecomposable object in 
$\mathcal{W}^{p}(\vec{k},N)$. There exists a $T\in 
\mathrm{Std}^{(N^{\ell})}_{\vec{k}}$ and $t\in\Z$ such 
that $Q$ is isomorphic to a direct summand of $P^T\{t\}$ given by 
a primitive idempotent $e_Q \in \mathrm{End}(P^T\{t\})$. We define  
$$Q\mapsto (A^T\{t\},e_Q)\in \dot{\mathcal{W}^{\circ}}(\vec{k},N).$$
The choice of $T$ and $t$ need not be unique, but we simply choose one. A 
different choice will lead to a naturally isomorphic functor.  

This gives a well-defined functor 
\begin{equation}
\label{eq:embeddingWpW0}
\mathcal{W}^p(\vec{k},N) \hookrightarrow 
\dot{\mathcal{W}^{\circ}}(\vec{k},N),
\end{equation}
which is fully faithful.  
\vskip0.2cm
We now have to show that this functor is essentially surjective. It suffices 
to show that any indecomposable object in $\dot{\mathcal{W}^{\circ}}(\vec{k},N)$ 
is isomorphic to the image of an indecomposable object in 
$\mathcal{W}^p(\vec{k},N)$. Since the functor 
in~\eqref{eq:embeddingWpW0} is fully faithful, 
it induces an embedding 
$$
K_0^v(\mathcal{W}^p(\vec{k},N))\hookrightarrow 
K_0^v(\dot{\mathcal{W}}^{\circ}(\vec{k},N))
$$
which sends Grothendieck classes of indecomposables to 
Grothendieck classes of indecomposables. Therefore, all we have to 
do is show that this map between the Grothendieck groups is surjective. 
We do this by proving that both vector spaces have the same dimension. 

In Lemma~\ref{lem:injectivity}, I showed that the linear map 
$$\delta_{\vec{k},N}'\colon W(\vec{k},N)\to K_0^v(\mathcal{W}^p(\vec{k},N))$$
from Definition~\ref{def:gammamap} is injective. 

So we have a sequence of two embeddings
\begin{equation}
\label{eq:twoembeddings}
W(\vec{k},N)\hookrightarrow K_0^v(\mathcal{W}^p(\vec{k},N))\hookrightarrow 
K_0^v(\dot{\mathcal{W}}^{\circ}(\vec{k},N))
\end{equation}
and their composite is an isomorphism by 
Theorem~\ref{thm:categorification1}. This 
shows that both embeddings in~\eqref{eq:twoembeddings} are isomorphisms, 
which finishes the proof of this lemma. 
\end{proof}
In the proof of Lemma~\ref{lem:embeddingwebalgintowebcat} we have 
also obtained the following result.  
\begin{corollary}
\label{cor:categorification13}
The linear map $\delta_{\vec{k},N}'\colon W(\vec{k},N)\to 
K_0^v(\mathcal{W}^p(\vec{k},N))$ is an isomorphism. 
\end{corollary}

Let $\delta'\colon W_{\Lambda}\to K_0^v(\mathcal{W}^p_{\Lambda})$
be the bijective $\Uv{m}$-intertwiner defined by 
$$\delta':=\bigoplus_{\vec{k}\in \Lambda(m,m)_N}\delta_{\vec{k},N}'.$$
We can now prove our main result.  
\begin{theorem}
\label{thm:categorification2}
There exists a strong $\mathfrak{sl}_m$ $2$-representation on the 
Abelian category $\mathcal{W}_{\Lambda}$, which can be restricted to 
$\mathcal{W}_{\Lambda}^p$.

There exists   
a duality on $\mathcal{W}_{\Lambda}$, also denoted $\circledast$, 
which commutes with the action of $\mathcal{E}_{\pm i}\idm_{\lambda}$, for 
any $i=1,\ldots,m-1$ and $\lambda\in\Z^{m-1}$. This duality can be 
restricted to $\mathcal{W}_{\Lambda}^p$. 

Moreover, up to natural isomorphism, there exists a unique equivalence
$$\mathcal{W}_{\Lambda}\to \mathcal{V}_{\Lambda}$$
of Abelian strong $\mathfrak{sl}_m$ $2$-representations intertwining 
the dualities, which can be restricted to an equivalence 
$$\mathcal{V}_{\Lambda}^p\to \mathcal{W}_{\Lambda}^p$$
such that the following square commutes 
\begin{equation}
\label{eq:Moritasquare}
\begin{CD}
W_ {\Lambda}&@>{\delta'}>>&K_0^v(\mathcal{W}_{\Lambda}^p)\\
@V{}VV&&@VV{}V\\
V_{\Lambda}&@>{\delta}>>& K_0^v(\mathcal{V}_{\Lambda}^p) 
\end{CD}.
\end{equation}
All maps in~\eqref{eq:Moritasquare} are $\Uv{m}$-intertwiners, 
isometries and intertwine the relevant involutions.  
\end{theorem}
\begin{proof}
In~\eqref{eq:defPT} I defined a projective module 
$P_T\in \mathcal{V}_{\Lambda}^p$, for any $T\in \mathrm{Std}^{(N^{\ell})}$. 
By Theorem~\ref{thm:categorification1}, we have 
\begin{equation}
\label{eq:firstiso}
\mathrm{HOM}(P_S,P_T)\cong \mathrm{EXT}(\hat{A}^S,\hat{A}^T)
\end{equation}
for any $S,T\in T\in \mathrm{Std}^{(N^{\ell})}_{\vec{k}}$. Let 
$$P_{\Lambda}:=\bigoplus_{\vec{k}\in\Lambda(m,m)_N}
\bigoplus_{T\in\mathrm{Std}^{(N^{\ell})}_{\vec{k}}} P_T\{d(\vec{k})\}.$$
Then~\eqref{eq:firstiso} implies 
$$\mathrm{END}(P_{\Lambda})\cong H_{\Lambda}^{\mathrm{opp}},$$
so $P_{\Lambda}$ is a $R_{\Lambda}-H_{\Lambda}$ bimodule. 

By~\eqref{eq:LT2}, Theorem~\ref{thm:cyclKLR} and~\eqref{eq:LT2cat}, 
we see that $P_{\Lambda}$ is a projective generator of $\mathcal{V}_{\Lambda}$. 

By one of Morita's main results on Morita equivalence (see Theorem 5.55 
in~\cite{rot}, for example), the above observations show that the 
exact functor $\mathcal{W}_{\Lambda}\to \mathcal{V}_{\Lambda}$ defined by 
\begin{equation}
\label{eq:Morita}
M\mapsto P_{\Lambda}\otimes_{H_{\Lambda}} M
\end{equation}
is an equivalence. It also maps projective modules to projective modules, so 
it restricts to an equivalence $\mathcal{W}_{\Lambda}^p\to \mathcal{V}_{\Lambda}^p$. 
Note that under the equivalence~\eqref{eq:Morita}, we have 
$$P^T\mapsto P_T$$
for all $T\in \mathrm{Std}^{(N^{\ell})}$. 
\vskip0.5cm

Up to natural isomorphism, there is now a unique strong 
$\mathfrak{sl}_m$ $2$-representation on $\mathcal{W}_{\Lambda}$ such 
that the equivalence in~\eqref{eq:Morita} becomes an intertwiner and 
the square in~\eqref{eq:Moritasquare} commutes. We can describe 
this action concretely as follows. 

For any signed sequence $\ui=(\epsilon_1 i_1,\ldots,\epsilon_l 
i_l)$ of simple $\mathfrak{sl}_m$ roots and any $\vec{k}\in P_{\Lambda}$, 
we define the $H(\vec{k}',N)-H(\vec{k},N)$ bimodule 
$$ 
H_{\ui}(\vec{k},N):=\bigoplus_{S\in \mathrm{Std}^{(N^{\ell})}_{\vec{k}'}, T\in 
\mathrm{Std}^{(N^{\ell})}_{\vec{k}}}
\mathrm{EXT}(\hat{A}^S,\widehat{E}_{\ui,[\vec{k}]}\btime{R^{\vec{k}}}\hat{A}^T)
\{d(\vec{k})-d(\vec{k}')\},
$$
where $\vec{k}'=\vec{k}+\sum_{j=1}^l \epsilon_{j}\alpha_{i_j}$. The grading 
shift $d(\vec{k})-d(\vec{k}')$ matches precisely 
Brundan and Kleshchev's grading shift in Lemma 4.4 in~\cite{bk}. To see 
this, take for example $\ui=(-i)$. Then 
$\vec{k}'=\vec{k}-\alpha_i$, so  
$$d(\vec{k})-d(\vec{k}')=1-(k_i-k_{i+1}).$$
This is exactly the shift corresponding to the functor 
$K_i^{-1}\langle 1\rangle$ on $\mathcal{V}_{\Lambda}$ 
in Lemma 4.4 in~\cite{bk}. 

From the equivalence in~\eqref{eq:Morita} it follows that 
$H_{\ui}(\vec{k},N)$ is {\em sweet}, meaning that it is projective as a left 
module and as a right module (but not as a bimodule!). Therefore, tensoring 
with $H_{\ui}(\vec{k},N)$ defines an exact endofunctor on $\mathcal{W}_{\Lambda}$ 
which can be restricted to $\mathcal{W}_{\Lambda}^p$. 

Furthermore, we have 
\begin{gather*}
H_{\ui}(\vec{k},N)\cong\\
H_{\epsilon_1 i_1}(\vec{k}+\sum_j^{l-1}\epsilon_j\alpha_j,N)\otimes_{H(\vec{k}+\sum_j^{l-1}\epsilon_j\alpha_j,N)}\cdots\otimes_{H(\vec{k}+\epsilon_l\alpha_l,N)} H_{\epsilon_l i_l}(\vec{k},N).
\end{gather*}

Note that by Theorem 9.2 in~\cite{my} any $2$-morphism in $\u(\mathfrak{sl}_m)$ gives rise 
to a homomorphism of matrix factorizations  
$$\oE_{\ui,[\vec{k}]}\to \oE_{\ui',[\vec{k}]}$$ 
which induces a bimodule map 
$$H_{\ui}(\vec{k},N)\to H_{\ui'}(\vec{k},N).$$
This relation between $2$-morphisms and bimodule maps is compatible with 
compositions, tensor products and units.

It now follows that the square of isomorphisms 
in~\eqref{eq:Moritasquare} commutes. Moreover, all arrows are isometric 
$\Uv{m}$-intertwiners. 
\vskip0.5cm
We can also define a duality on $\mathcal{W}_{\Lambda}$ such that 
the maps in the square~\eqref{eq:Moritasquare} intertwine 
the various bar-involutions, as I now explain. 

By the duality $\circledast$ on $\mathcal{V}_{\Lambda}$ and~\eqref{eq:firstiso} 
we get a graded anti-automorphism 
$$*\colon H_{\Lambda}\to H_{\Lambda}.$$
Just as Brundan and Kleshchev did for 
$\mathcal{V}_{\Lambda}$ (see the text above Theorem~\ref{thm:cyclKLR}), 
we can therefore define a duality 
$\circledast\colon \mathcal{W}_{\Lambda}\to \mathcal{W}_{\Lambda}$.

By definition the equivalence in~\eqref{eq:Morita} intertwines the 
dualities on $\mathcal{W}_{\Lambda}$ and $\mathcal{V}_{\Lambda}$. In particular, 
we see that $(P^T)^{\circledast}\cong P^T$. By 
Proposition~\ref{prop:webpsiinvariant} it follows that 
$\delta'\colon W_{\Lambda}\to K_0^v(\mathcal{W}_{\Lambda}^p)$ 
intertwines $\psi$ and $\circledast$.  
 
\end{proof}

\begin{remark}
Let $S,T\in \mathrm{Std}^{(N^{\ell})}$. If we had a definition of 
$\mathfrak{sl}_N$-foams like we have for $N=2$~\cite{bn}, 
$N=3$~\cite{kv} and in a limited case for $N\geq 4$~\cite{msv2}, 
and if we had an isomorphism between 
the space of $\mathfrak{sl}_N$-foams from $A^S$ to $A^T$ and the elements of 
$\mathrm{EXT}(\hat{A}^S,\hat{A}^T)$ as in~\cite{kr,msv2,mv2}, we could 
describe $*\colon H_{\Lambda}\to H_{\Lambda}$ more directly. 
As in~\cite{mpt}, it would be given by the symmetry on foams given 
by reflection in a plane paralel to the source and target webs 
together with orientation reversal. 
\end{remark}
\subsection{Two consequences}
For any $T\in \mathrm{Std}^{(N^{\ell})}$, let $Q^T\in \mathcal{W}_{\Lambda}^p$ 
correspond to the indecomposable $\circledast$-invariant 
projective module $Q_T\in \mathcal{V}_{\Lambda}^p$ (Theorem~\ref{thm:cyclKLR}) 
under the equivalence in Theorem~\ref{thm:categorification2}. Note that 
$Q^T$ is also indecomposable and $\circledast$-invariant (i.e. up to 
isomorphism, of course), because the equivalence intertwines the duality.  

Under the same equivalence $P^T$ corresponds to $P_T$. 
By~\eqref{eq:LT2cat} and the commutativity of the 
square in~\eqref{eq:Moritasquare}, we get 
$$[Q^T]=[P^T]+\sum_{S\prec T}\beta_{S T}(v)[P^S]$$
and the following corollary. 
\begin{corollary}
\label{cor:dualcanonical}
Under the isomorphism $\delta'\colon W_{\Lambda}\to 
K_0^v(\mathcal{W}_{\Lambda}^p)$ we have 
$$b^T\mapsto [Q^T],$$ 
for any $T\in \mathrm{Std}^{(N^{\ell})}$. 
\end{corollary}
\vskip0.5cm
Analogously to the case for $N=2,3$, Theorem~\ref{thm:categorification2} 
also implies that we can determine the center of $H(\vec{k},N)$. By the 
work of Brundan, Kleshchev and Ostrik~\cite{bru2,bk2,bo} it is known that 
the center of $R_{\Lambda}(k_1-k_2,\ldots,k_{m-1}-k_m)$ is isomorphic to 
$$H^*(X^{\Lambda}_{\vec{k}}).$$
As before, $X^{\Lambda}_{\vec{k}}$ is the Spaltenstein variety of partial flags in 
$\C^{m}$ of type $\vec{k}$ and nilpotent linear operator of Jordan type 
$(\ell^N)$. Since Morita equivalent algebras have isomorphic centers, we get:
\begin{corollary}
\label{cor:center}
For any $\vec{k}\in\Lambda(m,m)_N$, there exists an isomorphism of 
graded complex algebras  
$$Z(H(\vec{k},N))\cong H^*(X^{\Lambda}_{\vec{k}}).$$
\end{corollary}
\noindent I could give an explicit degree preserving isomorphism, as was 
done in~\cite{kh2} for $N=2$ and in~\cite{mpt} for $N=3$, but I do not use 
it here and will therefore omit it.

%%%%%%%%%%%%%%%%%%%%%%%%%%%%%%%%%%%%%%%%%%%%%%%%%%%%%%%%%%%%%%%%%%%%%%%%%%%%%%%%%%%%%%%%%%%%%%%%%%%%%%%%
%
%
% Reference
%
%
%%%%%%%%%%%%%%%%%%%%%%%%%%%%%%%%%%%%%%%%%%%%%%%%%%%%%%%%%%%%%%%%%%%%%%%%%%%%%%%%%%%%%%%%%%%%%%%%%%%%%%%%
%\newpage

\input{biblio}
\vskip0.3cm
\noindent Marco Mackaay: {\sl \small CAMGSD, Instituto Superior T\'{e}cnico, 
Lisboa, Portugal; Universidade do Algarve, Faro, Portugal} 
\newline \noindent {\tt \small email: mmackaay@ualg.pt}
\end{document}

%% file: figure/diag26.tex
\unitlength 0.1in
\begin{picture}(  4.2700,  4.0000)(  1.7300, -6.0000)
\special{pn 8}%
\special{pa 600 200}%
\special{pa 600 232}%
\special{pa 596 268}%
\special{pa 590 304}%
\special{pa 582 342}%
\special{pa 570 380}%
\special{pa 558 418}%
\special{pa 542 454}%
\special{pa 524 488}%
\special{pa 506 518}%
\special{pa 486 546}%
\special{pa 466 568}%
\special{pa 446 586}%
\special{pa 424 596}%
\special{pa 402 600}%
\special{pa 380 598}%
\special{pa 358 586}%
\special{pa 336 570}%
\special{pa 316 548}%
\special{pa 296 520}%
\special{pa 278 490}%
\special{pa 260 456}%
\special{pa 246 420}%
\special{pa 232 384}%
\special{pa 220 346}%
\special{pa 212 308}%
\special{pa 204 270}%
\special{pa 202 236}%
\special{pa 200 204}%
\special{pa 200 200}%
\special{sp}%
\special{sh 1}%
\special{pa 200 200}%
\special{pa 180 268}%
\special{pa 200 254}%
\special{pa 220 268}%
\special{pa 200 200}%
\special{fp}%
\put(2.0000,-1.0000){\makebox(0,0){${}_{a}$}}%
\end{picture}%

%% file: figure/diag27.tex
\unitlength 0.1in
\begin{picture}(  4.2700,  4.0000)(  1.7300, -6.0000)
\special{pn 8}%
\special{pa 600 600}%
\special{pa 600 568}%
\special{pa 596 532}%
\special{pa 590 496}%
\special{pa 582 458}%
\special{pa 570 420}%
\special{pa 558 382}%
\special{pa 542 346}%
\special{pa 524 312}%
\special{pa 506 282}%
\special{pa 486 254}%
\special{pa 466 232}%
\special{pa 446 214}%
\special{pa 424 204}%
\special{pa 402 200}%
\special{pa 380 202}%
\special{pa 358 214}%
\special{pa 336 230}%
\special{pa 316 252}%
\special{pa 296 280}%
\special{pa 278 310}%
\special{pa 260 344}%
\special{pa 246 380}%
\special{pa 232 416}%
\special{pa 220 454}%
\special{pa 212 492}%
\special{pa 204 530}%
\special{pa 202 564}%
\special{pa 200 596}%
\special{pa 200 600}%
\special{sp}%
\special{sh 1}%
\special{pa 200 600}%
\special{pa 180 532}%
\special{pa 200 546}%
\special{pa 220 532}%
\special{pa 200 600}%
\special{fp}%
\put(2.0000,-2.0000){\makebox(0,0){${}_{a}$}}%
\end{picture}%

%% file: figure/diag1.tex
\unitlength 0.1in
\begin{picture}(  8.0000,  7.0000)(  2.0000,-10.0000)
\special{pn 8}%
\special{pa 600 600}%
\special{pa 200 300}%
\special{fp}%
\special{sh 1}%
\special{pa 200 300}%
\special{pa 242 356}%
\special{pa 244 332}%
\special{pa 266 324}%
\special{pa 200 300}%
\special{fp}%
\special{pn 8}%
\special{pa 600 600}%
\special{pa 1000 300}%
\special{fp}%
\special{sh 1}%
\special{pa 1000 300}%
\special{pa 936 324}%
\special{pa 958 332}%
\special{pa 960 356}%
\special{pa 1000 300}%
\special{fp}%
\special{pn 8}%
\special{pa 600 1000}%
\special{pa 600 600}%
\special{fp}%
\special{sh 1}%
\special{pa 600 600}%
\special{pa 580 668}%
\special{pa 600 654}%
\special{pa 620 668}%
\special{pa 600 600}%
\special{fp}%
\put(3.5000,-5.0000){\makebox(0,0){${}_{a}$}}%
\put(8.5000,-5.0000){\makebox(0,0){${}_{b}$}}%
\put(7.50000,-8.0000){\makebox(0,0){${}_{a+b}$}}%
\end{picture}%

%% file: figure/diag2.tex
\unitlength 0.1in
\begin{picture}(  8.0000,  7.0000)(  2.0000,-10.0000)
\special{pn 8}%
\special{pa 1000 1000}%
\special{pa 600 700}%
\special{fp}%
\special{sh 1}%
\special{pa 600 700}%
\special{pa 642 756}%
\special{pa 644 732}%
\special{pa 666 724}%
\special{pa 600 700}%
\special{fp}%
\special{pn 8}%
\special{pa 200 1000}%
\special{pa 600 700}%
\special{fp}%
\special{sh 1}%
\special{pa 600 700}%
\special{pa 536 724}%
\special{pa 558 732}%
\special{pa 560 756}%
\special{pa 600 700}%
\special{fp}%
\special{pn 8}%
\special{pa 600 700}%
\special{pa 600 300}%
\special{fp}%
\special{sh 1}%
\special{pa 600 300}%
\special{pa 580 368}%
\special{pa 600 354}%
\special{pa 620 368}%
\special{pa 600 300}%
\special{fp}%
\put(8.5000,-8.0000){\makebox(0,0){${}_{b}$}}%
\put(3.5000,-8.0000){\makebox(0,0){${}_{a}$}}%
\put(7.50000,-5.0000){\makebox(0,0){${}_{a+b}$}}%
\end{picture}%

%% file: figure/diag4.tex
\unitlength 0.1in
\begin{picture}(  4.0000,  7.0000)(  2.0000,-10.0000)
\put(3.5000,-9.0000){\makebox(0,0){${}_{a}$}}%
\special{pn 8}%
\special{pa 200 1000}%
\special{pa 400 650}%
\special{fp}%
\special{sh 1}%
\special{pa 300 825}%825
\special{pa 250 873}%873
\special{pa 274 871}%871
\special{pa 284 893}%893
\special{pa 300 825}%
\special{fp}%
\special{pn 8}%
\special{pa 600 300}%
\special{pa 400 650}%
\special{fp}%
\special{pn 8}%
\special{pa 330 610}%
\special{pa 400 650}%
\special{fp}%
\special{sh 1}%
\special{pa 500 475}%
\special{pa 550 427}%
\special{pa 526 429}%
\special{pa 516 407}%
\special{pa 500 475}%
\special{fp}%
\put(3.250000,-4.0000){\makebox(0,0){${}_{N-a}$}}%
\end{picture}%

%% file: figure/diag3.tex
\unitlength 0.1in
\begin{picture}(  4.0000,  7.0000)(  2.0000,-10.0000)
\put(3.5000,-9.0000){\makebox(0,0){${}_{a}$}}%
\special{pn 8}%
\special{pa 200 1000}%
\special{pa 400 650}%
\special{fp}%
\special{sh 1}%
\special{pa 500 475}%
\special{pa 450 523}%
\special{pa 474 521}%
\special{pa 484 543}%
\special{pa 500 475}%
\special{fp}%
\special{pn 8}%
\special{pa 600 300}%
\special{pa 400 650}%
\special{fp}%
\special{pn 8}%
\special{pa 330 610}%
\special{pa 400 650}%
\special{fp}%
\special{sh 1}%
\special{pa 300 825}%
\special{pa 350 777}%
\special{pa 326 779}%
\special{pa 316 757}%
\special{pa 300 825}%
\special{fp}%
\put(3.50000,-4.0000){\makebox(0,0){${}_{N-a}$}}%
\special{pn 4}%
\special{sh 1}%
\special{ar 400 650 5 5 0  6.28318530717959E+0000}%
\end{picture}%

%% file: figure/diag5.tex
\unitlength 0.1in
\begin{picture}(  4.0000,  7.0000)(  2.0000,-10.0000)
\put(3.5000,-9.0000){\makebox(0,0){${}_{a}$}}%
\special{pn 8}%
\special{pa 200 1000}%
\special{pa 400 650}%
\special{fp}%
\special{sh 1}%
\special{pa 300 825}%
\special{pa 250 873}%
\special{pa 274 871}%
\special{pa 284 893}%
\special{pa 300 825}%
\special{fp}%
\special{pn 8}%
\special{pa 600 300}%
\special{pa 400 650}%
\special{fp}%
\special{pn 8}%
\special{pa 470 690}%
\special{pa 400 650}%
\special{fp}%
\special{sh 1}%
\special{pa 500 475}%
\special{pa 550 427}%
\special{pa 526 429}%
\special{pa 516 407}%
\special{pa 500 475}%
\special{fp}%
\put(3.250000,-4.0000){\makebox(0,0){${}_{N-a}$}}%
\end{picture}%

%% file: figure/diag12.tex
\unitlength 0.1in
\begin{picture}(  2.5400,  6.0000)(  2.7300, -8.0000)
\special{pn 8}%
\special{pa 400 800}%
\special{pa 400 650}%
\special{fp}%
\special{sh 1}%
\special{pa 400 650}%
\special{pa 380 718}%
\special{pa 400 704}%
\special{pa 420 718}%
\special{pa 400 650}%
\special{fp}%
\special{pn 8}%
\special{pa 400 350}%
\special{pa 400 200}%
\special{fp}%
\special{sh 1}%
\special{pa 400 200}%
\special{pa 380 268}%
\special{pa 400 254}%
\special{pa 420 268}%
\special{pa 400 200}%
\special{fp}%
\special{pn 8}%
\special{ar 400 500 100 150  0.0000000 6.2831853}%
\special{sh 1}%
\special{pa 500 475}%
\special{pa 480 543}%
\special{pa 500 529}%
\special{pa 520 543}%
\special{pa 500 475}%
\special{fp}%
\special{sh 1}%
\special{pa 300 475}%
\special{pa 280 543}%
\special{pa 300 529}%
\special{pa 320 543}%
\special{pa 300 475}%
\special{fp}%
\put(4.0000,-9.0000){\makebox(0,0){${}_{b+a}$}}%
\put(4.0000,-1.0000){\makebox(0,0){${}_{b+a}$}}%
\put(2.0000,-5.0000){\makebox(0,0){${}_{b}$}}%
\put(6.0000,-5.0000){\makebox(0,0){${}_{a}$}}%
\end{picture}%

%% file: figure/diag10.tex
\unitlength 0.1in
\begin{picture}(  1.0200,  6.0000)(  3.7300, -8.0000)
\special{pn 8}%
\special{pa 400 800}%
\special{pa 400 200}%
\special{fp}%
\special{sh 1}%
\special{pa 400 500}%
\special{pa 380 568}%
\special{pa 400 554}%
\special{pa 420 568}%
\special{pa 400 500}%
\special{fp}%
\put(6.0000,-5.0000){\makebox(0,0){${}_{b+a}$}}%
\end{picture}%

%% file: figure/diag13.tex
\unitlength 0.1in
\begin{picture}(  2.5400,  6.0000)(  2.7300, -8.0000)
\special{pn 8}%
\special{pa 400 800}%
\special{pa 400 650}%
\special{fp}%
\special{sh 1}%
\special{pa 400 650}%
\special{pa 380 718}%
\special{pa 400 704}%
\special{pa 420 718}%
\special{pa 400 650}%
\special{fp}%
\special{pn 8}%
\special{pa 400 350}%
\special{pa 400 200}%
\special{fp}%
\special{sh 1}%
\special{pa 400 200}%
\special{pa 380 268}%
\special{pa 400 254}%
\special{pa 420 268}%
\special{pa 400 200}%
\special{fp}%
\special{pn 8}%
\special{ar 400 500 100 150  0.0000000 6.2831853}%
\special{sh 1}%
\special{pa 500 475}%
\special{pa 480 543}%
\special{pa 500 529}%
\special{pa 520 543}%
\special{pa 500 475}%
\special{fp}%
\special{sh 1}%
\special{pa 300 525}%
\special{pa 320 459}%
\special{pa 300 473}%
\special{pa 280 459}%
\special{pa 300 525}%
\special{fp}%
\put(4.0000,-9.0000){\makebox(0,0){${}_{a}$}}%
\put(4.0000,-1.0000){\makebox(0,0){${}_{a}$}}%
\put(2.0000,-5.0000){\makebox(0,0){${}_{b}$}}%
\put(7.0000,-5.0000){\makebox(0,0){${}_{b+a}$}}%
\end{picture}%

%% file: figure/diag10a.tex
\unitlength 0.1in
\begin{picture}(  1.0200,  6.0000)(  3.7300, -8.0000)
\special{pn 8}%
\special{pa 400 800}%
\special{pa 400 200}%
\special{fp}%
\special{sh 1}%
\special{pa 400 500}%
\special{pa 380 568}%
\special{pa 400 554}%
\special{pa 420 568}%
\special{pa 400 500}%
\special{fp}%
\put(5.50000,-5.0000){\makebox(0,0){${}_{a}$}}%
\end{picture}%

%% file: figure/diag14.tex
\unitlength 0.1in
\begin{picture}(  8.0000,  6.0000)(  2.0000, -8.0000)
\special{pn 8}%
\special{pa 600 400}%
\special{pa 600 200}%
\special{fp}%
\special{sh 1}%
\special{pa 600 300}%
\special{pa 580 368}%
\special{pa 600 354}%
\special{pa 620 368}%
\special{pa 600 300}%
\special{fp}%
\special{pn 8}%
\special{pa 200 800}%
\special{pa 600 400}%
\special{fp}%
\special{sh 1}%
\special{pa 500 500}%
\special{pa 440 534}%
\special{pa 462 538}%
\special{pa 468 562}%
\special{pa 500 500}%
\special{fp}%
\special{sh 1}%
\special{pa 300 700}%
\special{pa 240 734}%
\special{pa 262 738}%
\special{pa 268 762}%
\special{pa 300 700}%
\special{fp}%
\special{pn 8}%
\special{pa 1000 800}%
\special{pa 600 400}%
\special{fp}%
\special{sh 1}%
\special{pa 800 600}%
\special{pa 834 662}%
\special{pa 838 638}%
\special{pa 862 634}%
\special{pa 800 600}%
\special{fp}%
\special{pn 8}%
\special{pa 600 800}%
\special{pa 400 600}%
\special{fp}%
\special{sh 1}%
\special{pa 500 700}%
\special{pa 534 762}%
\special{pa 538 738}%
\special{pa 562 734}%
\special{pa 500 700}%
\special{fp}%
\put(9.50000,-3.0000){\makebox(0,0){${}_{a+b+c}$}}%
\put(3.0000,-5.0000){\makebox(0,0){${}_{a+b}$}}%
\put(2.0000,-7.0000){\makebox(0,0){${}_{a}$}}%
\put(6.0000,-7.0000){\makebox(0,0){${}_{b}$}}%
\put(10.0000,-7.0000){\makebox(0,0){${}_{c}$}}%
\end{picture}%

%% file: figure/diag15.tex
\unitlength 0.1in
\begin{picture}(  8.0000,  6.0000)(  2.0000, -8.0000)
\special{pn 8}%
\special{pa 600 400}%
\special{pa 600 200}%
\special{fp}%
\special{sh 1}%
\special{pa 600 300}%
\special{pa 580 368}%
\special{pa 600 354}%
\special{pa 620 368}%
\special{pa 600 300}%
\special{fp}%
\special{pn 8}%
\special{pa 200 800}%
\special{pa 600 400}%
\special{fp}%
\special{sh 1}%
\special{pa 400 600}%
\special{pa 340 634}%
\special{pa 362 638}%
\special{pa 368 662}%
\special{pa 400 600}%
\special{fp}%
\special{pn 8}%
\special{pa 1000 800}%
\special{pa 600 400}%
\special{fp}%
\special{sh 1}%
\special{pa 700 500}%
\special{pa 734 562}%
\special{pa 738 538}%
\special{pa 762 534}%
\special{pa 700 500}%
\special{fp}%
\special{sh 1}%
\special{pa 900 700}%
\special{pa 934 762}%
\special{pa 938 738}%
\special{pa 962 734}%
\special{pa 900 700}%
\special{fp}%
\special{pn 8}%
\special{pa 600 800}%
\special{pa 800 600}%
\special{fp}%
\special{sh 1}%
\special{pa 700 700}%
\special{pa 640 734}%
\special{pa 662 738}%
\special{pa 668 762}%
\special{pa 700 700}%
\special{fp}%
\put(9.50000,-3.0000){\makebox(0,0){${}_{a+b+c}$}}%
\put(9.50000,-5.0000){\makebox(0,0){${}_{b+c}$}}%
\put(2.0000,-7.0000){\makebox(0,0){${}_{a}$}}%
\put(6.0000,-7.0000){\makebox(0,0){${}_{b}$}}%
\put(10.0000,-7.0000){\makebox(0,0){${}_{c}$}}%
\end{picture}%

%% file: figure/diag17.tex
\unitlength 0.1in
\begin{picture}(  6.5400,  6.0000)(  1.7300, -8.0000)
\special{pn 8}%
\special{pa 200 800}%
\special{pa 200 200}%
\special{fp}%
\special{sh 1}%
\special{pa 200 250}%
\special{pa 180 318}%
\special{pa 200 304}%
\special{pa 220 318}%
\special{pa 200 250}%
\special{fp}%
\special{sh 1}%
\special{pa 200 500}%
\special{pa 180 568}%
\special{pa 200 554}%
\special{pa 220 568}%
\special{pa 200 500}%
\special{fp}%
\special{sh 1}%
\special{pa 200 700}%
\special{pa 180 768}%
\special{pa 200 754}%
\special{pa 220 768}%
\special{pa 200 700}%
\special{fp}%
\special{pn 8}%
\special{pa 800 800}%
\special{pa 800 200}%
\special{fp}%
\special{sh 1}%
\special{pa 800 250}%
\special{pa 780 318}%
\special{pa 800 304}%
\special{pa 820 318}%
\special{pa 800 250}%
\special{fp}%
\special{sh 1}%
\special{pa 800 400}%
\special{pa 780 468}%
\special{pa 800 454}%
\special{pa 820 468}%
\special{pa 800 400}%
\special{fp}%
\special{sh 1}%
\special{pa 800 700}%
\special{pa 780 768}%
\special{pa 800 754}%
\special{pa 820 768}%
\special{pa 800 700}%
\special{fp}%
\special{pn 8}%
\special{pa 200 650}%
\special{pa 800 550}%
\special{fp}%
\special{sh 1}%
\special{pa 500 600}%
\special{pa 432 592}%
\special{pa 448 610}%
\special{pa 438 632}%
\special{pa 500 600}%
\special{fp}%
\special{pn 8}%
\special{pa 200 450}%
\special{pa 800 350}%
\special{fp}%
\special{sh 1}%
\special{pa 500 400}%
\special{pa 432 392}%
\special{pa 448 410}%
\special{pa 438 432}%
\special{pa 500 400}%
\special{fp}%
\put(1.0000,-8.0000){\makebox(0,0){${}_{a}$}}%
\put(9.0000,-8.0000){\makebox(0,0){${}_{b}$}}%
\put(0.0000,-5.50000){\makebox(0,0){${}_{a-s}$}}%
\put(10.0000,-4.50000){\makebox(0,0){${}_{b+s}$}}%
\put(5.0000,-7.0000){\makebox(0,0){${}_{s}$}}%
\put(5.0000,-3.0000){\makebox(0,0){${}_{t}$}}%
\put(-1.0000,-2.0000){\makebox(0,0){${}_{a-s-t}$}}%
\put(11.0000,-2.0000){\makebox(0,0){${}_{b+s+t}$}}%
\end{picture}%

%% file: figure/diag16.tex
\unitlength 0.1in
\begin{picture}(  6.5400,  6.0000)(  1.7300, -8.0000)
\special{pn 8}%
\special{pa 200 800}%
\special{pa 200 200}%
\special{fp}%
\special{sh 1}%
\special{pa 200 300}%
\special{pa 180 368}%
\special{pa 200 354}%
\special{pa 220 368}%
\special{pa 200 300}%
\special{fp}%
\special{sh 1}%
\special{pa 200 650}%
\special{pa 180 718}%
\special{pa 200 704}%
\special{pa 220 718}%
\special{pa 200 650}%
\special{fp}%
\special{pn 8}%
\special{pa 800 800}%
\special{pa 800 200}%
\special{fp}%
\special{sh 1}%
\special{pa 800 300}%
\special{pa 780 368}%
\special{pa 800 354}%
\special{pa 820 368}%
\special{pa 800 300}%
\special{fp}%
\special{sh 1}%
\special{pa 800 650}%
\special{pa 780 718}%
\special{pa 800 704}%
\special{pa 820 718}%
\special{pa 800 650}%
\special{fp}%
\special{pn 8}%
\special{pa 200 550}%
\special{pa 800 450}%
\special{fp}%
\special{sh 1}%
\special{pa 500 500}%
\special{pa 432 492}%
\special{pa 448 510}%
\special{pa 438 532}%
\special{pa 500 500}%
\special{fp}%
\put(1.0000,-8.0000){\makebox(0,0){${}_{a}$}}%
\put(9.0000,-8.0000){\makebox(0,0){${}_{b}$}}%
\put(5.0000,-4.0000){\makebox(0,0){${}_{s+t}$}}%
\put(-1.0000,-2.0000){\makebox(0,0){${}_{a-s-t}$}}%
\put(11.0000,-2.0000){\makebox(0,0){${}_{b+s+t}$}}%
\end{picture}%

%% file: figure/diag18.tex
\unitlength 0.1in
\begin{picture}(  6.5400,  6.0000)(  1.7300, -8.0000)
\special{pn 8}%
\special{pa 200 800}%
\special{pa 200 200}%
\special{fp}%
\special{sh 1}%
\special{pa 200 250}%
\special{pa 180 318}%
\special{pa 200 304}%
\special{pa 220 318}%
\special{pa 200 250}%
\special{fp}%
\special{sh 1}%
\special{pa 200 500}%
\special{pa 180 568}%
\special{pa 200 554}%
\special{pa 220 568}%
\special{pa 200 500}%
\special{fp}%
\special{sh 1}%
\special{pa 200 700}%
\special{pa 180 768}%
\special{pa 200 754}%
\special{pa 220 768}%
\special{pa 200 700}%
\special{fp}%
\special{pn 8}%
\special{pa 800 800}%
\special{pa 800 200}%
\special{fp}%
\special{sh 1}%
\special{pa 800 250}%
\special{pa 780 318}%
\special{pa 800 304}%
\special{pa 820 318}%
\special{pa 800 250}%
\special{fp}%
\special{sh 1}%
\special{pa 800 500}%
\special{pa 780 568}%
\special{pa 800 554}%
\special{pa 820 568}%
\special{pa 800 500}%
\special{fp}%
\special{sh 1}%
\special{pa 800 700}%
\special{pa 780 768}%
\special{pa 800 754}%
\special{pa 820 768}%
\special{pa 800 700}%
\special{fp}%
\special{pn 8}%
\special{pa 200 700}%
\special{pa 800 600}%
\special{fp}%
\special{sh 1}%
\special{pa 500 650}%
\special{pa 432 642}%
\special{pa 448 660}%
\special{pa 438 682}%
\special{pa 500 650}%
\special{fp}%
\special{pn 8}%
\special{pa 800 425}%
\special{pa 200 325}%
\special{fp}%
\special{sh 1}%
\special{pa 500 375}%
\special{pa 562 407}%
\special{pa 554 385}%
\special{pa 570 367}%
\special{pa 500 375}%
\special{fp}%
\put(1.0000,-8.0000){\makebox(0,0){${}_{a}$}}%
\put(9.0000,-8.0000){\makebox(0,0){${}_{b}$}}%
\put(0.0000,-5.0000){\makebox(0,0){${}_{a-s}$}}%
\put(10.0000,-5.0000){\makebox(0,0){${}_{b+s}$}}%
\put(5.0000,-7.0000){\makebox(0,0){${}_{s}$}}%
\put(5.0000,-3.0000){\makebox(0,0){${}_{t}$}}%
\put(-1.0000,-2.0000){\makebox(0,0){${}_{a-s+t}$}}%
\put(11.0000,-2.0000){\makebox(0,0){${}_{b+s-t}$}}%
\end{picture}%

%% file: figure/diag19.tex
\unitlength 0.1in
\begin{picture}(  6.5400,  6.0000)(  1.7300, -8.0000)
\special{pn 8}%
\special{pa 200 800}%
\special{pa 200 200}%
\special{fp}%
\special{sh 1}%
\special{pa 200 250}%
\special{pa 180 318}%
\special{pa 200 304}%
\special{pa 220 318}%
\special{pa 200 250}%
\special{fp}%
\special{sh 1}%
\special{pa 200 500}%
\special{pa 180 568}%
\special{pa 200 554}%
\special{pa 220 568}%
\special{pa 200 500}%
\special{fp}%
\special{sh 1}%
\special{pa 200 700}%
\special{pa 180 768}%
\special{pa 200 754}%
\special{pa 220 768}%
\special{pa 200 700}%
\special{fp}%
\special{pn 8}%
\special{pa 800 800}%
\special{pa 800 200}%
\special{fp}%
\special{sh 1}%
\special{pa 800 250}%
\special{pa 780 318}%
\special{pa 800 304}%
\special{pa 820 318}%
\special{pa 800 250}%
\special{fp}%
\special{sh 1}%
\special{pa 800 500}%
\special{pa 780 568}%
\special{pa 800 554}%
\special{pa 820 568}%
\special{pa 800 500}%
\special{fp}%
\special{sh 1}%
\special{pa 800 700}%
\special{pa 780 768}%
\special{pa 800 754}%
\special{pa 820 768}%
\special{pa 800 700}%
\special{fp}%
\special{pn 8}%
\special{pa 800 700}%
\special{pa 200 600}%
\special{fp}%
\special{sh 1}%
\special{pa 500 650}%
\special{pa 562 682}%
\special{pa 554 660}%
\special{pa 570 642}%
\special{pa 500 650}%
\special{fp}%
\special{pn 8}%
\special{pa 200 425}%
\special{pa 800 325}%
\special{fp}%
\special{sh 1}%
\special{pa 500 375}%
\special{pa 432 367}%
\special{pa 448 385}%
\special{pa 438 407}%
\special{pa 500 375}%
\special{fp}%
\put(1.0000,-8.0000){\makebox(0,0){${}_{a}$}}%
\put(9.0000,-8.0000){\makebox(0,0){${}_{b}$}}%
\put(-1.0000,-5.0000){\makebox(0,0){${}_{a+t-r}$}}%
\put(11.0000,-5.0000){\makebox(0,0){${}_{b+r-t}$}}%
\put(5.0000,-7.250000){\makebox(0,0){${}_{t-r}$}}%
\put(5.0000,-3.0000){\makebox(0,0){${}_{s-r}$}}%
\put(-1.0000,-2.0000){\makebox(0,0){${}_{a-s+t}$}}%
\put(11.0000,-2.0000){\makebox(0,0){${}_{b+s-t}$}}%
\end{picture}%

%% file: figure/id-lambda.tex
\unitlength 0.1in
\begin{picture}(  6.000,  10.5000)(  -7.5000, -10.50000)
\put(-9.250000,-7.750000){\makebox(0,0){${}_{k_{m}}$}}%
\put(-6.2750000,-7.750000){\makebox(0,0){${}_{k_{m-1}}$}}%
\put(-1.250000,-7.750000){\makebox(0,0){${}_{k_{2}}$}}%
\put(1.750000,-7.750000){\makebox(0,0){${}_{k_{1}}$}}%
% VECTOR
\special{pn 8}%
\special{pa -850 875}%
\special{pa -850 325}%
\special{fp}%
\special{sh 1}%
\special{pa -850 325}%
\special{pa -870 393}%
\special{pa -850 379}%
\special{pa -830 393}%
\special{pa -850 325}%
\special{fp}%
% VECTOR
\special{pn 8}%
\special{pa -500 875}%
\special{pa -500 325}%
\special{fp}%
\special{sh 1}%
\special{pa -500 325}%
\special{pa -520 393}%
\special{pa -500 379}%
\special{pa -480 393}%
\special{pa -500 325}%
\special{fp}%
% VECTOR
\special{pn 8}%
\special{pa 250 875}%
\special{pa 250 325}%
\special{fp}%
\special{sh 1}%
\special{pa 250 325}%
\special{pa 230 393}%
\special{pa 250 379}%
\special{pa 270 393}%
\special{pa 250 325}%
\special{fp}%
% VECTOR
\special{pn 8}%
\special{pa -50 875}%
\special{pa -50 325}%
\special{fp}%
\special{sh 1}%
\special{pa -50 325}%
\special{pa -70 393}%
\special{pa -50 379}%
\special{pa -30 393}%
\special{pa -50 325}%
\special{fp}%
% DOT
\special{pn 4}%
\special{sh 1}%
\special{ar -150 600 10 10 0  6.28318530717959E+0000}%
\special{pn 4}%
\special{sh 1}%
\special{ar -250 600 10 10 0  6.28318530717959E+0000}%
\special{pn 4}%
\special{sh 1}%
\special{ar -350 600 10 10 0  6.28318530717959E+0000}%
\end{picture}%

%% file: figure/e-mu-i-1.tex
\unitlength 0.1in
\begin{picture}(  26.000,  10.5000)(  -7.5000, -10.50000)
\put(-5.750000,-7.750000){\makebox(0,0){${}_{k_{m}}$}}%
\put(-2.250000,-7.750000){\makebox(0,0){${}_{k_{i+2}}$}}%
\put(2.0000,-4.250000){\makebox(0,0){${}_{k_{i+1}-1}$}}%
\put(10.250000,-4.250000){\makebox(0,0){${}_{k_{i}+1}$}}%
\put(14.750000,-7.750000){\makebox(0,0){${}_{k_{i-1}}$}}%
\put(18.250000,-7.750000){\makebox(0,0){${}_{k_{1}}$}}%
\put(2.750000,-7.750000){\makebox(0,0){${}_{k_{i+1}}$}}%
\put(9.50000,-7.750000){\makebox(0,0){${}_{k_{i}}$}}%
\put(6.0000,-5.000000){\makebox(0,0){${}_{1}$}}%
% VECTOR
\special{pn 8}%
\special{pa -500 875}%
\special{pa -500 325}%
\special{fp}%
\special{sh 1}%
\special{pa -500 325}%
\special{pa -520 393}%
\special{pa -500 379}%
\special{pa -480 393}%
\special{pa -500 325}%
\special{fp}%
% VECTOR
\special{pn 8}%
\special{pa -100 875}%
\special{pa -100 325}%
\special{fp}%
\special{sh 1}%
\special{pa -100 325}%
\special{pa -120 393}%
\special{pa -100 379}%
\special{pa -80 393}%
\special{pa -100 325}%
\special{fp}%
% VECTOR
\special{pn 8}%
\special{pa 400 875}%
\special{pa 400 325}%
\special{fp}%
\special{sh 1}%
\special{pa 400 325}%
\special{pa 380 393}%
\special{pa 400 379}%
\special{pa 420 393}%
\special{pa 400 325}%
\special{fp}%
% VECTOR
\special{pn 8}%
\special{pa 800 875}%
\special{pa 800 325}%
\special{fp}%
\special{sh 1}%
\special{pa 800 325}%
\special{pa 780 393}%
\special{pa 800 379}%
\special{pa 820 393}%
\special{pa 800 325}%
\special{fp}%
% VECTOR
\special{pn 8}%
\special{pa 1300 875}%
\special{pa 1300 325}%
\special{fp}%
\special{sh 1}%
\special{pa 1300 325}%
\special{pa 1280 393}%
\special{pa 1300 379}%
\special{pa 1320 393}%
\special{pa 1300 325}%
\special{fp}%
% VECTOR
\special{pn 8}%
\special{pa 1700 875}%
\special{pa 1700 325}%
\special{fp}%
\special{sh 1}%
\special{pa 1700 325}%
\special{pa 1680 393}%
\special{pa 1700 379}%
\special{pa 1720 393}%
\special{pa 1700 325}%
\special{fp}%
% VECTOR
\special{pn 8}%
\special{pa 400 650}%
\special{pa 800 550}%
\special{fp}%
\special{sh 1}%
\special{pa 800 550}%
\special{pa 730 548}%
\special{pa 748 564}%
\special{pa 740 586}%
\special{pa 800 550}%
\special{fp}%
% DOT
\special{pn 4}%
\special{sh 1}%
\special{ar 1400 600 10 10 0  6.28318530717959E+0000}%
\special{pn 4}%
\special{sh 1}%
\special{ar 1500 600 10 10 0  6.28318530717959E+0000}%
\special{pn 4}%
\special{sh 1}%
\special{ar 1600 600 10 10 0  6.28318530717959E+0000}%
\special{pn 4}%
\special{sh 1}%
\special{ar -200 600 10 10 0  6.28318530717959E+0000}%
\special{pn 4}%
\special{sh 1}%
\special{ar -300 600 10 10 0  6.28318530717959E+0000}%
\special{pn 4}%
\special{sh 1}%
\special{ar -400 600 10 10 0  6.28318530717959E+0000}%
\end{picture}%

%% file: figure/f-mu-i-1.tex
\unitlength 0.1in
\begin{picture}(  26.000,  10.5000)(  -7.5000, -10.50000)
\put(-5.750000,-7.750000){\makebox(0,0){${}_{k_{m}}$}}%
\put(-2.250000,-7.750000){\makebox(0,0){${}_{k_{i+2}}$}}%
\put(2.0000,-4.250000){\makebox(0,0){${}_{k_{i+1}+1}$}}%
\put(10.250000,-4.250000){\makebox(0,0){${}_{k_{i}-1}$}}%
\put(14.750000,-7.750000){\makebox(0,0){${}_{k_{i-1}}$}}%
\put(18.250000,-7.750000){\makebox(0,0){${}_{k_{1}}$}}%
\put(2.750000,-7.750000){\makebox(0,0){${}_{k_{i+1}}$}}%
\put(9.0000,-7.750000){\makebox(0,0){${}_{k_{i}}$}}%
\put(6.0000,-5.000000){\makebox(0,0){${}_{1}$}}%
% VECTOR
\special{pn 8}%
\special{pa -500 875}%
\special{pa -500 325}%
\special{fp}%
\special{sh 1}%
\special{pa -500 325}%
\special{pa -520 393}%
\special{pa -500 379}%
\special{pa -480 393}%
\special{pa -500 325}%
\special{fp}%
% VECTOR
\special{pn 8}%
\special{pa -100 875}%
\special{pa -100 325}%
\special{fp}%
\special{sh 1}%
\special{pa -100 325}%
\special{pa -120 393}%
\special{pa -100 379}%
\special{pa -80 393}%
\special{pa -100 325}%
\special{fp}%
% VECTOR
\special{pn 8}%
\special{pa 400 875}%
\special{pa 400 325}%
\special{fp}%
\special{sh 1}%
\special{pa 400 325}%
\special{pa 380 393}%
\special{pa 400 379}%
\special{pa 420 393}%
\special{pa 400 325}%
\special{fp}%
% VECTOR
\special{pn 8}%
\special{pa 800 875}%
\special{pa 800 325}%
\special{fp}%
\special{sh 1}%
\special{pa 800 325}%
\special{pa 780 393}%
\special{pa 800 379}%
\special{pa 820 393}%
\special{pa 800 325}%
\special{fp}%
% VECTOR
\special{pn 8}%
\special{pa 1300 875}%
\special{pa 1300 325}%
\special{fp}%
\special{sh 1}%
\special{pa 1300 325}%
\special{pa 1280 393}%
\special{pa 1300 379}%
\special{pa 1320 393}%
\special{pa 1300 325}%
\special{fp}%
% VECTOR
\special{pn 8}%
\special{pa 1700 875}%
\special{pa 1700 325}%
\special{fp}%
\special{sh 1}%
\special{pa 1700 325}%
\special{pa 1680 393}%
\special{pa 1700 379}%
\special{pa 1720 393}%
\special{pa 1700 325}%
\special{fp}%
% VECTOR
\special{pn 8}%
\special{pa 800 650}%
\special{pa 400 550}%
\special{fp}%
\special{sh 1}%
\special{pa 400 550}%
\special{pa 460 586}%
\special{pa 452 564}%
\special{pa 470 548}%
\special{pa 400 550}%
\special{fp}%
% DOT
\special{pn 4}%
\special{sh 1}%
\special{ar 1400 600 10 10 0  6.28318530717959E+0000}%
\special{pn 4}%
\special{sh 1}%
\special{ar 1500 600 10 10 0  6.28318530717959E+0000}%
\special{pn 4}%
\special{sh 1}%
\special{ar 1600 600 10 10 0  6.28318530717959E+0000}%
\special{pn 4}%
\special{sh 1}%
\special{ar -200 600 10 10 0  6.28318530717959E+0000}%
\special{pn 4}%
\special{sh 1}%
\special{ar -300 600 10 10 0  6.28318530717959E+0000}%
\special{pn 4}%
\special{sh 1}%
\special{ar -400 600 10 10 0  6.28318530717959E+0000}%
\end{picture}%

%% file: figure/e-mu-i-a.tex
\unitlength 0.1in
\begin{picture}(  26.000,  10.5000)(  -7.5000, -10.50000)
\put(-5.750000,-7.750000){\makebox(0,0){${}_{k_{m}}$}}%
\put(-2.250000,-7.750000){\makebox(0,0){${}_{k_{i+2}}$}}%
\put(2.0000,-4.250000){\makebox(0,0){${}_{k_{i+1}-a}$}}%
\put(10.250000,-4.250000){\makebox(0,0){${}_{k_{i}+a}$}}%
\put(14.750000,-7.750000){\makebox(0,0){${}_{k_{i-1}}$}}%
\put(18.250000,-7.750000){\makebox(0,0){${}_{k_{1}}$}}%
\put(2.750000,-7.750000){\makebox(0,0){${}_{k_{i+1}}$}}%
\put(9.50000,-7.750000){\makebox(0,0){${}_{k_{i}}$}}%
\put(6.0000,-5.000000){\makebox(0,0){${}_{a}$}}%
% VECTOR
\special{pn 8}%
\special{pa -500 875}%
\special{pa -500 325}%
\special{fp}%
\special{sh 1}%
\special{pa -500 325}%
\special{pa -520 393}%
\special{pa -500 379}%
\special{pa -480 393}%
\special{pa -500 325}%
\special{fp}%
% VECTOR
\special{pn 8}%
\special{pa -100 875}%
\special{pa -100 325}%
\special{fp}%
\special{sh 1}%
\special{pa -100 325}%
\special{pa -120 393}%
\special{pa -100 379}%
\special{pa -80 393}%
\special{pa -100 325}%
\special{fp}%
% VECTOR
\special{pn 8}%
\special{pa 400 875}%
\special{pa 400 325}%
\special{fp}%
\special{sh 1}%
\special{pa 400 325}%
\special{pa 380 393}%
\special{pa 400 379}%
\special{pa 420 393}%
\special{pa 400 325}%
\special{fp}%
% VECTOR
\special{pn 8}%
\special{pa 800 875}%
\special{pa 800 325}%
\special{fp}%
\special{sh 1}%
\special{pa 800 325}%
\special{pa 780 393}%
\special{pa 800 379}%
\special{pa 820 393}%
\special{pa 800 325}%
\special{fp}%
% VECTOR
\special{pn 8}%
\special{pa 1300 875}%
\special{pa 1300 325}%
\special{fp}%
\special{sh 1}%
\special{pa 1300 325}%
\special{pa 1280 393}%
\special{pa 1300 379}%
\special{pa 1320 393}%
\special{pa 1300 325}%
\special{fp}%
% VECTOR
\special{pn 8}%
\special{pa 1700 875}%
\special{pa 1700 325}%
\special{fp}%
\special{sh 1}%
\special{pa 1700 325}%
\special{pa 1680 393}%
\special{pa 1700 379}%
\special{pa 1720 393}%
\special{pa 1700 325}%
\special{fp}%
% VECTOR
\special{pn 8}%
\special{pa 400 650}%
\special{pa 800 550}%
\special{fp}%
\special{sh 1}%
\special{pa 800 550}%
\special{pa 730 548}%
\special{pa 748 564}%
\special{pa 740 586}%
\special{pa 800 550}%
\special{fp}%
% DOT
\special{pn 4}%
\special{sh 1}%
\special{ar 1400 600 10 10 0  6.28318530717959E+0000}%
\special{pn 4}%
\special{sh 1}%
\special{ar 1500 600 10 10 0  6.28318530717959E+0000}%
\special{pn 4}%
\special{sh 1}%
\special{ar 1600 600 10 10 0  6.28318530717959E+0000}%
\special{pn 4}%
\special{sh 1}%
\special{ar -200 600 10 10 0  6.28318530717959E+0000}%
\special{pn 4}%
\special{sh 1}%
\special{ar -300 600 10 10 0  6.28318530717959E+0000}%
\special{pn 4}%
\special{sh 1}%
\special{ar -400 600 10 10 0  6.28318530717959E+0000}%
\end{picture}%

%% file: figure/f-mu-i-a.tex
\unitlength 0.1in
\begin{picture}(  26.000,  10.5000)(  -7.5000, -10.50000)
\put(-5.750000,-7.750000){\makebox(0,0){${}_{k_{m}}$}}%
\put(-2.250000,-7.750000){\makebox(0,0){${}_{k_{i+2}}$}}%
\put(2.0000,-4.250000){\makebox(0,0){${}_{k_{i+1}+a}$}}%
\put(10.250000,-4.250000){\makebox(0,0){${}_{k_{i}-a}$}}%
\put(14.750000,-7.750000){\makebox(0,0){${}_{k_{i-1}}$}}%
\put(18.250000,-7.750000){\makebox(0,0){${}_{k_{1}}$}}%
\put(2.750000,-7.750000){\makebox(0,0){${}_{k_{i+1}}$}}%
\put(9.0000,-7.750000){\makebox(0,0){${}_{k_{i}}$}}%
\put(6.0000,-5.000000){\makebox(0,0){${}_{a}$}}%
% VECTOR
\special{pn 8}%
\special{pa -500 875}%
\special{pa -500 325}%
\special{fp}%
\special{sh 1}%
\special{pa -500 325}%
\special{pa -520 393}%
\special{pa -500 379}%
\special{pa -480 393}%
\special{pa -500 325}%
\special{fp}%
% VECTOR
\special{pn 8}%
\special{pa -100 875}%
\special{pa -100 325}%
\special{fp}%
\special{sh 1}%
\special{pa -100 325}%
\special{pa -120 393}%
\special{pa -100 379}%
\special{pa -80 393}%
\special{pa -100 325}%
\special{fp}%
% VECTOR
\special{pn 8}%
\special{pa 400 875}%
\special{pa 400 325}%
\special{fp}%
\special{sh 1}%
\special{pa 400 325}%
\special{pa 380 393}%
\special{pa 400 379}%
\special{pa 420 393}%
\special{pa 400 325}%
\special{fp}%
% VECTOR
\special{pn 8}%
\special{pa 800 875}%
\special{pa 800 325}%
\special{fp}%
\special{sh 1}%
\special{pa 800 325}%
\special{pa 780 393}%
\special{pa 800 379}%
\special{pa 820 393}%
\special{pa 800 325}%
\special{fp}%
% VECTOR
\special{pn 8}%
\special{pa 1300 875}%
\special{pa 1300 325}%
\special{fp}%
\special{sh 1}%
\special{pa 1300 325}%
\special{pa 1280 393}%
\special{pa 1300 379}%
\special{pa 1320 393}%
\special{pa 1300 325}%
\special{fp}%
% VECTOR
\special{pn 8}%
\special{pa 1700 875}%
\special{pa 1700 325}%
\special{fp}%
\special{sh 1}%
\special{pa 1700 325}%
\special{pa 1680 393}%
\special{pa 1700 379}%
\special{pa 1720 393}%
\special{pa 1700 325}%
\special{fp}%
% VECTOR
\special{pn 8}%
\special{pa 800 650}%
\special{pa 400 550}%
\special{fp}%
\special{sh 1}%
\special{pa 400 550}%
\special{pa 460 586}%
\special{pa 452 564}%
\special{pa 470 548}%
\special{pa 400 550}%
\special{fp}%
% DOT
\special{pn 4}%
\special{sh 1}%
\special{ar 1400 600 10 10 0  6.28318530717959E+0000}%
\special{pn 4}%
\special{sh 1}%
\special{ar 1500 600 10 10 0  6.28318530717959E+0000}%
\special{pn 4}%
\special{sh 1}%
\special{ar 1600 600 10 10 0  6.28318530717959E+0000}%
\special{pn 4}%
\special{sh 1}%
\special{ar -200 600 10 10 0  6.28318530717959E+0000}%
\special{pn 4}%
\special{sh 1}%
\special{ar -300 600 10 10 0  6.28318530717959E+0000}%
\special{pn 4}%
\special{sh 1}%
\special{ar -400 600 10 10 0  6.28318530717959E+0000}%
\end{picture}%

%% file: figure/diag31.tex
\unitlength 0.1in
\begin{picture}(  6.0000,  6.0000)(  1.0000, -7.0000)
\special{pn 8}%
\special{pa 400 700}%
\special{pa 400 500}%
\special{dt 0.045}%
\special{sh 1}%
\special{pa 400 500}%
\special{pa 380 568}%
\special{pa 400 554}%
\special{pa 420 568}%
\special{pa 400 500}%
\special{fp}%
\special{pn 8}%
\special{pa 400 500}%
\special{pa 600 200}%
\special{fp}%
\special{sh 1}%
\special{pa 600 200}%
\special{pa 546 244}%
\special{pa 570 244}%
\special{pa 580 268}%
\special{pa 600 200}%
\special{fp}%
\special{pn 8}%
\special{pa 400 500}%
\special{pa 200 200}%
\special{fp}%
\special{sh 1}%
\special{pa 300 350}%
\special{pa 320 418}%
\special{pa 330 394}%
\special{pa 354 394}%
\special{pa 300 350}%
\special{fp}%
\special{pn 4}%
\special{sh 1}%
\special{pa 300 350}%
\special{pa 280 284}%
\special{pa 270 306}%
\special{pa 246 306}%
\special{pa 300 350}%
\special{fp}%
\put(3.0000,-6.50000){\makebox(0,0){${}_{N}$}}%
\put(1.50000,-2.0000){\makebox(0,0){${}_{a}$}}%
\put(6.50000,-2.0000){\makebox(0,0){${}_{a}$}}%
\end{picture}%

%% file: figure/diag32.tex
\unitlength 0.1in
\begin{picture}(  6.0000,  6.0000)(  1.0000, -7.0000)
\special{pn 8}%
\special{pa 400 700}%
\special{pa 400 500}%
\special{dt 0.045}%
\special{sh 1}%
\special{pa 400 500}%
\special{pa 380 568}%
\special{pa 400 554}%
\special{pa 420 568}%
\special{pa 400 500}%
\special{fp}%
\special{pn 8}%
\special{pa 400 500}%
\special{pa 600 200}%
\special{fp}%
\special{sh 1}%
\special{pa 500 350}%
\special{pa 446 394}%
\special{pa 470 394}%
\special{pa 480 418}%
\special{pa 500 350}%
\special{fp}%
\special{pn 8}%
\special{pa 400 500}%
\special{pa 200 200}%
\special{fp}%
\special{sh 1}%
\special{pa 200 200}%
\special{pa 220 268}%
\special{pa 230 244}%
\special{pa 254 244}%
\special{pa 200 200}%
\special{fp}%
\special{sh 1}%
\special{pa 500 350}%
\special{pa 554 306}%
\special{pa 530 306}%
\special{pa 520 284}%
\special{pa 500 350}%
\special{fp}%
\put(3.0000,-6.50000){\makebox(0,0){${}_{N}$}}%
\put(1.50000,-2.0000){\makebox(0,0){${}_{a}$}}%
\put(6.50000,-2.0000){\makebox(0,0){${}_{a}$}}%
\end{picture}%

%% file: figure/diag33.tex
\unitlength 0.1in
\begin{picture}(  8.0000,  8.0000)(  1.0000, -9.0000)
\special{pn 8}%
\special{pa 400 400}%
\special{pa 200 200}%
\special{fp}%
\special{sh 1}%
\special{pa 200 200}%
\special{pa 234 262}%
\special{pa 238 238}%
\special{pa 262 234}%
\special{pa 200 200}%
\special{fp}%
\special{pn 8}%
\special{pa 400 400}%
\special{pa 600 200}%
\special{fp}%
\special{sh 1}%
\special{pa 600 200}%
\special{pa 540 234}%
\special{pa 562 238}%
\special{pa 568 262}%
\special{pa 600 200}%
\special{fp}%
\special{pn 8}%
\special{pa 600 600}%
\special{pa 400 400}%
\special{fp}%
\special{sh 1}%
\special{pa 400 400}%
\special{pa 434 462}%
\special{pa 438 438}%
\special{pa 462 434}%
\special{pa 400 400}%
\special{fp}%
\special{pn 8}%
\special{pa 600 600}%
\special{pa 800 400}%
\special{fp}%
\special{pn 8}%
\special{pa 800 400}%
\special{pa 800 200}%
\special{fp}%
\special{sh 1}%
\special{pa 800 200}%
\special{pa 780 268}%
\special{pa 800 254}%
\special{pa 820 268}%
\special{pa 800 200}%
\special{fp}%
\special{pn 8}%
\special{pa 600 800}%
\special{pa 600 600}%
\special{dt 0.045}%
\special{sh 1}%
\special{pa 600 600}%
\special{pa 580 668}%
\special{pa 600 654}%
\special{pa 620 668}%
\special{pa 600 600}%
\special{fp}%
\put(5.0000,-7.50000){\makebox(0,0){${}_{N}$}}%
\put(2.0000,-1.50000){\makebox(0,0){${}_{a}$}}%
\put(6.0000,-1.50000){\makebox(0,0){${}_{b}$}}%
\put(10.50000,-1.50000){\makebox(0,0){${}_{c=N-a-b}$}}%
\end{picture}%

%% file: figure/diag34.tex
\unitlength 0.1in
\begin{picture}( 10.0000,  8.0000)(  1.0000, -9.0000)
\special{pn 8}%
\special{pa 400 800}%
\special{pa 400 600}%
\special{dt 0.045}%
\special{sh 1}%
\special{pa 400 600}%
\special{pa 380 668}%
\special{pa 400 654}%
\special{pa 420 668}%
\special{pa 400 600}%
\special{fp}%
\special{pn 8}%
\special{pa 400 600}%
\special{pa 200 400}%
\special{fp}%
\special{sh 1}%
\special{pa 200 400}%
\special{pa 234 462}%
\special{pa 238 438}%
\special{pa 262 434}%
\special{pa 200 400}%
\special{fp}%
\special{pn 8}%
\special{pa 200 200}%
\special{pa 200 400}%
\special{fp}%
\special{sh 1}%
\special{pa 200 400}%
\special{pa 220 334}%
\special{pa 200 348}%
\special{pa 180 334}%
\special{pa 200 400}%
\special{fp}%
\special{pn 8}%
\special{pa 400 600}%
\special{pa 600 400}%
\special{fp}%
\special{sh 1}%
\special{pa 600 400}%
\special{pa 540 434}%
\special{pa 562 438}%
\special{pa 568 462}%
\special{pa 600 400}%
\special{fp}%
\special{pn 8}%
\special{pa 800 600}%
\special{pa 600 400}%
\special{fp}%
\special{sh 1}%
\special{pa 600 400}%
\special{pa 634 462}%
\special{pa 638 438}%
\special{pa 662 434}%
\special{pa 600 400}%
\special{fp}%
\special{pn 8}%
\special{pa 600 400}%
\special{pa 600 300}%
\special{fp}%
\special{sh 1}%
\special{pa 600 300}%
\special{pa 580 368}%
\special{pa 600 354}%
\special{pa 620 368}%
\special{pa 600 300}%
\special{fp}%
\special{pn 8}%
\special{pa 600 200}%
\special{pa 600 300}%
\special{fp}%
\special{sh 1}%
\special{pa 600 300}%
\special{pa 620 234}%
\special{pa 600 248}%
\special{pa 580 234}%
\special{pa 600 300}%
\special{fp}%
\special{pn 8}%
\special{pa 800 600}%
\special{pa 1000 400}%
\special{fp}%
\special{sh 1}%
\special{pa 1000 400}%
\special{pa 940 434}%
\special{pa 962 438}%
\special{pa 968 462}%
\special{pa 1000 400}%
\special{fp}%
\special{pn 8}%
\special{pa 1000 200}%
\special{pa 1000 400}%
\special{fp}%
\special{sh 1}%
\special{pa 1000 400}%
\special{pa 1020 334}%
\special{pa 1000 348}%
\special{pa 980 334}%
\special{pa 1000 400}%
\special{fp}%
\special{pn 8}%
\special{pa 800 800}%
\special{pa 800 600}%
\special{dt 0.045}%
\special{sh 1}%
\special{pa 800 600}%
\special{pa 780 668}%
\special{pa 800 654}%
\special{pa 820 668}%
\special{pa 800 600}%
\special{fp}%
\special{pn 4}%
\put(7.0000,-7.50000){\makebox(0,0){${}_{N}$}}%
\put(3.0000,-7.50000){\makebox(0,0){${}_{N}$}}%
\put(2.0000,-1.50000){\makebox(0,0){${}_{a}$}}%
\put(6.0000,-1.50000){\makebox(0,0){${}_{b}$}}%
\put(12.50000,-1.50000){\makebox(0,0){${}_{c=N-a-b}$}}%
\end{picture}%

%% file: figure/e-dot.tex
\unitlength 0.1in
\begin{picture}(  4.0000,  4.00)(  1.5500, -8.500)
\special{pn 8}%
\special{pa 400 800}%
\special{pa 400 400}%
\special{fp}%
\special{sh 1}%
\special{pa 400 400}%
\special{pa 380 468}%
\special{pa 400 454}%
\special{pa 420 468}%
\special{pa 400 400}%
\special{fp}%
\put(4.0000,-8.50000){\makebox(0,0){${}_{i}$}}%
\put(5.0000,-7.0000){\makebox(0,0){${}_{{\lambda}}$}}%
\put(2.0000,-7.0000){\makebox(0,0){${}_{{\lambda+i'}}$}}%
%\put(3.250000,-6.0000){\makebox(0,0){${}_{1}$}}%
\special{pn 4}%
\special{sh 1}%
\special{ar 400 600 25 25 0  6.28318530717959E+0000}%
\end{picture}%

%% file: figure/f-dot.tex
\unitlength 0.1in
\begin{picture}(  4.0000,  4.00)(  1.5500, -8.500)%
\put(4.0000,-8.50000){\makebox(0,0){${}_{i}$}}%
\put(5.0000,-7.0000){\makebox(0,0){${}_{{\lambda}}$}}%
\put(2.0000,-7.0000){\makebox(0,0){${}_{{\lambda-i'}}$}}%
%\put(3.250000,-6.0000){\makebox(0,0){${}_{1}$}}%
\special{pn 8}%
\special{pa 400 390}%
\special{pa 400 800}%
\special{fp}%
\special{sh 1}%
\special{pa 400 800}%
\special{pa 420 734}%
\special{pa 400 748}%
\special{pa 380 734}%
\special{pa 400 800}%
\special{fp}%
\special{pn 4}%
\special{sh 1}%
\special{ar 400 600 25 25 0  6.28318530717959E+0000}%
\end{picture}%

%% file: figure/e-f-cup.tex
\unitlength 0.1in
\begin{picture}(6.00,4.0)(1.500,-6.0)
% SPLINE
\special{pn 8}%
\special{pa 200 200}%
\special{pa 200 233}%
\special{pa 202 267}%
\special{pa 205 301}%
\special{pa 210 336}%
\special{pa 217 370}%
\special{pa 227 403}%
\special{pa 240 433}%
\special{pa 257 460}%
\special{pa 278 483}%
\special{pa 303 502}%
\special{pa 332 515}%
\special{pa 365 523}%
\special{pa 399 526}%
\special{pa 433 523}%
\special{pa 466 516}%
\special{pa 496 502}%
\special{pa 521 484}%
\special{pa 542 461}%
\special{pa 559 434}%
\special{pa 573 404}%
\special{pa 583 372}%
\special{pa 590 338}%
\special{pa 595 303}%
\special{pa 598 268}%
\special{pa 600 234}%
\special{pa 600 202}%
\special{pa 600 200}%
\special{sp}%
% VECTOR
\special{sh 1}%
\special{pa 433 523}%
\special{pa 366 543}%
\special{pa 380 523}%
\special{pa 366 503}%
\special{pa 433 523}%
\special{fp}%
\put(5.50000,-2.0000){\makebox(0,0){${}_{i}$}}%
\put(6.0000,-5.20000){\makebox(0,0){${}_{{\l}}$}}%
\end{picture}%

%% file: figure/f-e-cup.tex
\unitlength 0.1in
\begin{picture}(6.00,4.0)(1.500,-6.0)
% SPLINE
\special{pn 8}%
\special{pa 200 200}%
\special{pa 200 233}%
\special{pa 202 267}%
\special{pa 205 301}%
\special{pa 210 336}%
\special{pa 217 370}%
\special{pa 227 403}%
\special{pa 240 433}%
\special{pa 257 460}%
\special{pa 278 483}%
\special{pa 303 502}%
\special{pa 332 515}%
\special{pa 365 523}%
\special{pa 399 526}%
\special{pa 433 523}%
\special{pa 466 516}%
\special{pa 496 502}%
\special{pa 521 484}%
\special{pa 542 461}%
\special{pa 559 434}%
\special{pa 573 404}%
\special{pa 583 372}%
\special{pa 590 338}%
\special{pa 595 303}%
\special{pa 598 268}%
\special{pa 600 234}%
\special{pa 600 202}%
\special{pa 600 200}%
\special{sp}%
% VECTOR
\special{sh 1}%
\special{pa 365 523}%
\special{pa 431 543}%
\special{pa 418 523}%
\special{pa 431 503}%
\special{pa 365 523}%
\special{fp}%
\put(5.50000,-2.0000){\makebox(0,0){${}_{i}$}}%
\put(6.0000,-5.20000){\makebox(0,0){${}_{{\l}}$}}%
\end{picture}%

%% file: figure/e-f-cap.tex
\unitlength 0.1in
\begin{picture}(6.00,4.0)(1.500,-6.7500)
% SPLINE
\special{pn 8}%
\special{pa 200 600}%
\special{pa 200 567}%
\special{pa 202 533}%
\special{pa 205 499}%
\special{pa 210 464}%
\special{pa 217 430}%
\special{pa 227 397}%
\special{pa 240 367}%
\special{pa 257 340}%
\special{pa 278 317}%
\special{pa 303 298}%
\special{pa 332 285}%
\special{pa 365 277}%
\special{pa 399 274}%
\special{pa 433 277}%
\special{pa 466 284}%
\special{pa 496 298}%
\special{pa 521 316}%
\special{pa 542 339}%
\special{pa 559 366}%
\special{pa 573 396}%
\special{pa 583 428}%
\special{pa 590 462}%
\special{pa 595 497}%
\special{pa 598 532}%
\special{pa 600 566}%
\special{pa 600 598}%
\special{pa 600 600}%
\special{sp}%
% VECTOR
\special{sh 1}%
\special{pa 365 277}%
\special{pa 431 257}%
\special{pa 418 277}%
\special{pa 431 297}%
\special{pa 365 277}%
\special{fp}%
\put(5.50000,-6.0000){\makebox(0,0){${}_{i}$}}%
\put(6.0000,-2.70000){\makebox(0,0){${}_{{\l}}$}}%
\end{picture}%

%% file: figure/f-e-cap.tex
\unitlength 0.1in
\begin{picture}(6.00,4.0)(1.500,-6.7500)
% SPLINE
\special{pn 8}%
\special{pa 200 600}%
\special{pa 200 567}%
\special{pa 202 533}%
\special{pa 205 499}%
\special{pa 210 464}%
\special{pa 217 430}%
\special{pa 227 397}%
\special{pa 240 367}%
\special{pa 257 340}%
\special{pa 278 317}%
\special{pa 303 298}%
\special{pa 332 285}%
\special{pa 365 277}%
\special{pa 399 274}%
\special{pa 433 277}%
\special{pa 466 284}%
\special{pa 496 298}%
\special{pa 521 316}%
\special{pa 542 339}%
\special{pa 559 366}%
\special{pa 573 396}%
\special{pa 583 428}%
\special{pa 590 462}%
\special{pa 595 497}%
\special{pa 598 532}%
\special{pa 600 566}%
\special{pa 600 598}%
\special{pa 600 600}%
\special{sp}%
% VECTOR
\special{sh 1}%
\special{pa 433 277}%
\special{pa 366 257}%
\special{pa 380 277}%
\special{pa 366 297}%
\special{pa 433 277}%
\special{fp}%
\put(5.50000,-6.0000){\makebox(0,0){${}_{i}$}}%
\put(6.0000,-2.70000){\makebox(0,0){${}_{{\l}}$}}%
\end{picture}%

%% file: figure/e-j-i1.tex
\unitlength 0.1in
\begin{picture}(  6.1000,  3.0000)(  3.9500, -8.50000)
% LINE
\special{pn 8}%
\special{pa 400 800}%
\special{pa 402 770}%
\special{pa 412 742}%
\special{pa 428 716}%
\special{pa 448 694}%
\special{pa 474 672}%
\special{pa 504 652}%
\special{pa 536 634}%
\special{pa 568 616}%
\special{pa 602 600}%
\special{pa 638 582}%
\special{pa 670 564}%
\special{pa 702 546}%
\special{pa 730 526}%
\special{pa 756 504}%
\special{pa 776 482}%
\special{pa 790 456}%
\special{pa 798 428}%
\special{pa 800 400}%
\special{sp}%
% LINE
\special{pn 8}%
\special{pa 800 800}%
\special{pa 798 770}%
\special{pa 788 742}%
\special{pa 774 716}%
\special{pa 752 694}%
\special{pa 726 672}%
\special{pa 698 652}%
\special{pa 666 634}%
\special{pa 632 616}%
\special{pa 598 600}%
\special{pa 564 582}%
\special{pa 530 564}%
\special{pa 500 546}%
\special{pa 470 526}%
\special{pa 446 504}%
\special{pa 426 482}%
\special{pa 410 456}%
\special{pa 402 428}%
\special{pa 400 400}%
\special{sp}%
% VECTOR
\special{pn 8}%
\special{pa 400 410}%
\special{pa 400 400}%
\special{fp}%
\special{sh 1}%
\special{pa 400 400}%
\special{pa 380 468}%
\special{pa 400 454}%
\special{pa 420 468}%
\special{pa 400 400}%
\special{fp}%
% VECTOR
\special{pn 8}%
\special{pa 800 410}%
\special{pa 800 400}%
\special{fp}%
\special{sh 1}%
\special{pa 800 400}%
\special{pa 780 468}%
\special{pa 800 454}%
\special{pa 820 468}%
\special{pa 800 400}%
\special{fp}%
\put(4.0000,-9.0000){\makebox(0,0){${}_{i}$}}%
\put(8.0000,-9.0000){\makebox(0,0){${}_{l}$}}%
\put(9.0000,-6.0000){\makebox(0,0){${}_{{\l}}$}}%
\put(2.0000,-6.0000){\makebox(0,0){${}_{{\l+i'+l'}}$}}%
\end{picture}%

%% file: figure/f-j-i1.tex
\unitlength 0.1in
\begin{picture}(  6.1000,  3.0000)(  3.9500, -8.50000)
% LINE
\special{pn 8}%
\special{pa 400 800}%
\special{pa 402 770}%
\special{pa 412 742}%
\special{pa 428 716}%
\special{pa 448 694}%
\special{pa 474 672}%
\special{pa 504 652}%
\special{pa 536 634}%
\special{pa 568 616}%
\special{pa 602 600}%
\special{pa 638 582}%
\special{pa 670 564}%
\special{pa 702 546}%
\special{pa 730 526}%
\special{pa 756 504}%
\special{pa 776 482}%
\special{pa 790 456}%
\special{pa 798 428}%
\special{pa 800 400}%
\special{sp}%
% LINE
\special{pn 8}%
\special{pa 800 800}%
\special{pa 798 770}%
\special{pa 788 742}%
\special{pa 774 716}%
\special{pa 752 694}%
\special{pa 726 672}%
\special{pa 698 652}%
\special{pa 666 634}%
\special{pa 632 616}%
\special{pa 598 600}%
\special{pa 564 582}%
\special{pa 530 564}%
\special{pa 500 546}%
\special{pa 470 526}%
\special{pa 446 504}%
\special{pa 426 482}%
\special{pa 410 456}%
\special{pa 402 428}%
\special{pa 400 400}%
\special{sp}%
% VECTOR
\special{pn 8}%
\special{pa 400 790}%
\special{pa 400 800}%
\special{fp}%
\special{sh 1}%
\special{pa 400 800}%
\special{pa 420 734}%
\special{pa 400 748}%
\special{pa 380 734}%
\special{pa 400 800}%
\special{fp}%
% VECTOR
\special{pn 8}%
\special{pa 800 790}%
\special{pa 800 800}%
\special{fp}%
\special{sh 1}%
\special{pa 800 800}%
\special{pa 820 734}%
\special{pa 800 748}%
\special{pa 780 734}%
\special{pa 800 800}%
\special{fp}%
\put(4.0000,-9.0000){\makebox(0,0){${}_{i}$}}%
\put(8.0000,-9.0000){\makebox(0,0){${}_{l}$}}%
\put(9.0000,-6.0000){\makebox(0,0){${}_{{\l}}$}}%
\put(2.0000,-6.0000){\makebox(0,0){${}_{{\l-i'-l'}}$}}%
\end{picture}%

%% file: figure/cyclo-rel.tex
\unitlength 0.1in
\begin{picture}(  13.0000,  4.00)(  1.5500, -8.500)%
\put(4.0000,-8.50000){\makebox(0,0){${}_{i_1}$}}%
\put(6.0000,-8.50000){\makebox(0,0){${}_{i_2}$}}%
\put(12.0000,-8.50000){\makebox(0,0){${}_{i_t}$}}%
\put(13.5000,-6.0000){\makebox(0,0){${}_{\Lambda_{i_t}}$}}%
%\put(3.250000,-6.0000){\makebox(0,0){${}_{1}$}}%
\special{pn 8}%
\special{pa 400 390}%
\special{pa 400 800}%
\special{fp}%
\special{sh 1}%
\special{pa 400 800}%
\special{pa 420 734}%
\special{pa 400 748}%
\special{pa 380 734}%
\special{pa 400 800}%
\special{fp}%
\special{pa 600 390}%
\special{pa 600 800}%
\special{fp}%
\special{sh 1}%
\special{pa 600 800}%
\special{pa 620 734}%
\special{pa 600 748}%
\special{pa 580 734}%
\special{pa 600 800}%
\special{fp}%
\special{pa 1200 390}%
\special{pa 1200 800}%
\special{fp}%
\special{sh 1}%
\special{pa 1200 800}%
\special{pa 1220 734}%
\special{pa 1200 748}%
\special{pa 1180 734}%
\special{pa 1200 800}%
\special{fp}%
\special{pn 4}%
\special{sh 1}%
\special{ar 800 600 10 10 0  6.28318530717959E+0000}%
\special{pn 4}%
\special{sh 1}%
\special{ar 900 600 10 10 0  6.28318530717959E+0000}%
\special{pn 4}%
\special{sh 1}%
\special{ar 1000 600 10 10 0  6.28318530717959E+0000}%
\special{pn 4}%
\special{sh 1}%
\special{ar 1200 600 25 25 0  6.28318530717959E+0000}%
\end{picture}%